\documentclass[preprint]{imsart}
\usepackage{xr}

\RequirePackage[OT1]{fontenc}
\RequirePackage{amsthm,amsmath,amssymb}
\RequirePackage[numbers]{natbib}
\RequirePackage[colorlinks,citecolor=blue,urlcolor=blue]{hyperref}
\setattribute{journal}{name}{}  

\usepackage[margin=1in]{geometry} 

\usepackage[pdftex]{color,graphicx}
\usepackage{caption}
\usepackage{subcaption}
\usepackage{epstopdf}

\startlocaldefs
\numberwithin{equation}{section}
\theoremstyle{plain}
\newtheorem{theorem}{Theorem}[section]
\newtheorem{proposition}[theorem]{Proposition}
\newtheorem{lemma}[theorem]{Lemma}
\newtheorem{corollary}[theorem]{Corollary}
\newtheorem{remark}{Remark}[section]

\newcommand{\E}{{\mathbb E}}

\newcommand{\R}{{\mathbb R}}
\renewcommand{\P}{{\mathbb P}}

\newcommand{\dos}{{\mathfrak{d}_S}}
\newcommand{\M}{{\mathcal{M}}}
\newcommand{\Rr}{{\mathfrak{R}}}

\newcommand{\Ps}{{\mathcal{P}}}
\newcommand{\Rs}{{\mathcal{R}}}
\newcommand{\G}{{\mathcal{G}}}
\newcommand{\los}{{\mathfrak{L}}}

\newcommand{\ob}{{\mathring{B}}}

\newcommand{\vol}{{\mathrm{Vol}}}

\newcommand{\shel}{{\mathfrak{H}}^2}
\newcommand{\hel}{{\mathfrak{H}}}
\newcommand{\bx}{{\mathbf{X}}}
\newcommand{\by}{{\mathbf{Y}}}
\newcommand\norm[1]{\left\lVert#1\right\rVert}
\newcommand{\RN}[1]{\textup{\uppercase\expandafter{\romannumeral#1}}}

\def\qt#1{\qquad\text{#1}}

\def\argmax{\mathop{\rm argmax}}

\endlocaldefs

\begin{document}

\begin{frontmatter}
\title{On the nonparametric maximum likelihood estimator for
  Gaussian location mixture densities with application to
  Gaussian denoising}
\runtitle{NPMLE for Gaussian Location Mixtures}

\begin{aug}
\author{\fnms{Sujayam}
  \snm{Saha}\ead[label=e1]{sujayam@berkeley.edu}}
\and
\author{\fnms{Adityanand} \snm{Guntuboyina}\thanksref{t1}\ead[label=e2]{aditya@stat.berkeley.edu}}

\thankstext{t1}{Supported by NSF CAREER Grant DMS-16-54589}
\runauthor{Sujayam Saha and Adityanand Guntuboyina}

\affiliation{University of California at Berkeley}

\address{Department of Statistics\\
345 Evans Hall\\
University of California \\
Berkeley, CA - 94720 \\
\printead{e1}}

\address{Department of Statistics\\
423 Evans Hall\\
University of California\\
Berkeley, CA - 94720 \\
\printead{e2}}
\end{aug}

\begin{abstract}
We study the Nonparametric Maximum
Likelihood Estimator (NPMLE) for estimating Gaussian location mixture
densities in $d$-dimensions from independent observations. Unlike
usual likelihood-based methods for fitting mixtures, NPMLEs are  
based on convex optimization. We prove finite sample
results on the Hellinger accuracy of every NPMLE. Our results imply, in
particular, that every NPMLE achieves near parametric risk (up to
logarithmic multiplicative factors) when the true density is a
discrete Gaussian mixture without any prior information on the number
of mixture components. NPMLEs can naturally be used to yield empirical Bayes
estimates of the Oracle Bayes estimator in the Gaussian denoising
problem. We prove bounds for the accuracy of the empirical Bayes
estimate as an approximation to the Oracle Bayes 
estimator. Here our results imply that the empirical Bayes estimator
performs at nearly the optimal level (up to logarithmic factors) for
denoising in clustering situations without any prior 
knowledge of the number of clusters.  
\end{abstract}



\end{frontmatter}

\section{Introduction}\label{sec_intro}
In this paper, we study the performance of the Nonparametric Maximum
Likelihood Estimator (NPMLE) for estimating a Gaussian location
mixture density in multiple dimensions. We also study the performance
of the empirical Bayes estimator based on the NPMLE for
estimating the Oracle Bayes estimator in the problem of Gaussian
denoising. 

By a Gaussian location mixture density in $\R^d, d \geq 1$, we refer 
to a density of the form
\begin{equation}\label{glmd}
f_G(x) := \int \phi_d(x - \theta) dG(\theta)
\end{equation}
for some probability $G$ on $\R^d$ where  $\phi_d(z) := (2
\pi)^{-d/2} \exp \left(-\|z\|^2/2 \right)$ is the standard
$d$-dimensional normal density ($\|z\|$ is the usual 
Euclidean norm of $z$). Note that $f_G$ is the  density of the random
vector $X = \theta + Z$ where $\theta$ and $Z$ are independent
$d$-dimensional random vectors with $\theta$ having distribution $G$
(i.e., $\theta \sim G$)
and $Z$ having the Gaussian distribution with zero mean and identity
covariance matrix (i.e., $Z \sim N(0, I_d)$).  We let $\M$ to be the
class of all Gaussian location mixture densities i.e., densities of
the form $f_G$ as $G$ varies over all probability measures on $\R^d$.   

Given $n$ independent $d$-dimensional data vectors $X_1, \dots, X_n$  
(throughout the paper, we assume that $n \ge 2$) generated from an
unknown  Gaussian location mixture density $f^* \in \M$, we study the
problem of estimating $f^*$ from $X_1, \dots, X_n$. This problem is
fundamental to the area of estimation in mixture models to which a
number of books (see, for example, 
\citet{MR624267, MR838090, lindsay1995mixture, bohning2000computer,
  mclachlan2004finite, schlattmann2009medical}) and papers have 
been devoted. We focus on the situation where $d$ is small
or moderate,  $n$ is large and where no specific prior information is
available about the mixing measure corresponding to $f^*$. Consistent
estimation in the case where $d$ is comparable in size to $n$ needs
simplifying assumptions on $f^*$ (such as that the mixing measure is
discrete with a small number of atoms and that it is concentrated on a
set of sparse vectors in $\R^d$) which we do not make in this
paper. Let us also note here that we focus on the problem of
estimating $f^*$ and not on estimating the mixing measure
corresponding to $f^*$. 

There are two well-known likelihood-based approaches to estimating
Gaussian location mixtures: (a) the first approach 
involves fixing an integer $k$ and performing maximum likelihood
estimation over $\M_k$ which is the collection of all
densities $f_G \in \M$ where $G$ is discrete and has at most $k$
atoms, and (b) the second approach involves performing maximum
likelihood estimation over the entire class $\M$. This results in the
Nonparametric Maximum Likelihood Estimator (NPMLE) for $f^*$ and is
the focus of this paper.   

The first approach (maximum likelihood estimation over $\M_k$ for a
fixed $k$) is quite popular. However, it suffers from the
two well-known issues: choosing $k$ is non-trivial and, moreover,
maximizing likelihood over $\M_k$ results in a non-convex optimization
problem. This non-convex algorithm is usually approximately solved by
the EM algorithm (see, for example, \citet{DempsterLR77EM,
  mclachlan2007algorithm, watanabe2003algorithm}). Recent 
progress on obtaining a theoretical understanding of the behaviour of
the non-convex EM algorithm has been made by
\citet{balakrishnan2017statistical}. 
Analyzing these estimators for data-dependent choices of $k$ is
well-known to be difficult. \citet{maugis2011non} (see also 
\citet{maugis2013adaptive}) proposed a penalization likelihood
criterion to choose $k$ by suitably employing the general theory of
non-asymptotic model selection via penalization 
due to \citet{BirgeMassart94lecam}, \citet{BarronBirgeMassart} and
\citet{Massart03Flour} and, moreoever, \citet{maugis2011non}
established nonasymptotic risk properties of the resulting
estimator. The computational aspects of their estimator are quite
involved however (see \citet{maugis2011data}) as their estimators are
based on solving multiple non-convex optimization problems.   




The present paper studies the second likelihood-based approach
involving nonparametric maximum likelihood estimation of $f^*$. This
method is unaffected by non-convexity and the need for  choosing
$k$. Formally, by an NPMLE, 
we mean any maximizer $\hat{f}_n$ of $\sum_{i=1}^n \log f(X_i)$ as
$f$ varies over $\M$: 
\begin{equation}\label{kw}
\hat{f}_n  \in \argmax_{f \in \M} \frac{1}{n} \sum_{i=1}^n \log
f(X_i). 
\end{equation}
Note that because the maximization is over the entire class $\M$ of all
Gaussian location mixtures (and not on any non-convex subset
such as $\M_k$), the optimization in $\eqref{kw}$ is a convex
problem. Indeed, the objective function in \eqref{kw} is concave in
$f$ and the constraint set $\M$ is a convex class of densities.  

The idea of using NPMLEs for estimating mixture densities has a long
history (see, for example, the classical references
\citet{kiefer1956consistency, lindsay1983geometry1,
  lindsay1983geometry2, lindsay1995mixture,
  bohning2000computer}). The optimization problem \eqref{kw} and its
solutions have been studied by many authors. It is known that
maximizers of $f \mapsto 
\sum_{i=1}^n \log f(X_i)$ exist over $\M$ which implies that NPMLEs
exist. Maximizers are non-unique however so there exist multiple NPMLEs.
Nevertheless, for every NPMLE $\hat{f}_n$, the values $\hat{f}(X_i)$
for $i = 1, \dots, n$ are unique (this is essentially because 
the objective function in the optimization \eqref{kw} only depends on
$f$ through the values $f(X_1), \dots, f(X_n)$). Proofs of these basic
facts can be found, for example, in \citet[Chapter
2]{bohning2000computer}.  

There exist many algorithms in the literature for approximately
solving the optimization \eqref{kw} (note that though \eqref{kw} is a
convex optimization problem, it is infinite-dimensional which is
probably why exact algorithms seem to be unavailable). These
algorithms range from: (a) vertex direction methods and vertex
exchange methods (see the review papers: \citet{bohning1995review},
\citet{lindsay1995review} and the references therein), (b) EM 
algorithms (see \citet{laird1978nonparametric} and
\citet{jiang2009general}), and (c) modern large-scale interior point
methods (see \citet{koenker2014convex} and
\citet{feng2016nonparametric}). Most of these methods focus on the 
case $d = 1$ and involve maximizing the likelihood over 
mixture densities where the mixing measure is supported on a fixed
fine grid in the range of the data. The algorithm of
\citet{koenker2014convex} is highly scalable (relying on 
the commercial convex optimization library \citet{mosek2015mosek}) and
can obtain an approximate NPMLE 
efficiently even for large sample sizes ($n$ of the order
$100,000$). See Section \ref{imple} for more algorithmic and
implementation details as well as some simulation results. 

Let us now describe the main objectives and contributions of the
current paper. Our first goal is to investigate  the 
theoretical properties of NPMLEs. In particular, we study the accuracy
of  $\hat{f}_n$ as an estimator of the density $f^*$ from which the
data $X_1,\dots, X_n$ are generated. We shall use, as our loss 
function, the squared Hellinger distance:
\begin{equation}\label{sqh}
  \hel^2(f, g) := \int \left(\sqrt{f(x)} - \sqrt{g(x)} \right)^2  dx,
\end{equation}
which is one of the most commonly used loss functions for density
estimation problems. We present a detailed analysis of the risk, $\E
\hel^2(\hat{f}_n, f^*)$, of every NPMLE (the expectation here is taken
with respect to $X_1, \dots, X_n$ distributed independently according
to $f^*$). The other common loss function used in density estimation
is the total variation distance. The total variation distance is
bounded from above by a constant multiple of $\hel$ so that upper
bounds for risk under the squared Hellinger distance automatically
imply upper bounds for risk in squared total variation distance.  
 
Our results imply that, for a large class of true densities $f^* \in
\M$, the risk of every NPMLE $\hat{f}_n$ is parametric (i.e.,
$n^{-1}$) up to multiplicative factors that are logarithmic in
$n$. In particular, our results imply that when $f^* \in
\M_k$ for some $1 \leq k \leq n$, then every NPMLE
has risk $k/n$ up to a logarithmic multiplicative factor in $n$. It is 
not hard to see that the minimax risk over $\M_k$ is bounded from
below by $k/n$ which implies therefore that every NPMLE is nearly
minimax over $\M_k$ (ignoring logarithmic factors in $n$) for every $k
\geq 1$. This is interesting because NPMLEs do not use any a priori 
knowledge of $k$. The price in squared Hellinger risk that is paid for
not knowing $k$ in advance is only logarithmic in $n$. Our results are
non-asymptotic and the bounds for risk over $\M_k$ hold even when $k =
k(n)$ grows with $n$. Our results also imply that NPMLEs have
parametric risk (again up to multiplicative logarithmic factors) when
the mixing measure of $f^*$ is supported on a fixed compact subset of 
$\R^d$. Note that we have assumed that the covariance matrix of every
Gaussian component of mixture densities in the class $\M$ is the
identity matrix. Our results can be extended to the case of arbitrary
and unknown covariance matrices provided a lower bound on the
eigenvalues is available (see Theorem \ref{amix}) (on the other hand, when no
\textit{a priori} information on the covariance matrices is available,
it is well-known that likelihood-based approaches are
infeasible). These results are described in Section
\ref{hela}.   

Previous results on the Hellinger accuracy of NPMLEs were due to
\citet{zhang2009generalized} (see also
\citet{ghosal2001entropies} for related results) who dealt 
with the univariate ($d = 1$) case. Here the Hellinger
accuracy was analyzed under conditions on the moments of the mixing measure
corresponding to $f^*$. The accuracy of NPMLEs in the interesting case
when $f^* \in \M_k$ does not appear to have been studied
previously even in $d = 1$. We study the Hellinger risk of NPMLEs for
all $d \geq 1$ and also under a much broader set of assumptions on
$f^*$ compared to existing papers. 

We would like to mention here that numerous papers have appeared in
the theoretical computer science community establishing rigorous
theoretical results for estimating densities in $\M_k$. For example,
the papers \citet{daskalakis2014faster, suresh2014near,
  bhaskara2015sparse,   chan2013learning, chan2014efficient,
  acharya2017sample,   li2015nearly} have results on estimating
densities in $\M_k$ with rigorous bounds on the error in
estimation. The estimation error is mostly measured in terms of the
total variation distance which is smaller (up to constant
multiplicative factors) compared to the Hellinger distance used in the
present paper. Their sample complexity results imply rates of
estimation of $k/n$ up to logarithmic factors in $n$ for densities in $\M_k$ in terms of
the squared total variation distance and hence these results are
comparable to our results for the NPMLE. The estimation procedures
used in these papers range from (a) hypothesis selection  
over a set of candidate estimators via an improved version of the
Scheff{\'e} estimate (\cite{daskalakis2014faster, suresh2014near}; see
\citet[Chapter 6]{devroye2001combinatorial} for background on the 
Scheff{\'e} estimate), (b) reduction to finding sparse solutions to a
non-negative linear systems (\cite{bhaskara2015sparse}), and (c)
fitting piecewise polynomial densities (\cite{chan2013learning,
  chan2014efficient, acharya2017sample, li2015nearly}; these papers
have the sharpest results). These 
methods are very interesting and, remarkably, come with precise time
complexity guarantees. They are not based on likelihood
maximization however and, in our opinion, conceptually more involved
compared to the NPMLE. An additional minor
difference between our work and this literature is that $k$ is taken
to be a constant (and sometimes even known) in these papers while we
allow $k = k(n)$ to grow with $n$ and, moreover, the NPMLE does not
need any prior knowledge of $k$. 

Let us now describe briefly the proof techniques underlying our risk
results for the NPMLEs. Our technical arguments are based on standard
ideas from the literature on  
empirical processes for assessing the performance of maximum
likelihood estimators (see \citet{vaartwellner96book,
  wong1995probability,   zhang2009generalized}). These techniques
involve bounding the covering numbers of the space of Gaussian
location mixture 
densities. For each compact subset $S \subseteq \R^d$, we prove
covering number bounds for $\M$ under the supremum distance
($L_{\infty}$) on $S$. Our bounds can be seen as extensions of the
one-dimensional covering number results of
\citet{zhang2009generalized} (which are themselves enhancements of
corresponding results in \citet{ghosal2001entropies}). The covering
number results of \citet{zhang2009generalized} can be viewed as
special instances of our bounds for the case when $S = [-M, M]$. The 
extension to arbitrary 
compact sets $S$ is crucial for dealing with rates for densities in
$\M_k$. For proving the final Hellinger risk bounds of $\hat{f}_n$
from these $L_{\infty}$ covering numbers, we use appropriate
modifications of tail arguments from \citet{zhang2009generalized}. A
sketch of these ideas is given in Subsection \ref{denso}. 

The second goal of the present paper is to use NPMLEs to yield
empirical Bayes estimates in the Gaussian denoising problem. By
Gaussian denoising, we refer to the problem of estimating vectors
$\theta_1, \dots, \theta_n \in \R^d$ from independent $d$-dimensional
observations $X_1, \dots, X_n$ generated as   
\begin{equation}\label{dem}
  X_i \sim N(\theta_i, I_d) \qt{for $i = 1, \dots, n$}. 
\end{equation}
The naive estimator in this denoising problem simply estimates each
$\theta_i$  by $X_i$. It is well-known that, depending on the structure of the 
unknown $\theta_1, \dots, \theta_n$, it is possible to achieve 
significant improvement over the naive estimator by using information
from $X_j, j \neq i$ in addition to $X_i$ for estimating
$\theta_i$. An ideal 
prototype for such information sharing across observations is given by
the \textit{Oracle Bayes} estimator which will be denoted by
$\hat{\theta}_1^*, \dots, \hat{\theta}_n^*$ and is defined in the
following way: 
\begin{equation*}
\hat{\theta}_i^* := \E \left(\theta | X = X_i \right) \qt{where
	$\theta \sim \bar{G}_n$ and $X | \theta \sim N(\theta, I_d)$}
\end{equation*}
and $\bar{G}_n$ is the empirical measure corresponding to the true
set of parameters $\theta_1, \dots, \theta_n$. In other words,
$\hat{\theta}_i^*$ is the posterior mean of $\theta$ given $X = X_i$
under the model $X | \theta \sim N(\theta, I_d)$ and the prior 
$\theta \sim \bar{G}_n$. This is an Oracle estimator that is
infeasible in practice as it uses information on the
unknown parameters $\theta_1,
\dots, \theta_n$ via their empirical measure $\bar{G}_n$. It
has the important well-known property (see, for example,
\citet{robbins1951asymptotically}) that $\hat{\theta}_i^*$ can be 
written as $T^*(X_i)$ for each $i = 1, \dots, n$ where $T^*: \R^d
\rightarrow \R^d$ minimizes  
\begin{equation}\label{basri}
  \E \left[\frac{1}{n} \sum_{i=1}^n \|T(X_i) - \theta_i \|^2\right]
\end{equation}
over all possible functions $T: \R^d \rightarrow \R^d$. Estimators for
$\theta_1, \dots, \theta_n$ which are of the form $T(X_1), \dots,
T(X_n)$ for a single non-random function $T$ are known as separable
estimators and the best separable estimator is given by
$\hat{\theta}_1^*, \dots, \hat{\theta}_n^*$. We shall show that
$\hat{\theta}_1^*, \dots, \hat{\theta}_n^*$ can be estimated
accurately by a natural estimator constructed using any NPMLE
\eqref{kw} based on $X_1, \dots, X_n$.  

To motivate the estimator, observe first that it is 
well-known (see, for example, \citet{robbins1956, brown1971admissible,
Stein81, efron2011tweedie}) that $\hat{\theta}_i^*$ has the following
alternative expression as a consequence of Tweedie's formula:  
\begin{equation}\label{ob}
\hat{\theta}_i^* = X_i + \frac{\nabla
	f_{\bar{G}_n}(X_i)}{f_{\bar{G}_n}(X_i)}
\end{equation}
where $f_{\bar{G}_n}$ is the Gaussian location mixture density with
mixing measure $\bar{G}_n$ (defined as in \eqref{glmd}). From the
above expression, it is clear that the Oracle Bayes estimator can be
estimated from the data $X_1, \dots, X_n$ provided one can estimate
the Gaussian location mixture density, $f_{\bar{G}_n}$, from the data
$X_1, \dots, X_n$. For this purpose, as insightfully observed in
\citet{jiang2009general}, any NPMLE, $\hat{f}_n$, as in \eqref{kw} can
be used. Indeed, if $\hat{f}_n$ denotes any NPMLE  based on the data
$X_1, \dots, X_n$, then \citet{jiang2009general} argued that
$\hat{f}_n$ is a good estimator for $f_{\bar{G}_n}$ under \eqref{dem}
so that $\hat{\theta}_i^*$ is estimable by   
\begin{equation}\label{es.int}
\hat{\theta}_i := X_i + \frac{\nabla \hat{f}_n(X_i)}{\hat{f}_n(X_i)}.    
\end{equation}
This yields a completely tuning-free solution to the Gaussian
denoising problem (note however that the
noise distribution is assumed to be completely known as $N(0,
I_d)$). This is the General Maximum Likelihood empirical Bayes
estimator of  
\citet{jiang2009general} who proposed it and studied its theoretical
properties in detail for estimating sparse univariate normal
means. To the best of our knowledge, the properties of the estimator
\eqref{es.int} for multidimensional denoising problems have not been
previously explored. More generally, the empirical Bayes approach to
the Gaussian denoising problem goes back to
\citet{robbins1950generalization,   robbins1951asymptotically,
  robbins1964empirical}. The effectiveness of nonparametric empirical
Bayes estimators for estimating sparse normal means has been 
explored by many authors including \citet{johnstone2005needles,
  brown2009nonparametric,   jiang2009general, donoho2013achieving,
  koenker2014convex} but most work seems restricted to the univariate
setting. On the other hand, there exists prior work on parametric
empirical Bayes methods in the multivariate Gaussian denoising problem
(see, for example, \cite{efron1972empirical, efron1976multivariate}) 
but the role of nonparametric empirical Bayes methods in multivariate
Gaussian denoising does not seem to have been explored previously.  

We perform a detailed study of the accuracy of $\hat{\theta}_i$ in
\eqref{es.int} as an estimator of the Oracle Bayes estimator 
$\hat{\theta}_i^*$ for $i = 1, \dots, n$ in terms of the following
squared error  risk measure:   
\begin{equation}\label{sqba.int}
  \Rr_n(\hat{\theta}, \hat{\theta}^*) := \E \left[ \frac{1}{n}
  \sum_{i=1}^n \|\hat{\theta}_i - \hat{\theta}_i^*\|^2 \right]
\end{equation}
where the expectation is taken with respect to $X_1, \dots, X_n$
generated independently according to \eqref{dem}. The risk 
$\Rr_n(\hat{\theta}, \hat{\theta}^*)$ depends on the
configuration of the unknown parameters $\theta_1, \dots,
\theta_n$ and we perform a detailed study of the risk for natural
configurations of the points $\theta_1, \dots, \theta_n \in
\R^d$. Our results imply that, under natural assumptions on $\theta_1,
\dots, \theta_n$, the risk $\Rr_n(\hat{\theta}, \hat{\theta}^*)$ is
bounded by the parametric rate $1/n$ up to logarithmic multiplicative
factors. For example, when the number of distinct vectors among
$\theta_1, \dots, \theta_n$ equals $k = k(n)$ for some $k \leq n$ (an 
assumption which makes sense in clustering situations), we prove that
the risk $\Rr_n(\hat{\theta}, \hat{\theta}^*)$ is bounded from above
by  the parametric rate $k/n$ up to logarithmic multiplicative factors
in $n$. This result is especially remarkable because the estimator
\eqref{es.int} is tuning free and does not have knowledge of $k$. We
also prove that the analogous minimax risk over this class is bounded
from below by $k/n$ implying that the empirical Bayes estimate is
minimax up to logarithmic multiplicative factors. Our result also
implies that when    $\theta_1, \dots, \theta_n$ take values in a
bounded region on $\R^d$, then also the risk $\Rr_n(\hat{\theta},
\hat{\theta}^*)$ is nearly parametric. Summarizing, our results imply
that, under a wide range of assumptions on $\theta_1, \dots,
\theta_n$, the empirical Bayes estimator $\hat{\theta}_i$ performs
comparably to the Oracle Bayes estimator $\hat{\theta}_i^*$ for
denoising. We also prove some results about denoising in the
heteroscedastic setting where the data $X_1, \dots, X_n$ are
independently generated according to $X_i \sim N(\theta_i, \Sigma_i)$
for more general unknown covariance matrices $\Sigma_1, \dots,
\Sigma_n$. These results are in Section \ref{gadeno}. The results and 
the proof techniques are inspired by the arguments of
\citet{jiang2009general} who studied the univariate denoising problem
under sparsity assumptions. We generalize their arguments to
multidimensions; a sketch of our proof techniques is provided in
Subsection \ref{gendo}.

In addition to theoretical results, we also present simulation
evidence for the effectiveness of $\hat{\theta}_i$ in the Gaussian
denoising problem in Section \ref{imple} (where we also present some
implementation and algorithmic details for computing approximate
NPMLEs). Here, we illustrate the performance of \eqref{es.int} for
denoising when the true parameter vectors $\theta_1, \dots, \theta_n$
take values in certain natural regions in $\R^2$. We also numerically
analyze the performance of \eqref{es.int} in clustering situations
when $\theta_1, \dots, \theta_n$ take $k$ distinct values for some
small $k$ (these results are given in Section \ref{desim} of the
technical appendix at the end of the paper). Here we
compare the performance 
of \eqref{es.int} to other natural procedures such as $k$-means with
$k$ selected via the gap 
statistic (see \citet{tibshirani2001estimating}). We argue that
\eqref{es.int} performs efficiently in terms of the risk
$\Rr_n(\hat{\theta}, \hat{\theta}^*)$. In terms of a purely clustering
based comparison index (such as the Adjusted Rand Index), we argue
that the performance of \eqref{es.int} is still reasonable. 

The rest of the paper is organized in the following manner. In Section 
\ref{hela}, we state our results on the Hellinger accuracy of NPMLEs
for estimating Gaussian location mixture densities. Section
\ref{gadeno} has statements of our results on the risk
$\Rr_n(\hat{\theta}, \hat{\theta}^*)$  in the denoising
problem. An overview of the key ideas in the proofs of the main
results is given in Section \ref{proids}. Section \ref{imple} has
algorithmic details and simulation 
evidence for the effectiveness of \eqref{es.int} for
denoising. Complete proofs of all the results in the paper are 
given in the technical appendix at the end of the paper. Specifically, proofs
for results in Section \ref{hela} are given in Section 
\ref{hela.pf}  while proofs for Section \ref{gadeno} are in Section
\ref{gadeno.pf}. Some additional observations on the
heteroscedastic Gaussian denoising problem are in the
Section \ref{hete} of the technical appendix. Metric entropy results for 
multivariate Gaussian location mixture densities play a crucial rule
in the proofs of the main results; these results are proved
in Section \ref{mmps}. Section \ref{zhafo} contains the statement and proof for a crucial 
ingredient for the proof of the main denoising theorem. Finally,
additional technical results needed in the proofs of the main results
are collected in Section \ref{auxre} together with
their proofs while additional simulation results are in Section
\ref{desim}. 

\section{Hellinger Accuracy of NPMLE}\label{hela} \label{HELA}
For data $X_1, \dots, X_n$, let $\hat{f}_n$ be any NPMLE defined as
in \eqref{kw}. In this section, we study the accuracy of
$\hat{f}_n$ in terms of the squared Hellinger distance (defined in
\eqref{sqh}). All the results in this section are fully proved in
Section \ref{hela.pf} while
Subsection \ref{denso} contains a sketch of the key ideas in the
proof of Theorem \ref{dens} (which is the main result of this
section). 

For investigations into the performance of $\hat{f}_n$, it is most
natural to assume that the data $X_1, \dots, X_n$ are independent
observations having common density $f^* \in \M$ in which case we seek
bounds on $\hel^2(\hat{f}_n, f^*)$. However, following
\citet{zhang2009generalized}, we work under the more general
assumption that $X_1, \dots, X_n$ are  independent but not identically
distributed and that each $X_i$ has a density that belongs to the
class $\M$. This additional generality will be used in Section
\ref{gadeno} for proving results on the Empirical Bayes estimator
\eqref{es.int} for the Gaussian denoising problem.   

Specifically, we assume that $X_1, \dots, X_n$ are independent and
that each $X_i$ has density $f_{G_i}$ for some probability measures
$G_1, \dots, G_n$ on $\R^d$. This distributional assumption on the
data $X_1, \dots, X_n$ includes the following two important special
cases: (a) $G_1, \dots, G_n$ are all identically equal to $G$ (say): in
  this case, the observations $X_1, \dots, X_n$ are identically distributed with
  common density $f^* = f_{G} \in \M$, and (b) Each $G_i$ is
  degenerate at some $\theta_i \in \R^d$: here each data point $X_i$
  is normal with $X_i \sim N_d(\theta_i, I_d)$ and this has been
  referred to as the compound decision setting by Robbins. 

We let $\bar{G}_n := (G_1 + \dots + G_n)/n$ to be the average of the
probability measures $G_1, \dots, G_n$. In the case when $G_1, \dots,
G_n$ are all identically equal to $G$, then clearly $\bar{G}_n =
G$. On the other hand, when each $G_i$ is degenerate at some $\theta_i \in 
\R^d$ (i.e., the compound decision setting), then $\bar{G}_n$ equals
the empirical measure corresponding to $\theta_1, \dots, \theta_n$.    

Under the above \textit{independent but not identically distributed}
assumption on $X_1, \dots, X_n$, it has been insightfully pointed out
by \citet{zhang2009generalized} that every NPMLE $\hat{f}_n$ based on
$X_1, \dots, X_n$ (defined as in \eqref{kw}) is really estimating
$f_{\bar{G}_n}$. Note that $f_{\bar{G}_n}$ denotes the average of the
densities of $X_1, \dots, X_n$. 

In this section, we shall prove bounds for the accuracy of any NPMLE
$\hat{f}_n$ as an estimator for $f_{\bar{G}_n}$ under the Hellinger
distance i.e., for $\hel(\hat{f}_n, f_{\bar{G}_n})$. In order to state
our main theorem, we need to introduce the following 
notation. For nonempty sets $S \subseteq \R^d$, we define the function
$\dos: \R^d \rightarrow [0, \infty)$ by 
  \begin{equation} \label{ds}
\dos(x) := \inf_{u \in S} \norm{x - u} \qt{for $x \in \R^d$}
  \end{equation}
where $\norm{\cdot}$ is the usual Euclidean norm on $\R^d$. Also for
$S \subseteq \R^d$, we let  
\begin{equation}\label{s1}
  S^1 := \left\{x : \dos(x) \leq 1 \right\}. 
\end{equation}
Our bound on $\hel(\hat{f}_n, f_{\bar{G}_n})$ will be controlled by
the following quantity. For every probability measure $G$ on $\R^d$,
every non-empty compact set $S \subseteq \R^d$ and every $M \geq
\sqrt{10 \log n}$, let $\epsilon_n(M, S, G)$ be defined via   
\begin{equation} \label{si.ra}
\begin{split}
  \epsilon^2_n(M, S, G) &:= \mathrm{Vol}(S^1) \frac{M^d}{n} \left(\sqrt{\log n}
  \right)^{d+2} \\ &+  \left(\log n \right) \inf_{p \geq
    \frac{d+1}{2 \log n}} \left(\frac{2 \mu_p(\mathfrak{d}_S, G)}{M}
  \right)^{p}
\end{split}
\end{equation}
where $S^1$ is defined in \eqref{s1} and $\mu_p(\dos, G)$ is defined
as the moment   
\begin{equation*}
  \mu_p(\mathfrak{d}_S, G) := \left(\int_{\R^d}
    \left(\mathfrak{d}_S(\theta) \right)^p
    d {G}(\theta) \right)^{1/p} \qt{for $p > 0$}. 
\end{equation*}
Note that the moments $\mu_p(\dos, G)$ quantify how the probability
(under $G$) decays as one moves away from the set $S$. 

The next theorem proves that $\hel^2(\hat{f}_n, f_{\bar{G}_n})$ is bounded
(with high probability and in expectation) by a constant (depending on
$d$) multiple of $\epsilon_n^2(M, S, \bar{G}_n)$ for every estimator
$\hat{f}_n$ having the property that the likelihood of the data at $\hat{f}_n$ is
not too small compared to the likelihood at $f_{\bar{G}_n}$ (made
precise in inequality \eqref{apfn}). Every NPMLE trivially satisfies
this condition (as it maximizes likelihood) but the theorem also
applies to certain approximate likelihood maximizers.   

The bound given the following theorem holds for every compact
set $S \subseteq \R^d$  and $M \geq \sqrt{10 \log n}$. As will be seen
later in this section, under some simplifying assumptions on
$\bar{G}_n$, our bound for $\hel(\hat{f}_n, f_{\bar{G}_n})$ can be
optimized over $S$  and $M$ to produce an explicit bound.   

\begin{theorem}\label{dens}
	Let $X_1, \dots, X_n$ be independent random vectors with $X_i
	\sim f_{G_i}$ and let $\bar{G}_n := (G_1 + \dots +
        G_n)/n$. Fix $M \geq \sqrt{10 \log n}$ and a non-empty compact
        set $S \subseteq \R^d$ and let $\epsilon_n(M, S, \bar{G}_n)$
        be  defined via \eqref{si.ra}. Then there exists a positive 
        constant $C_d$ (depending only on $d$) such that for every
        estimator $\hat{f}_n$ based on the data $X_1, \dots, X_n$
        satisfying 
        \begin{equation}\label{apfn}
          \prod_{i=1}^n \frac{\hat{f}_n(X_i)}{f_{\bar{G}_n}(X_i)} \geq
          \exp \left[\frac{C_d (\beta - \alpha)}{\min(1-\alpha, \beta)} n
            \epsilon_n^2(M, S, \bar{G}_n) \right] 
        \end{equation}
for some $0 < \beta \leq \alpha < 1$, we have  
\begin{equation}\label{dens.eq}
    \P \left\{\hel(\hat{f}_n, f_{\bar{G}_n}) \geq  \frac{t
        \epsilon_n(M, S, \bar{G}_n) \sqrt{C_d}}{\sqrt{\min(1 - \alpha,
          \beta)}} \right\} \leq 2 n^{-t^2} 
\end{equation}
for every $t \geq 1$ and, moreover,
\begin{equation}
  \label{dens.1}
  \E \shel(\hat{f}_n, f_{\bar{G}_n}) \leq \frac{4 C_d}{\min(1 -
    \alpha, \beta)} \epsilon_n^2(M, S, \bar{G}_n). 
\end{equation}
\end{theorem}

Theorem \ref{dens} asserts that the risk $\E\hel^2(\hat{f}_n,
\bar{f}_{\bar{G}_n})$ is bounded from above by a constant (depending
on $d$, $\alpha$ and $\beta$) multiple of $\epsilon^2_n(M, S,
\bar{G}_n)$ for every 
$M \geq \sqrt{10 \log n}$ and compact subset $S \subseteq \R^d$. This
is true for every estimator $\hat{f}_n$ satisfying \eqref{apfn}. Every
NPMLE satistfies \eqref{apfn} with $\alpha = \beta = 0.5$ (note that
the right hand side of \eqref{apfn} is always less than or equal to
one because $\beta \leq \alpha$).  

Theorem \ref{dens} is novel to the best of our knowledge. When $d = 1$
and $S$ is taken to be $[-R, R]$ for some $R \geq 0$, then the conclusion
given by Theorem \ref{dens} appears implicitly in \citet[Proof of
  Theorem 1]{zhang2009generalized}. The presence of an arbitrary
  compact set $S$ allows the derivation of interesting adaptation
  results for discrete mixing distributions (as will be clear from the
  special cases of Theorem \ref{dens} that are given below). Such
  results cannot be derived if the arbitrary $S$ is replaced by only a
  box or a ball such as $[-R, R]$ as in the 
  univariate result of \citet{zhang2009generalized}. Indeed, suppose
  that $\bar{G}_n$ is a discrete measure gives equal probability to the two
  points $R$ and $-R$ for a large value of $R$. Then the bound of
  \citet{zhang2009generalized} gives a multiplicative factor involving
  $R$ in the risk bounds which make them quite suboptimal when $R$ is
  large. On the other hand, Theorem \ref{dens} applied with $S = \{-R,
  R\}$ gives a near-parametric risk bound (see Theorem \ref{rb.clus}
  below). One can further think of the support of $\bar{G}_n$ being a
  collection of discrete points, curves and regions (all the while
  being bounded) for general $d \geq 2$, where a direct extension of
  Zhang's result would produce an upper bound directly proportional to
  the volume of the  bounding box of the shapes mentioned above; while
  our result will depend on the total volume of the \textit{fattenings} of each of the
shapes described above. In cases where the total fattened volume is a
constant while the separation between the different shapes increases
as a function $n$, our result will yield a tighter upper bound (as a
negative power of $n$) than Zhang's result and its naive
multi-dimensional extension.

Our proof of Theorem \ref{dens} (given in Section \ref{hela.pf}) is
greatly inspired by 
\citet[Proof of Theorem 1]{zhang2009generalized}. An overview of this
proof is provided in Subsection \ref{denso} where we explain the main
ideas as well as points of departure between our proof and the
arguments in \citet[Proof of Theorem 1]{zhang2009generalized}. 


To get the best rate for $\hel(\hat{f}_n, f_{\bar{G}_n})$ from Theorem
\ref{dens}, we need to choose $M$ and $S$ so that $\epsilon_n(M, 
S, \bar{G}_n)$ is small. These choices obviously depend on $\bar{G}_n$
and in the next result, we describe how to choose $M$ and $S$ based on 
reasonable assumptions on $\bar{G}_n$. This leads to explicit
rates for $\hel(\hat{f}_n, f_{\bar{G}_n})$. Note that, more generally,
Theorem \ref{dens} implies that $\hat{f}_n$ is consistent (in the
Hellinger distance) for $f_{\bar{G}_n}$ provided $\bar{G}_n$ is such
that 
\begin{equation*}
  \inf_{S  \text{ compact}, M \geq \sqrt{10 \log n}} \epsilon_n(M, S,
  \bar{G}_n) \rightarrow 0 \qt{as $n \rightarrow \infty$}. 
\end{equation*}

For simplicity, we shall
assume, for the next result, that $\hat{f}_n$  
is an NPMLE so that \eqref{apfn} is satisfied with $\alpha =
\beta = 0.5$. We shall also only state the results on the risk
$\E \hel^2(\hat{f}_n, f_{\bar{G}_n})$. 

\begin{corollary}\label{dza}
  Let $X_1, \dots, X_n$ be independent random vectors with $X_i \sim
  f_{G_i}$ and let $\bar{G}_n := (G_1 + \dots + G_n)/n$. Let
  $\hat{f}_n$ be an NPMLE based on $X_1, \dots, X_n$ defined as in
  \eqref{kw}. Below $C_d$ denotes a positive constant depending on $d$
  alone. 
  \begin{enumerate}
  \item Suppose that $\bar{G}_n$ is supported on a
  compact subset $S$ of $\R^d$. Then 
\begin{equation}\label{dza1.eq}
  \E \shel(\hat{f}_n, f_{\bar{G}_n}) \leq C_d \frac{\vol(S^1)}{n}
  (\log n)^{d+1}. 
\end{equation}
\item Suppose there exist a compact subset $S \subseteq \R^d$ and real
  numbers $0 < \alpha \leq 2$ and $K \geq 1$ such that 
\begin{equation}\label{mmc}
  \mu_p(\dos, \bar{G}_n) \leq K p^{1/\alpha}  \qt{for all $p \geq 1$}. 
\end{equation}
Then 
\begin{equation}\label{dza2.eq}
  \E \shel(\hat{f}_n, f_{\bar{G}_n}) \leq C_d \frac{\vol(S^1) (K
    e^{1/\alpha})^d}{n} \left(\sqrt{\log n} \right)^{(2d/\alpha) + d +
    2}. 
\end{equation}
\item Suppose there exists a compact set $S \subseteq \R^d$ and real
  numbers $\mu > 0$ and $p > 0$ such that $\mu_p(\dos, \bar{G}_n) \leq
  \mu$. Then there exists a positive constant $C_{d, \mu, p}$
  (depending only on $d, \mu$ and $p$) such that  
  \begin{equation} \label{dza3.eq}
    \E \shel(\hat{f}_n, f_{\bar{G}_n}) \leq C_{d, \mu, p}
    \left(\frac{\vol(S^1)}{n} \right)^{\frac{p}{p+d}} \left(\sqrt{\log n} 
    \right)^{\frac{2 d + 2 p + dp}{p+d}} . 
  \end{equation}
  \end{enumerate}
\end{corollary}


Corollary \ref{dza} is a generalization of \citet[Theorem 
1]{zhang2009generalized} as the latter result can be seen as a special
case of Corollary \ref{dza} for $d = 1$ and $S = [-R, R]$ for some $R
\geq 0$. The fact that $S$ can be arbitrary in Corollary \ref{dza} allows
us to deduce the following important adaptation results of NPMLEs for
estimating Gaussian mixtures whose mixing measures are discrete. These
results are, to the best of our knowledge, novel.  

\begin{theorem}[Near parametric risk for discrete Gaussian
  mixtures] \label{rb.clus}
  Let $X_1, \dots, X_n$ be independent random vectors with $X_i \sim
  f_{G_i}$ and let $\bar{G}_n := (G_1 + \dots + G_n)/n$. Let
  $\hat{f}_n$ be an NPMLE based on $X_1, \dots, X_n$ defined as in
  \eqref{kw}. Then there exists a positive constant $C_d$ depending
  only on $d$ such that whenever $\bar{G}_n$ is a discrete probability
  measure that is supported on a set of cardinality $k$, we have 
\begin{equation}\label{rb.clus.eq}
  \E \shel(\hat{f}_n, f_{\bar{G}_n}) \leq C_d
  \left(\frac{k}{n} \right) (\log n)^{d+1}. 
\end{equation}
\end{theorem}

Note that \eqref{rb.clus.eq} directly follows from
\eqref{dza1.eq}. Indeed, when $\bar{G}_n$ is supported on a
finite set $S$ of cardinality $k$, we can apply inequality
\eqref{dza1.eq} to this $S$. It is easy to see then that $\vol(S^1)
\leq C_d k$ which proves \eqref{rb.clus.eq}.  

The significance of Theorem \ref{rb.clus} is the following. 
The right hand side of \eqref{rb.clus.eq} is the parametric
risk $k/n$ up to  an additional multiplicative factor that is
logarithmic in $n$. This inequality shows important adaptation
properties of NPMLEs. When the true unknown Gaussian mixture 
$f_{\bar{G}_n}$ is a discrete mixture having 
$k$ Gaussian components, then every NPMLE nearly (up to
logarithmic factors) achieves the parametric squared Hellinger risk
$k/n$. For a fixed $k$, it is well-known that fitting a $k$-component
Gaussian mixture via maximum likelihood is a non-convex problem that
is usually solved by the EM algorithm. On the other hand, NPMLE is 
given by a convex optimization algorithm, does not require any prior
specification of $k$ and still achieves the $k/n$ rate (up to
logarithmic factors) when the truth is a $k$-component Gaussian
mixture. We would also like to stress here that in Theorem
\ref{rb.clus} (and all other results in the paper), $k$ is allowed to
grow with $n$ (we can write $k(n)$ instead of $k$ but we are sticking
to $k$ for simplicity of notation).   

Note that Theorem \ref{rb.clus} applies to the case of independent
but not identically distributed $X_1, \dots, X_n$ which is more
general compared to the i.i.d assumption. This implies, in particular,
that \eqref{rb.clus.eq} also applies to the case when $X_1, \dots,
X_n$ are i.i.d having density $f^* \in \M$. In this case, we have 
\begin{equation}\label{kri}
 \sup_{f^* \in \M_k} \E \shel(\hat{f}_n, f^*) \leq C_d
  \left(\frac{k}{n} \right) (\log n)^{d+1}. 
\end{equation}
The interesting aspect of this inequality is that it holds for every
$k \geq 1$ and that the estimator $\hat{f}_n$ does not know or use any
information about $k$.  

It is straightforward to prove a minimax lower bound over $\M_k$ that 
complements Theorem \ref{rb.clus}. The following result proves that the
minimax risk over $\M_k$ is bounded from below by a constant multiple
of $k/n$. This implies that the NPMLE is minimax
optimal over $\M_k$ ignoring logarithmic factors of $n$. Moreover,
this optimality is adaptive since MLE does not require knowledge of
$k$. This minimax lower bound is stated for the i.i.d case which
implies that it holds for the more general independent but not
identically distributed case as well.   

\begin{lemma} \label{trilo1}
       For $k \geq 1$, let 
       \begin{equation*}
         \Rs(\M_k) := \inf_{\tilde{f}} \sup_{f \in \M_k} \E_f
         \hel^2(\tilde{f}, f)
       \end{equation*}
  where $\E_f$ denotes expectation when the data $X_1, \dots, X_n$
  are independent observations drawn from the density $f$. Then there
  exists a universal positive constant $C$ such that  
        \begin{equation}\label{trilo1.eq}
	\Rs(\M_k) \geq C \frac{k}{n} \qt{for every $1 \leq k \leq
          n$}. 
	\end{equation}
\end{lemma}

Inequality \eqref{kri} and Lemma \ref{trilo1} together imply that
every NPMLE $\hat{f}_n$ is minimax optimal up to logarithmic factors
in $n$ over the class $\M_k$ for every $k \geq 1$. This optimality is
adaptive since the NPMLE requires no information on $k$. The
logarithmic terms in \eqref{kri} are likely suboptimal but we are
unable to determine the exact power of $\log n$ in \eqref{kri}. 


So far we have studied estimation of Gaussian location mixture
densities where the covariance matrix of each Gaussian component is
fixed to be the identity matrix. We next show that the same estimator
(NPMLE defined as in \eqref{kw}) can be modified to estimate arbitrary 
Gaussian mixtures (where the covariance matrices can be different from
identity) provided a  lower bound on the eigenvalues of the covariance
matrices is available. Suppose that $h^*$ is the Gaussian mixture density  
\begin{equation}\label{hst}
  h^*(x) := \sum_{j=1}^k w_j \phi_d(x; \mu_j, \Sigma_j) \qt{for $x \in
    \R^d$}
\end{equation}
where $k \geq 1$, $\mu_1, \dots, \mu_k \in \R^d$ and $\Sigma_1, \dots,
\Sigma_k$ are $d \times d$ positive definite matrices. Here
$\phi_d(\cdot; \mu, \Sigma)$ denotes the $d$-variate normal density
with mean $\mu$ and covariance matrix $\Sigma$. Suppose
$\sigma_{\min}^2$ and $\sigma_{\max}^2$  are two positive numbers
that are, respectively, smaller and larger than all the eigenvalues of
$\Sigma_1, \dots, \Sigma_k$ i.e., 
  \begin{equation}\label{spe}
    \sigma^2_{\min} \leq \min_{1 \leq j \leq k} \lambda_{\min}(\Sigma_j) \leq
    \max_{1 \leq j \leq k} \lambda_{\max}(\Sigma_j) \leq \sigma_{\max}^2
  \end{equation}
Consider the problem estimating $h^*$ from i.i.d observations $Y_1,
\dots, Y_n$. It turns out that for every NPMLE $\hat{f}_n$ computed as
in \eqref{kw} based on
the data $X_1 := Y_1/\sigma_{\min}, \dots, X_n := Y_n/\sigma_{\min}$
can be coverted to a very good estimator for $h^*$ via  
\begin{equation}\label{hnh}
  \hat{h}_n(x) := \sigma_{\min}^{-d} \hat{f}_n(\sigma_{\min}^{-1} x)
  \qt{for $x \in \R^d$}. 
\end{equation}
Our next result shows that the squared Hellinger risk of $\hat{h}_n$
is bounded from above by $(k/n)$ up to a logarithmic factor in $n$
provided that $\sigma_{\max}/\sigma_{\min}$ is bounded by a constant.
This result implies that applying the NPMLE to $Y_i/\sigma_{\min}$
leads to a very accurate estimator even for heteroscedastic normal
observations. 

\begin{theorem}\label{amix}
  Let $Y_1, \dots, Y_n$ be independent and identically distributed
  observations having density $h^*$ defined in \eqref{hst}. Consider
  the estimator $\hat{h}_n$ for $h^*$ defined in \eqref{hnh}. Then
  \begin{equation} \label{amix.eq} 
    \E \shel(\hat{h}_n, h^*) \leq C_d \left(\frac{k}{n} \right)
    \left(\max(1, \tau) \right)^d (\log n)^{d+1}
  \end{equation}
where $\tau := \sqrt{\frac{\sigma_{\max}^2}{\sigma_{\min}^2} - 1}$.  
\end{theorem}

Theorem \ref{amix} shows that $\hat{h}_n$ achieves
near parametric risk $k/n$ (up to logarithmic factors in $n$) provided
$\tau$ is bounded from above by a constant.  Note that this estimator
$\hat{h}_n$ uses knowledge of $\sigma_{\min}^2$ but does not use
knowledge of any other feature of $h^*$ including the number of
components $k$. In particular, this is an estimation procedure which
(without knowing the value of $k$) achieves nearly the $k/n$ rate for
$k$-component well-conditioned Gaussian mixtures provided a lower
bound $\sigma_{\min}^2$ on eigenvalues is known \textit{a priori}.  

It is natural to compare Theorem \ref{amix} to the main results in
\citet{maugis2011non} where an adaptive procedure is developed for
estimating $k$-component Gaussian mixtures at the rate $k/n$ (up to a
logarithmic factor) without prior knowledge of $k$. The estimator of
\citet{maugis2011non} is very different from ours. They first fit
$m$-component Gaussian mixtures for different values of $m$ and then
select one of these estimators by optimizing a penalized
model-selection criterion. Thus, their procedure is based on solving
multiple non-convex optimization problems. Also, \citet{maugis2011non} impose 
upper and lower bounds on the means and the eigenvalues of the
covariance matrices of the components of the mixture densities. On the
contrary, our method is based on convex optimization and we only need
a lower bound on the eigenvalues of the covariance matrices (no bounds
on the means are necessary). On the flip side, the result of
\citet{maugis2011non}  has much better logarithmic factors compared to
Theorem \ref{amix} and it is also stated in the form of an Oracle
inequality.   

\section{Application to Gaussian
  Denoising}\label{gadeno}  \label{GADENO}
In this section, we explore the role of the NPMLE for estimating the
Oracle Bayes estimator in the Gaussian denoising problem. All the
results in this section are proved in the Section \ref{gadeno.pf} of
the technical appendix at the end of the paper. 

The goal is to estimate unknown vectors $\theta_1,\dots, \theta_n \in
\R^d$ in the compound decision setting where we observe 
independent random vectors $X_1, \dots, X_n$ such that $X_i \sim
N(\theta_i, I_d)$  for $i = 1, \dots, n$. The Oracle estimator is
$\hat{\theta}_i^*, i = 1, \dots, n$ which is given by \eqref{ob}
where $\bar{G}_n$ is the empirical measure corresponding to $\theta_1,
\dots, \theta_n$. 

It is natural to estimate the Oracle Bayes estimator by the Empirical
Bayes estimator $\hat{\theta}_i$ which is defined as
in \eqref{es.int} for $i = 1, \dots, n$. Here $\hat{f}_n$ is any NPMLE
based on $X_1, \dots, X_n$ (defined as in \eqref{kw}). We will gauge
the performance of $\hat{\theta}_i, i = 1, \dots, n$ as an estimator for
$\hat{\theta}_i^*, i = 1, \dots, n$ in terms of the squared error risk measure
$\Rr_n(\hat{\theta}, \hat{\theta}^*)$ defined in \eqref{sqba.int}.     

The main theorem of this section is given below. This is stated in a
form that is similar to the statement of Theorem \ref{dens}. It proves
that, for every compact set $S \subseteq \R^d$ and $M \geq \sqrt{10
  \log n}$, the risk $\Rr_n(\hat{\theta}, \hat{\theta}^*)$ is bounded
from above by $\epsilon^2_n(M, S, \bar{G}_n)$ (defined via
\eqref{si.ra}) up to the additional logarithmic multiplicative factor
$(\log n)^{\max(d, 3)}$. This additional logarithmic factor is a
consequence of our proof technique. 

\begin{theorem} \label{theorem:denoising_theorem}
    Let $X_1, \dots, X_n$ with independent random vectors with $X_i
    \sim N(\theta_i, I_d)$ for $i = 1, \dots, n$. Let $\bar{G}_n$
    denote the empirical measure corresponding to $\theta_1, \dots,
    \theta_n$. Let $\hat{f}_n$ denote an NPMLE based on $X_1, \dots,
    X_n$ defined as in \eqref{kw}. Let $\hat{\theta}_1, \dots,
    \hat{\theta}_n$ be as defined in \eqref{es.int} and let
    $\hat{\theta}_1^*, \dots, \hat{\theta}^*_n$ be as in
    \eqref{ob}. Also, let $\Rr_n(\hat \theta, \hat \theta^*)$ 
    be as in \eqref{sqba.int}. There exists a positive constant $C_d$
    (depending only on $d$) such that for every non-empty compact set
    $S \subseteq \R^d$ and $M \geq \sqrt{10 \log n}$, we have 
	\begin{equation*}
	\Rr_n(\hat \theta, \hat \theta^*) \leq C_d \epsilon_n^2(M,S,
        \bar{G}_n) 
        \left(\sqrt{\log n} \right)^{\max(d-2, 6)}. 
	\end{equation*}
\end{theorem}

\begin{remark}
For the case of $d=1$, \citet[Theorem 5]{jiang2009general} established
a related result on the risk of $\hat \theta_i$ in comparison to $\hat
\theta_i^*$. The risk used therein is 
\begin{equation}\label{jidis}
\left[ \E \left( \frac{1}{n} \sum_{i=1}^n |\hat{\theta}_i -
    \theta_i|^2 \right) \right]^{1/2} - \left[  \E \left( \frac{1}{n}
    \sum_{i=1}^n |\hat{\theta}^*_i - \theta_i|^2 \right)
\right]^{1/2}. 
\end{equation}
\citet{jiang2009general} investigated the above risk in the case where
$d = 1$ and $S = [-R, R]$ for some $R \geq 0$. The statement of
Theorem \ref{theorem:denoising_theorem} and its proof as well as the
following corollary are inspired by \citet[Proof of Theorem
5]{jiang2009general}.   
\end{remark}

Under specific reasonable assumptions on $\bar{G}_n$, it is possible
to choose $M$ and $S$ explicitly which leads to the following result
that is analogous to Corollary \ref{dza}. 

\begin{corollary}\label{dna}
Consider the same setting and notation as in Theorem
\ref{theorem:denoising_theorem}. Below $C_d$ denotes a positive
constant depending on $d$ alone.   
  \begin{enumerate}
  \item For every compact set $S \subseteq \R^d$ containing all the
    points $\theta_1, \dots, \theta_n$, we have 
\begin{equation}\label{dna1.eq}
\Rr_n(\hat \theta, \hat \theta^*) \leq C_d \frac{\vol(S^1)}{n}
\left(\sqrt{\log n}
\right)^{\max(3d, 2d+8)}. 
\end{equation}
\item For every compact subset $S \subseteq \R^d$ and real
  numbers $0 < \alpha \leq 2$ and $K \geq 1$ satisfying \eqref{mmc},
  we have   
\begin{equation}\label{dna2.eq}
\Rr_n(\hat \theta, \hat \theta^*) \leq C_d \frac{\vol(S^1) (K
    e^{1/\alpha})^d}{n} \left(\sqrt{\log n} \right)^{\max
    \left(\frac{2d}{\alpha} + 2d, \frac{2d}{\alpha} + d + 8 \right)}. 
\end{equation}
\item Suppose $S \subseteq \R^d$ is compact and real numbers $\mu > 0$
  and $p > 0$ are such that $\mu_p(\dos, \bar{G}_n) \leq 
  \mu$. Then there exists a positive constant $C_{d, \mu, p}$
  (depending only on $d, 
  \mu$ and $p$) such that 
  \begin{equation} \label{dna3.eq}
\Rr_n(\hat \theta, \hat \theta^*) \leq C_{d, \mu, p}
    \left(\frac{\vol(S^1)}{n} \right)^{\frac{p}{p+d}} \left(\sqrt{\log n} 
    \right)^{\frac{2 d + 2p + dp}{p+d} + \max(d-2, 6) }. 
  \end{equation}
  \end{enumerate}
\end{corollary}

Corollary \ref{dna} has interesting consequences. Inequality
\eqref{dna1.eq} states that when $\bar{G}_n$ is supported on a fixed
compact set $S$, then the risk $\Rr_n(\hat{\theta}, \hat{\theta}^*)$
is parametric upto logarithmic multiplicative factors in $n$. This is
especially interesting because $\hat{\theta}_1, \dots, \hat{\theta}_n$
do not use any knowledge of $S$. 

Corollary \ref{dna} also leads to the following result which gives an
upper bound for $\Rr_n(\hat{\theta}, \hat{\theta}^*)$ when $\theta_1,
\dots, \theta_n$ are clustered into $k$ groups. 
\begin{proposition}\label{pkde}
  Consider the same setting and notation as in Theorem
  \ref{theorem:denoising_theorem}. Suppose that $\theta_1, \dots,
  \theta_n$ satisfy
  \begin{equation}\label{kga}
    \max_{1 \leq i \leq n} \min_{1 \leq j \leq k} \norm{\theta_i -
      a_j} \leq R
  \end{equation}
  for some $a_1, \dots, a_k \in \R^d$ and $R \geq 0$. Then  
  \begin{equation}
    \label{pkde.eq}
    \Rr_n(\hat \theta, \hat \theta^*) \leq C_d \left(1 + R \right)^d
    \left(\frac{k}{n} \right) \left(\sqrt{\log n} \right)^{\max(3d, 2d+8)}.   
  \end{equation}
\end{proposition}
The assumption \eqref{kga} means that $\theta_1, \dots, \theta_n$ can
be grouped into $k$ balls each of radius $R$ centered at the
points $a_1, \dots, a_k$. When $R$ is not large, this implies
$\theta_1, \dots, \theta_n$ 
can be clustered into $k$ groups. In particular, when $R = 0$, the
assumption \eqref{kga} implies that $\theta_1, \dots, \theta_n$ take
only $k$ distinct values. In words, Proposition \ref{pkde} states that
when $\theta_1, \dots, \theta_n$ are clustered into $k$ groups, then  
$\hat{\theta}_1, \dots, \hat{\theta}_n$ estimate $\hat{\theta}^*_1,
\dots, \hat{\theta}^*_n$ in squared error loss with accuracy $k/n$ up 
to logarithmic multiplicative factors in $n$. The notable aspect
about this result is that the estimator does not use any knowledge of
$k$ and is tuning-free. It is well-known in the clustering literature
that choosing the optimal number of clusters is challenging
(see, for example, \citet{tibshirani2001estimating}). It is therefore
helpful that $\hat{\theta}_1, \dots, \hat{\theta}_n$
achieves nearly the $k/n$ rate in \eqref{kga} without explicitly
getting into the pesky problem of estimating $k$. Moreover,
$\hat{\theta}_1, \dots, \hat{\theta}_n$ is given by convex
optimization (on the other hand, one usually needs to deal with
non-convex optimization problems for solving clustering-type problems
even if the number of clusters $k$ is known). 

There exist techniques for estimating the number of clusters and
subsequently employing algorithms for minimizing the $k$-means
objective (notably, the ``gap statistic'' of
\citet{tibshirani2001estimating}). However, we are not aware of any
result analogous to Proposition \ref{pkde} for such techniques. There
also exist other techniques for clustering based on convex
optimization such as the method of convex clustering (see, for
example, \citet{lindsten2011just, hocking2011clusterpath,
  chen2015convex}) which is based on a fused lasso-type penalized
optimization. This method requires specification of tuning
parameters. While interesting theoretical development exists for
convex clustering (see, for example, \citet{radchenko2014consistent,
  zhu2014convex, tan2015statistical, wu2016new, wang2016sparse}), to
the best of our knowledge, a result similar to Proposition \ref{pkde}
is unavailable.   

It is straightforward to see that it is impossible to devise
estimators that achieve a rate that is faster than $k/n$ for the risk
measure $\Rr_n$. We provide a proof of this via a minimax lower bound
in the following lemma. The logarithmic factors can probably be
improved in Proposition \ref{pkde} but we are unable to do so at the
present moment. For the lower bound, let $\Theta_{n, d, k}$ denote the
class of all $n$-tuples $(\theta_1, \dots, \theta_n)$ with each
$\theta_i \in \R^d$ and such that the number of distinct vectors among
$\theta_1, \dots, \theta_n$ is equal to $k$. Equivalently, $\Theta_{n,
  d, k}$ consists of all $n$-tuples $(\theta_1, \dots, \theta_n)$
whose empirical measure is supported on a set of cardinality $k$. The
minimax risk for estimating $\hat{\theta}_1^*, \dots,
\hat{\theta}_n^*$ with $(\theta_1, \dots, \theta_n) \in \Theta_{n, d,
  k}$ in squared error loss  from the observations $X_1, \dots, X_n$
can be defined as    
\begin{equation*}
  \Rs^*(\Theta_{n, d, k}) := \inf_{\tilde{\theta}_1, \dots,
    \tilde{\theta}_n} \sup_{(\theta_1, \dots, \theta_n) \in \Theta_{n,
    d, k}} \E \left[\frac{1}{n} \sum_{i=1}^n \norm{\tilde{\theta}_i -
    \hat{\theta}_i^*}^2 \right] 
\end{equation*}
The following result proves that $\Rs^*(\Theta_{n,d,k})$ is at least
$Ck/n$ for a universal positive constant $C$. 
\begin{lemma}\label{denlo}
  Let $\Theta_{n, d, k}$ and $\Rs^*(\Theta_{n, d, k})$ be defined as
  above. There exists a universal positive constant $C$ such that 
  \begin{equation}\label{denlo.eq}
    \Rs^*(\Theta_{n, d, k}) \geq C \frac{k}{n} \qt{for every $1 \leq k
      \leq n$}. 
  \end{equation}
\end{lemma}
Lemma \ref{denlo}, together with Proposition \ref{pkde}, implies that
$\hat{\theta}_1, \dots, \hat{\theta}_n$ is nearly minimax optimal (up
to logarithmic multiplicative factors) for estimating
$\hat{\theta}_1^*, \dots, \hat{\theta}_n^*$ over the class $\Theta_{n,
  d, k}$. Moreover, this optimality is adaptive over $k$ because the
estimator does not use any knowledge of $k$.   

Before closing this section, let us remark that Theorem
\ref{theorem:denoising_theorem} can be generalized to work with
certain kinds of heteroscedasticity in the Gaussian
observations. Concretely, consider the problem of heteroscedastic
Gaussian denoising where the goal is to estimate $\theta_1, \dots,
\theta_n$ from independent observations $X_1, \dots, X_n$ generated
according to  
\begin{equation}\label{hetdi}
  X_i \sim N(\theta_i, \Sigma_i)
\end{equation}
for some \textit{unknown} covariance matrices $\Sigma_1, \dots,
\Sigma_n$. We work with the assumption that $\Sigma_i - I_d$ is positive
semi-definite (or equivalently, $\lambda_{\min}(\Sigma_i) \geq 1$) for
each $i = 1, \dots, n$. If $\Sigma_i -
\sigma^2_{\min} I_d$  is positive semi-definite for some other known
positive constant $\sigma^2_{\min}$, then one can reduce this to the
previous case by simply scaling the observations $X_1, \dots, X_n$ by
$\sigma^2_{\min}$. 

Note that we are considering the setting where $\Sigma_1, \dots,
\Sigma_n$ are unknown (satisfying $\Sigma_i - I_d$ is positive
semi-definite). This is different from the setting where $\Sigma_1,
\dots, \Sigma_n$ are exactly known and there has been previous work in
Empirical Bayes estimation under this latter assumption (see, for
example, \citet{xie2012sure} and \citet{weinstein2018group}). 

Under the assumption that $\Sigma_i - I_d$ is positive semi-definite,
it is clear that \eqref{hetdi} is equivalent to the statement that $X_i \sim
f_{G^0_i}$ where $G^0_i$ is the $N(\theta_i, \Sigma_i - I_d)$
distribution (here we take $N(\theta_i, \Sigma_i - I_d)$ to be the
Dirac probability measure centered at $\theta_i$ if $\Sigma_i =
I_d$). Therefore, as we have seen in Section \ref{hela}, the estimator
$\hat{f}_n$ based on $X_1, \dots, X_n$ (defined as in
\eqref{kw}) will be an accurate estimator of $f_{\bar{G}^0_n}$ where  
\begin{equation}\label{gengeb}
  \bar{G}^0_n := \frac{1}{n} \sum_{i=1}^n N(\theta_i, \Sigma_i - I_d)
\end{equation}
under reasonable assumptions on $\theta_1, \dots, \theta_n$ provided
$\sigma_{\max}$ is not too large (here $\sigma^2_{\max}$  is any upper
bound on $\max_{1 \leq i \leq n} \lambda_{\max} (\Sigma_i)$). As a
result, it is reasonable to believe that $\hat{\theta}_1, \dots,
\hat{\theta}_n$ (defined in \eqref{es.int})  will be close to
$\breve{\theta}_1^*, \dots, \breve{\theta}_n^*$ where 
\begin{equation}\label{es.bre}
  \breve{\theta}_i^* := X_i + \frac{\nabla
    f_{\bar{G}_n^0}(X_i)}{f_{\bar{G}_n^0}(X_i)} \qt{for $i = 1, \dots,
    n$}. 
\end{equation}
The next result rigorizes this intuition. Note that
$\breve{\theta}_i^*$ is also given by
\begin{equation}\label{hor}
  \breve{\theta}_i^* = \E (\theta | X = X_i) \qt{where $\theta \sim
    \bar{G}_n^0$ and $X|\theta \sim N(\theta, I_d)$}. 
\end{equation}
Intuitively, it makes sense that $\hat{\theta}_i$ estimates
$\breve{\theta}_i^*$ because an observation $X \sim N(\theta_0,
\Sigma)$ (with $\Sigma - I_d$ being positive semi-definite) can also be
thought of as being generated from $X | \theta \sim N(\theta, I_d)$
with $\theta \sim N(\theta_0, \Sigma - I_d)$. However, it should be
noted that $\breve{\theta}_1^*, \dots, \breve{\theta}_n^*$ is not the
best separable estimator for $\theta_1, \dots, \theta_n$ in the
heteroscedastic setting and this is explained later in this section
(after Proposition \ref{pkde.h}). 

\begin{theorem} \label{rgende}
    Let $X_1, \dots, X_n$ be independent random vectors with $X_i
    \sim N(\theta_i, \Sigma_i)$ for some covariance matrices
    $\Sigma_1, \dots, \Sigma_n$ with $\Sigma_i - I_d$ being positive
    semi-definite for every $i$. Suppose 
    $\sigma_{\max}^2$ is such that $\max_{1 \leq j \leq k}
    \lambda_{\max}(\Sigma_j) \leq \sigma^2_{\max}$ where
    $\lambda_{\max}(\Sigma_j)$ denotes the largest eigenvalue of
    $\Sigma_j$. Let $\hat{\theta}_1, \dots, \hat{\theta}_n$ be as
    defined in \eqref{es.int} and $\breve{\theta}_1^*, \dots,
    \breve{\theta}_n^*$ be as defined in \eqref{es.bre}. Then there
    exists a positive constant $C_d$ (depending only on $d$) such that
    for every non-empty compact set $S \subseteq \R^d$ and $M \geq
    \sqrt{10 \log n}$, we have 
	\begin{equation*}
	\Rr_n(\hat \theta, \breve \theta^*) := \E \left[\frac{1}{n}
          \sum_{i=1}^n \|\hat{\theta}_i - \breve{\theta}_i^* \|^2 \right] \leq C_d \sigma^2_{\max}
        \epsilon_n^2(M,S, \bar{G}_n^0) \left(\sqrt{\log n} \right)^{\max(d-2, 6)}
	\end{equation*}
    where $\epsilon_n(M, S, \bar{G}_n^0)$ is as defined in
    \eqref{si.ra}. 
\end{theorem}

Note that Theorem \ref{rgende} generalizes Theorem
\ref{theorem:denoising_theorem}. Indeed, Theorem
\ref{theorem:denoising_theorem} is the special case of Theorem
\ref{rgende} when $\Sigma_i = I_d$  for each $i$ because, in this special
case, $\sigma^2_{\max} = 1$ and $\bar{G}^0_n$, as defined in
\eqref{gengeb}, precisely equals the empirical measure corresponding to
$\theta_1, \dots, \theta_n$. Theorem \ref{rgende} leads to corollaries
that are similar to those derived from Theorem
\ref{theorem:denoising_theorem} (see, for example, Proposition
\ref{pkde.h} in the technical appendix which is the
analogue of Proposition \ref{pkde} for the heteroscedastic setting.    

We would like to remark here that Theorem \ref{rgende} is of limited
interest unless the heteroscedasticity is mild (by mild, we mean that
$\sigma^2_{\max}$ can be chosen to be close to 1). This is because
the Oracle estimator $\breve{\theta}_i^*$
(defined in \eqref{es.bre}) is different from the best separable
estimator (recall the best separable estimator is given by $T^*(X_i),
i = 1, \dots, n$ where $T^*$ minimizes \eqref{basri} over all
functions $T : \R^d \rightarrow \R^d$). A description of the best
separable estimator along with some results on the discrepancy between
the best separable estimator and \eqref{es.bre} is given in the
Section \ref{hete}. 

\section{Proof Ideas}\label{proids}
In this section, we provide a broad overview of the proofs of our main
results, Theorem \ref{dens}  and Theorem \ref{rgende}. Full proofs of
these theorems, of the remaining results in the paper as well as
statements and proofs of the supporting results that are used in the
proofs are given in the technical appendix at the end of the paper. 

\subsection{Proof overview of Theorem \ref{dens}} \label{denso}
Every estimator satisfying \eqref{apfn} is an approximate
MLE. Therefore the general theory of the rates of convergence of
maximum likelihood estimators from, say, \citet{vaartwellner96book,
  wong1995probability} can be used to bound $\hel(\hat{f}_n,
f_{\bar{G}_n})$. This general theory requires bounds on the
covering numbers of the underlying class of densities (covering
numbers are formally defined at the beginning of Section \ref{hela.pf}. In our
particular context, we need to bound covering numbers of the class
$\M$ (which consists of all densities of the form $f_G$ as $G$ varies
over all probability measures on $\R^d$). Our main covering number
result for $\M$ is stated next. 

For compact $S \subseteq \R^d$, let $\norm{\cdot}_{S}$ and
$\norm{\cdot}_{S, \nabla}$ denote pseudonorms given by 
\begin{equation*}
\norm{f}_{S} := \sup_{x \in S} |f(x)| ~~~ \text { and } ~~~
\norm{f}_{S, \nabla} := \sup_{x \in S} \norm{\nabla f(x)} 
\end{equation*} 
for densities $f \in \M$. These naturally lead to two
pseudometrics on $\M$ and we shall denote the
$\eta$-covering numbers of $\M$ under these
pseudometrics by $N(\eta, \M, \norm{\cdot}_S)$ and $N(\eta, \M,
\norm{\cdot}_{S, \nabla})$ respectively. The
following theorem, which could be of independent interest, gives upper
bounds for $N(\eta, \M, \norm{\cdot}_S)$ 
and $N(\eta, \M, \norm{\cdot}_{S, \nabla})$. We let $S^a
:= \left\{x : \dos(x) \leq a \right\}$ for $S \subseteq \R^d$
and $a > 0$ and use $N(a, S^a)$  to denote the $a$-covering number (in
the usual Euclidean distance) of the set $S^a$.  

\begin{theorem}\label{mm}
	There exists a positive constant $C_d$ depending on $d$ alone such that for every
	compact set $S \subseteq \R^d$ and $0 < \eta \leq \frac{2
          \sqrt{2\pi}}{(2\pi)^{d/2} \sqrt{e}}$, we have 
	\begin{equation}\label{mm.eq_f}
	\log N(\eta, \M, \norm{\cdot}_S) \leq C_d N(a, S^a) |\log \eta|^{d+1}
	\end{equation}
	and
	\begin{equation}\label{mm.eq_grad_f}
	\log N(\eta, \M, \norm{\cdot}_{S, \nabla}) \leq C_d N(a, S^a) |\log \eta|^{d+1}
	\end{equation}
	where $a$ is defined as
	\begin{equation} \label{eq:define_a}
	a := \sqrt{2 \log \frac{2 \sqrt{2\pi}}{(2 \pi)^{d/2} \eta}}.
	\end{equation}
\end{theorem}
To the best of our knowledge, Theorem \ref{mm} (proved in Section
\ref{mmps}) is novel
although certain special cases (such as when $d = 1$ and $S$ is a closed
interval) are known previously (see Remark \ref{premet}). The generalization for arbitrary 
compact sets $S$ is crucial for our results. Only the first assertion
(inequality \eqref{mm.eq_f}) is required for the proof of Theorem
\ref{dens}; the second assertion involving gradients is needed for the
proof of Theorem \ref{rgende}.  

Let us now sketch the proof of Theorem \ref{dens} assuming Theorem
\ref{mm}. The reader is welcome to read the full proof in the
technical appendix. As mentioned
previously, our proof is  inspired from \citet[Proof of Theorem
1]{zhang2009generalized} and differences between our proof and the
arguments of \cite{zhang2009generalized} are pointed out at the end of
this subsection. 

For simplicity, in this section, let us assume that
$\hat{f}_n$  is an NPMLE so that \eqref{apfn} holds for $\alpha =
\beta = 0.5$. The full proof (in the technical appendix) applies to
estimators satisfying \eqref{apfn} for 
arbitrary $0 < \beta \leq \alpha < 1$. Note first that trivially (for
every $t \geq 1$ and $\gamma_n > 0$)
\begin{equation*}
  \P \left\{\hel(\hat{f}_n, f_{\bar{G}_n}) \geq t \gamma_n \right\}
  = \P \left\{\hel(\hat{f}_n, f_{\bar{G}_n}) \geq t \gamma_n,
    \prod_{i=1}^n \frac{\hat{f}_n(X_i)}{f_{\bar{G}_n}(X_i)} \geq 1
  \right\}. 
\end{equation*}
The right hand side above can be easily controlled if $\hat{f}_n$ were
non-random. To deal with randomness, we cover $\M$ to within some
$\eta > 0$ in $L^{\infty}(S^M)$ (where $S^M 
:= \{x : \dos(x) \leq M\}$). From this cover, it is possible to deduce
the existence of a collection of non-random
densities $h_{0j}, j \in J$ in $\M$ for some finite set $J$ with
cardinality at most the right hand side of \eqref{mm.eq_f} such that
$\hel(h_{0j}, f_{\bar{G}_n}) \geq t \gamma_n$ and such that the inequality
\begin{equation*}
  \prod_{i=1}^n \hat{f}_n(X_i) \leq \max_{j \in J} \prod_{i : X_i \in
    S^M} \left\{h_{0j}(X_i) + 2 \eta \right\} \prod_{i : X_i \notin
    S^M} (2 \pi)^{-d/2}. 
\end{equation*}
holds whenever $\hel(\hat{f}_n, f_{\bar{G}_n}) \geq t \gamma_n$. From
here, it can be shown that for every function $v: \R^d 
\rightarrow (0, \infty)$ with $v(x) = \eta$ for $x \in S^M$, we have
  \begin{equation*}
    \prod_{i=1}^n \frac{\hat{f}_n(X_i)}{f_{\bar{G}_n}(X_i)} \leq
    \max_{j \in J} \prod_{i=1}^n \frac{h_{0j}(X_i) + 2
      v(X_i)}{f_{\bar{G}_n}(X_i)} \prod_{i : X_i \notin S^M}
    \frac{(2\pi)^{-d/2}}{2 v(X_i)}
  \end{equation*}
on the event $\hel(\hat{f}_n, f_{\bar{G}_n}) \geq t \gamma_n$. We take
\begin{equation} \label{vdef.ma}
v(x) := \begin{cases}
\eta &\text{ if } x \in S^M \\
\eta\left( \frac{M}{\mathfrak{d}_S(x)} \right)^{d+1} &\text{ otherwise }
\end{cases}
\end{equation}
The inequality above implies that 
\begin{align*}
    \P \left\{\hel(\hat{f}_n, f_{\bar{G}_n}) \geq t \gamma_n \right\}
    &\leq \sum_{j \in J} \P \left\{\prod_{i=1}^n \frac{h_{0j}(X_i) + 2
        v(X_i)}{f_{\bar{G}_n}(X_i)} \geq e^{-nt^2\gamma_n^2/2}
    \right\} \\ &+ \P \left\{\prod_{i : X_i \notin S^M} \frac{(2
        \pi)^{-d/2}}{2v(X_i)} \geq e^{nt^2\gamma_n^2/2} \right\} 
\end{align*}
The first term on the right hand side above is now controlled by
standard arguments for bounding likelihood ratio deviations in terms
of Hellinger distances (note that $\hel(h_{0j}, f_{\bar{G}_n}) \geq t
\gamma_n$). For the second term, we use Markov's inequality and the
following moment inequality (proved in Section \ref{auxre}) applied to the Lipschitz
function $g(x) := \dos(x)$. 
\begin{lemma}\label{tailmom}
Let $X_1, \dots, X_n$ be independent random variables with $X_i \sim
f_{G_i}$ and $\bar{G}_n := (G_1 + \dots + G_n)/n$. Let $g : \R^d
\rightarrow [0, \infty)$ be a $1$-Lipschitz function i.e., $g(x) -
g(y) \leq \|x - y\|$ for all $x, y \in \R^d$. Also let $\mu_p(g)$
denote the $p^{th}$ moment of $g$ under the 
measure $\bar{G}_n$ i.e.,  
\begin{equation*}
  \mu_p(g) := \left(\int_{\R^d} g(\theta)^p d\bar{G}_n(\theta)
  \right)^{1/p}.  
\end{equation*}
There then exists a positive constant $C_d$ depending only on $d$ such
that 
\begin{equation}\label{tailmom.eq}
\begin{split}
\E \left\{ \prod_{i=1}^n \left| a g(X_i) \right|^{I\{g(X_i)\geq
    M\}} \right\}^{\lambda} &\leq \exp\left\{C_d a^{\lambda} M^{\lambda
    + d - 2} \right. \\ & \left. + (aM)^{\lambda} n
  \left(\frac{2\mu_p(g)}{M}\right)^p\right\}    
\end{split}
\end{equation}
for every $a > 0, M \geq \sqrt{8 \log n}$ and $0 < \lambda \leq
\min(1, p)$. 

Further, there exists a positive constant $C_d$ depending only on $d$ such
	that
	\begin{equation} \label{tailprob.eq}
	\frac{1}{n} \sum_{i=1}^n \P\left[ g(X_i) \geq M \right] \leq
        C_d \frac{M^{d-2}}{n} + \inf_{p \geq \frac{d+1}{2 \log n}}
        \left( \frac{2 \mu_p(g)}{M} \right)^p 
	\end{equation}
	 for any $M \geq \sqrt{8 \log n}$.
\end{lemma}
The differences between our proof and that of \citet[Proof of Theorem
1]{zhang2009generalized} are the metric entropy
result (Theorem \ref{mm}), the breakup of the likelihood ratio into the
sets $S^M$ and $(S^M)^c$, the choice of $v(\cdot)$ function in
\eqref{vdef.ma} and the moment control in Lemma \ref{tailmom}.
\citet{zhang2009generalized} proved special cases of these ingredients
for $d = 1$ and $S = [-R, R]$  for some $R$ while our argument applies
to every $S$. As remarked previously, it is crucial to allow $S$ to be
arbitrary for obtaining adaptation results to discrete mixtures.  

\subsection{Proof overview of Theorem \ref{rgende}} \label{gendo}
A complete proof of Theorem \ref{rgende} is given in Section
\ref{pfsec.rgende}. This
subsection gives an overview of the main ideas. Let us now introduce
the following notation. Let $\bx$ denote the  
$d \times n$ matrix whose columns are the observed data vectors $X_1,
\dots, X_n$. For a density $f \in \M$, let $T_f(\bx)$ denote the $d
\times n$ matrix whose $i^{th}$ column is given by the $d \times 1$
vector: 
\begin{equation*}
  X_i + \frac{\nabla f(X_i)}{f(X_i)} \qt{for $i = 1, \dots, n$}. 
\end{equation*}
With this notation, we can clearly rewrite $\Rr_n(\hat{\theta},
\breve{\theta}^*)$ as 
\begin{equation*}
  \Rr_n(\hat{\theta}, \breve{\theta}^*) = \E \left(\frac{1}{n}
    \norm{T_{\hat{f}_n}(\bx) - T_{f_{\bar{G}^0_n}}(\bx)}_F^2 \right) 
\end{equation*}
where $\norm{\cdot}_F$ denotes the usual Frobenius norm for matrices. 

Now for $f \in \M$ and $\rho > 0$, let $T_f(\bx, \rho)$ be the $d
\times n$ matrix whose $i^{th}$ column is given by the $d \times 1$
vector: 
\begin{equation*}
  X_i + \frac{\nabla f(X_i)}{\max(f(X_i), \rho)} \qt{for $i = 1, \dots, n$}. 
\end{equation*}
The first important observation is that for $\rho_n := (2
\pi)^{-d/2}/n$, we have $T_{\hat{f}_n}(\bx, \rho_n) =
T_{\hat{f}_n}(\bx)$ and this follows from classical results about the
NPMLE. This allows us to write
\begin{align*}
  \Rr_n(\hat{\theta}, \breve{\theta}^*) &= \E \left(\frac{1}{n}
    \norm{T_{\hat{f}_n}(\bx, \rho_n) - T_{f_{\bar{G}^0_n}}(\bx)}_F^2
                                        \right)  \\
&\leq 2 \E \left(\frac{1}{n}
    \norm{T_{\hat{f}_n}(\bx, \rho_n) - T_{f_{\bar{G}^0_n}}(\bx, \rho)}_F^2
                                        \right) \\ &+ 2 \E \left(\frac{1}{n}
    \norm{T_{f_{\bar{G}^0_n}}(\bx, \rho_n) - T_{f_{\bar{G}^0_n}}(\bx)}_F^2
                                        \right). 
\end{align*}
Using the following lemma (proved in Section \ref{auxre}), the second term above is
bounded from above by $(\sqrt{\log n})^{\max(d-2, 0)} \epsilon_n^2(M,
S, \bar{G}_n^0)$.  
\begin{lemma}\label{dco}
  Fix a probability measure $G$ on $\R^d$ and let $0 < \rho \leq (2
  \pi)^{-d/2}/\sqrt{e}$. Let $L(\rho) := \sqrt{-\log((2\pi)^d
    \rho^2)}$. Then there exists a positive constant $C_d$
  such that for every compact set $S \subseteq \R^d$, we have 
  \begin{equation}\label{dco.eq}
    \begin{split}
    \Delta(G, \rho) &:= \int \left(1 - \frac{f_G}{\max(f_G, \rho)}
  \right)^2 \frac{\norm{\nabla f_G}^2}{f_{G}} \\ &\leq C_d N
  \left(\frac{4}{L(\rho)}, S \right)  L^d(\rho) \rho + d ~G(S^c).  
   \end{split}
  \end{equation}
\end{lemma}
We thus focus attention on the first term in the above bound
for $\Rr_n(\hat{\theta}, \breve{\theta}^*)$: 
\begin{equation*}
  A(\hat{f}_n) := \E \left(\frac{1}{n}
    \norm{T_{\hat{f}_n}(\bx, \rho_n) - T_{f_{\bar{G}^0_n}}(\bx, \rho)}_F^2
                                        \right). 
\end{equation*}
Now if $\hat{f}_n$ were non-random, the above term can be bounded from
above by a generalization (to $d \geq 1$) of \citet[Theorem
3]{jiang2009general} which bounds $A(f)$ in terms of $\hel^2(f,
f_{\bar{G}_n^0})$ for non-random $f \in \M$. We have stated this
general result as Theorem \ref{fcc} and proved it in Section
\ref{zhafo}. Of course,
this result cannot be directly used here because $\hat{f}_n$ is
random. However, Theorem \ref{dens} implies that $\hat{f}_n$ belongs
with high probability (specifically with probability at least $1 -
(2/n)$) to the set 
\begin{equation*}
  E_n := \left\{f \in \M : \hel(f, f_{\bar{G}_n^0}) \leq C_d
    \epsilon_n(M, S, \bar{G}_n^0)\right\}
\end{equation*}
where $C_d$ is the constant obtained from Theorem \ref{dens}. The idea
therefore is to cover the space $E_n$ to within $\eta$ by
deterministic densities $f_{G_1}, \dots, f_{G_N}$. For this covering,
we use the metric:  
\begin{equation*}
  \sup_{x : \dos(x) \leq M} \norm{\frac{\nabla f(x)}{\max(f(x),
      \rho_n)} - \frac{\nabla g(x)}{\max(g(x), \rho_n)}}. 
\end{equation*}
Covering numbers in this metric are given in Corollary \ref{rgnt} in
and this result is derived as
a corollary of our main covering number result in Theorem
\ref{mm}. With these deterministic densities, we bound $A(\hat{f}_n)$
via $A(\hat{f}_n) \leq 4 \sum_{i=1}^n (\E \zeta_{in}^2/n)$ where
\begin{align*}
\zeta_{1n} &:= \|T_{\hat{f}_n}(\mathbf{X}, \rho_n) -
             T_{f_{\bar{G}^0_n}}(\mathbf{X}, \rho_n) \|_F I\left\{\hat{f}_n
             \notin E_n\right\} \\
\zeta_{2n} &:= \left( \|T_{\hat{f}_n}(\mathbf{X}, \rho_n) -
             T_{f_{\bar{G}^0_n}}(\mathbf{X}, \rho_n) \|_F \right. \\ & \left. - \max_{1 \leq j
	\leq N}   \|T_{f_{G_j}}(\mathbf{X}, \rho_n) -
                                                                       T_{f_{\bar{G}^0_n}}(\mathbf{X},
                                                                       \rho_n)
                                                                       \|_F
                                                                       \right)_+
                                                                       I\left\{\hat{f}_n
                                                                       \in
                                                                       E_n
                                                                       \right\}
  \\
\zeta_{3n} &:= \max_{1 \leq j \leq N} \left(\|T_{f_{G_j}}(\mathbf{X},
             \rho_n) - 
T_{f_{\bar{G}_n^0}}(\mathbf{X}, \rho_n) \|_F \right. \\ & \left. - \E
                                                          \|T_{f_{G_j}}(\mathbf{X},
                                                          \rho_n) -
                                                          T_{f_{\bar{G}^0_n}}(\mathbf{X},
                                                          \rho_n) \|_F
                                                          \right)_+  \\
\zeta_{4n} &:= \max_{1 \leq j \leq N} \E \|T_{f_{G_j}}(\mathbf{X},
             \rho_n) - T_{f_{\bar{G}^0_n}}(\mathbf{X}, \rho_n) \|_F. 
\end{align*} 
Each of these terms is controlled to finish the proof of Theorem
\ref{rgende} in the following way: 
\begin{enumerate}
\item $\E \zeta_{1n}^2/n$ is bounded by $C_d \epsilon_n^2(M, S,
  \bar{G}_n^0)$ because 
  $\P\{\hat{f}_n \notin E_n\} \leq (2/n)$ and the fact that
  $T_{f}(\bx, \rho_n)$ can be bounded by a term involving $\rho_n$
  alone (this result is stated (and proved) as Lemma \ref{p1}).  
\item $\E \zeta_{2n}^2/n$ is bounded by $C_d \epsilon_n^2(M, S,
  \bar{G}_n^0)$ using the fact that $f_{G_1}, \dots, f_{G_N}$ form a
  covering of $E_n$. 
\item $\E \zeta_{3n}^2/n$ is bounded by $C_d \sigma^2_{\max}
  \epsilon_n^2(M, S, \bar{G}_n^0) (\log n)^2$ using measure
  concentration properties of Gaussian random variables and an upper
  bound  on $N$ which is given by the covering number result in Corollary
  \ref{rgnt}.  
\item $\E \zeta_{4n}^2/n$ is bounded by $C_d \epsilon_n^2(M, S,
  \bar{G}_n^0) (\log n)^3$ by Theorem \ref{fcc}
  which bounds $A(f)$ in terms of $\hel(f, f_{\bar{G}_n^0})$ 
  for every non-random $f$. 
\end{enumerate}
The structure of the proof and the main ideas are very similar to that
of \citet[Proof of Theorem 5]{jiang2009general}. Other than the fact
that our arguments hold for $d \geq 2$ and arbitrary
compact sets $S$, additional differences between our
proof and \cite[Proof of Theorem 5]{jiang2009general} are as
follows. The breakdown of the risk $\Rr_n(\hat{\theta},
\breve{\theta}^*)$ into various terms is different as
the authors of \cite{jiang2009general} work with the discrepancy measure
\eqref{jidis} while we work directly with the discrepancy between
$\hat{\theta}$ and $\breve{\theta}^*$. Our argument for
$T_{\hat{f}_n}(\bx) = T_{\hat{f}_n}(\bx, \rho_n)$ (given in inequality
\eqref{rob} near the beginning of the proof of Theorem \ref{rgende})
is more direct compared to the corresponding argument in
\cite[Proposition 2]{jiang2009general}. Our measure concentration
result (see Lemma \ref{devs}) involves $X_i \sim N(\theta_i,
\Sigma_i)$ and not Gaussian random vectors with 
identity covariance as in \cite[Proposition 4]{jiang2009general}.
Our control of $\E \zeta_{5n}^2/n$ (via Lemma \ref{dco}) is different
from and probably more direct compared to the corresponding argument
in \cite[Theorem 3(ii)]{jiang2009general}.

\section{Implementation Details and Some Simulation
  Results}\label{imple} 

In this section, we shall discuss some computational details
concerning the NPMLE and also provide numerical evidence for the
effectiveness of the estimator \eqref{es.int} based on the NPMLE for
denoising. 

For the optimization problem \eqref{kw}, it can be shown
that $\hat{f}_n$ exists and is non-unique. However $\hat{f}_n(X_1),
\dots, \hat{f}_n(X_n)$ are unique and they solve the finite
dimensional optimization problem:
\begin{align}
&\text{ argmax } \sum_{i=1}^n \log f_i \label{prob:conv_form} \\
&\text{ s.t. }  (f_1, \dots, f_n) \in \text{Conv}\left\{
  (\phi(X_1 - \theta), \dots, \phi(X_n - \theta)) : \theta \in \R^d
  \right\}. \nonumber 
\end{align}
where $\text{Conv}$ above stands for convex hull. The constraint set
in the above problem however involves every $\theta 
\in \R^d$. A natural way of computing an approximate solution is to fix a
finite data-driven set $F := \{a_1, \dots, a_m\} \subseteq \R^d$ and
restrict the infinite convex hull to the convex hull over $\theta$
belonging to this set. This leads to the problem: 
\begin{align}
&\text{ argmax } \sum_{i=1}^n \log f_i \label{prob:conv_approx} \\
&\text{ s.t. } (f_1, \dots, f_n) \in \text{Conv}\left\{
  (\phi_d(X_1 - \theta), \dots, \phi_d(X_n - \theta)) : \theta \in F 
  \right\}. \nonumber   
\end{align}
This can also be seen as an approximation to \eqref{kw} where the
densities $f \in \M$ are restricted to have atoms in $\{a_1, \dots,
a_m\} \subseteq \R^d$. \eqref{prob:conv_approx} is a convex
optimization problem over the probability simplex in $m$
dimensions and can be solved using many algorithms (for 
example, standard interior point methods as implemented in the
software, Mosek, can be used here). 

The effectiveness of \eqref{prob:conv_approx} as an approximation to
\eqref{kw} depends crucially on the choice of $\{a_1, \dots,
a_m\}$. For $d = 1$, \citet{koenker2014convex} propose the use of a
uniform grid within the range $[\min_{1 \leq i
  \leq n} X_i, \max_{1 \leq i   \le n} X_i]$ of the data. \citet{dicker2016high}
discuss this approach in more detail and recommend the choice $m :=
[\sqrt{n}]$. They also prove (see \cite[Theorem 2]{dicker2016high})
that the resulting approximate MLE, $\tilde{f}_n$, has a squared
Hellinger accuracy, $\shel(\tilde{f}_n, f_0)$, of $O_p((\log n)^2/n)$
when the mixing measure corresponding to $f_0$ has bounded
support. For $d \geq 1$, \citet{feng2016nonparametric} recommend
taking a regular grid in a compact region containing the data. They
also mention that empirical results seem ``fairly insensitive'' to the
choice of $m$. 

A proposal for selecting $\{a_1, \dots, a_m\}$ that is different from
gridding is the so called ``exemplar'' choice where one takes $m = n$
and $a_i = X_i$ for $i = 1, \dots, n$. This choice is proposed in
\citet{bohning2000computer} for $d = 1$ and in
\citet{lashkari2008convex} for $d \geq 1$. This avoids gridding which
can be problematic in multiple dimensions. Also, this method is
computationally feasible as long as $n$ is moderate (up to a few
thousands) but becomes expensive for larger $n$. In such instances, a
reasonable strategy is to take $a_1, \dots, a_m$ as a random subsample
of the data $X_1, \dots, X_n$. For fast implementations, one can also
extend the idea of \citet{koenker2014convex} by binning the
observations and weighting the likelihood terms in \eqref{kw} by
relative multinomial bin counts.

We shall now provide some graphical evidence of the effectiveness of the
NPMLE for denoising. For our plots, the NPMLE is approximately
computed via the algorithm \eqref{prob:conv_approx} where $a_1,
\dots, a_m$ are chosen to be the data points $X_1, \dots, X_n$ with $m
= n$ (i.e., we follow the exemplar recommendation of
\cite{bohning2000computer} and \cite{lashkari2008convex}). We use
the software, Mosek, to solve \eqref{prob:conv_approx}. The results of this paper do not apply
directly to these approximate NPMLEs and extending them is the subject
of future work. We argue however via simulations that these
approximate NPMLEs work well for denoising.  

In Figure \ref{fig:denoising_illustrations}, we illustrate the
performance of $\hat{\theta}_1, \dots, \hat{\theta}_n$ (defined as in
\eqref{es.int}) for denoising when the true vectors $\theta_1, \dots,
\theta_n$ take values in a bounded region of $\R^2$. The plots refer
to these estimates as the Empirical Bayes estimates and the quantities
\eqref{ob} as the Oracle Bayes estimates. 
\begin{figure} 
    \centering
    \begin{subfigure}[t]{0.45\textwidth} 
        \centering
        \includegraphics[width = 0.95\textwidth]{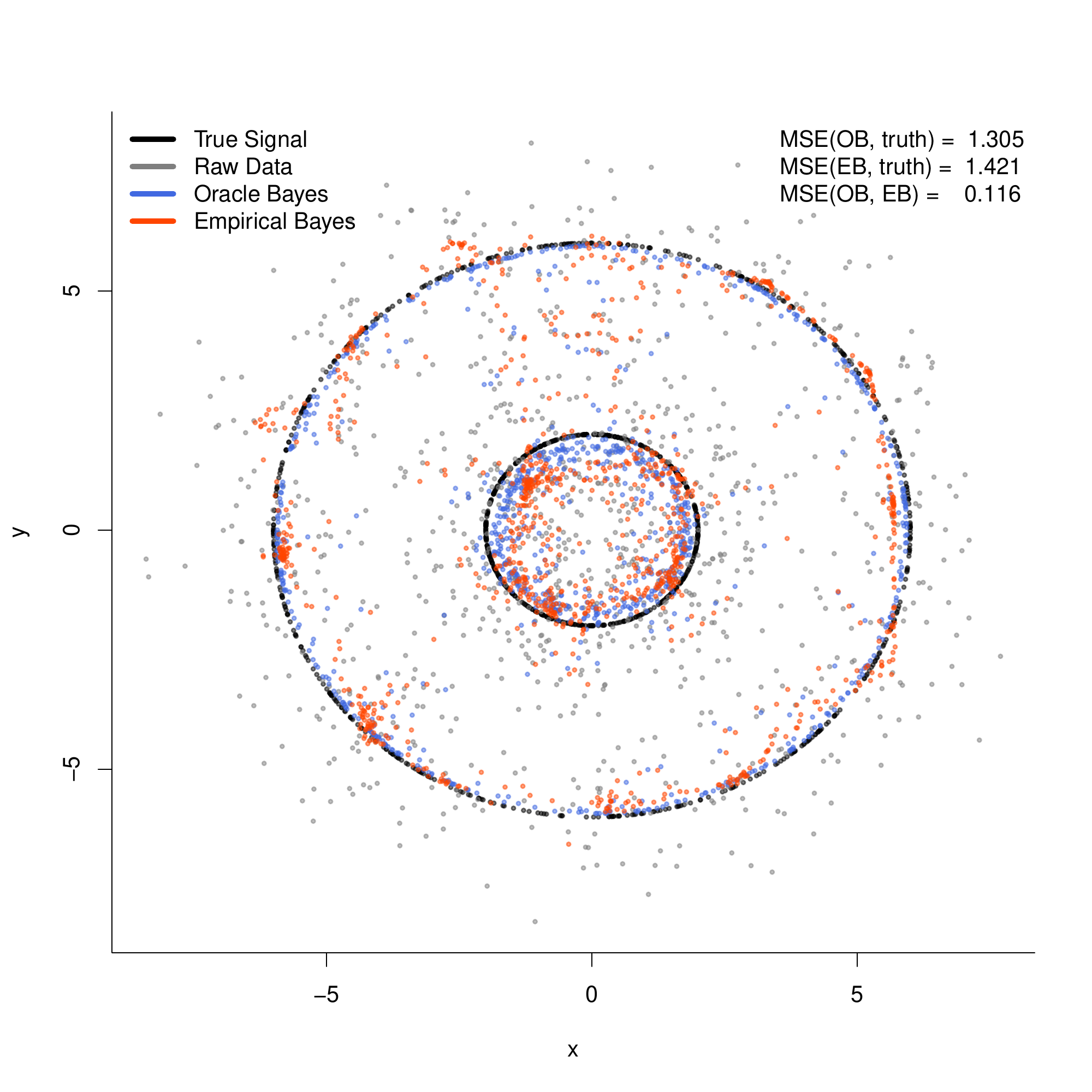}
        \caption{\textbf{Two circles:} $n = 1000$. Half of $\{\theta_i\}_{i=1}^n$ are drawn uniformly at random from each of the concentric circles of radii $2$ and $6$ respectively.}
    \end{subfigure}%
    ~ 
    \centering
    \begin{subfigure}[t]{0.45\textwidth}
        \centering
        \includegraphics[width = 0.95\textwidth]{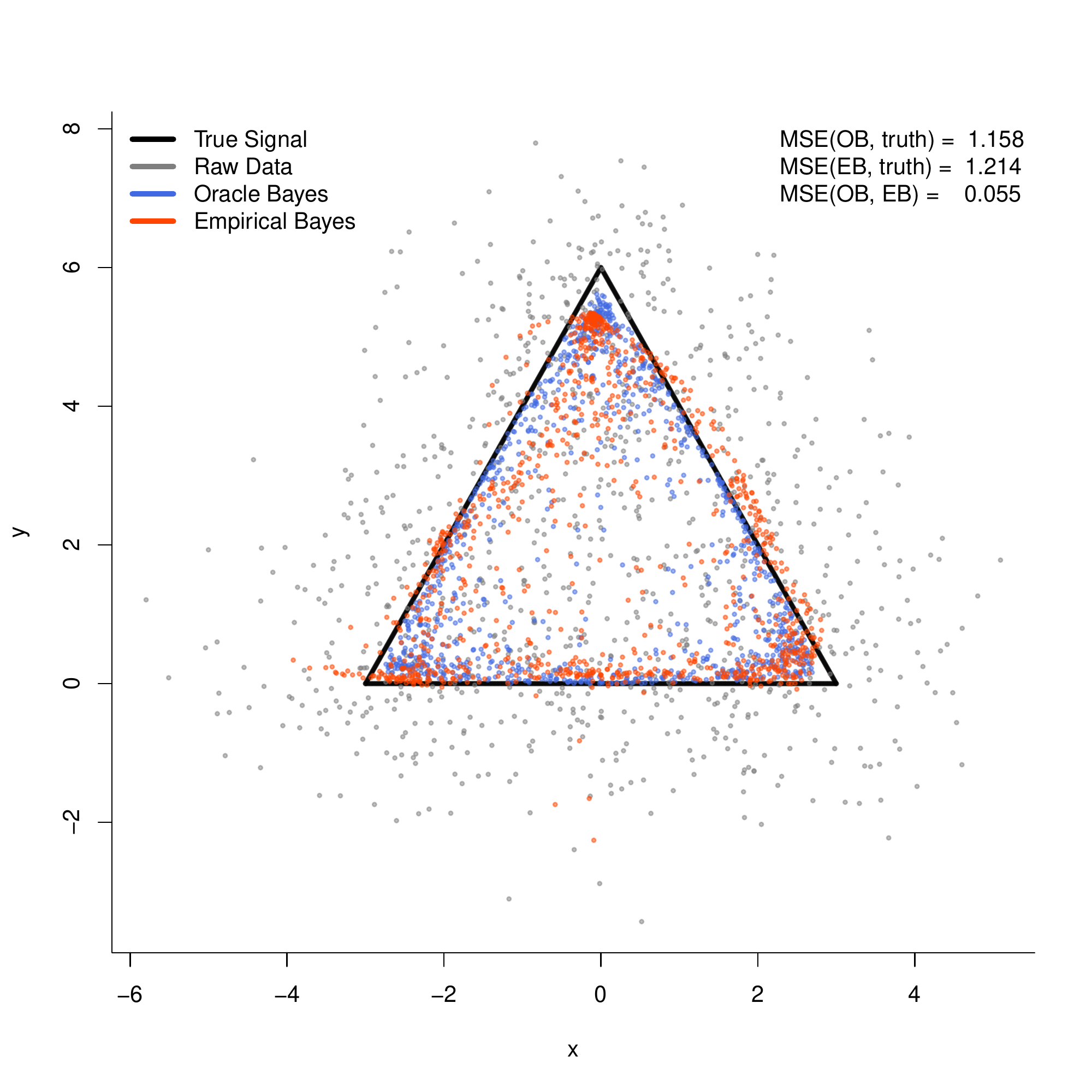}
        \caption{\textbf{Triangle:} $n = 999$. A third of $\{\theta_i\}_{i=1}^n$ are drawn uniformly at random from each edge of the triangle with vertices $(-3,0)$, $(0,6)$ and $(3,0)$}
    \end{subfigure}%
    ~ 
    \vspace{5mm}
    \\
    ~
    \centering
    \begin{subfigure}[t]{0.45\textwidth}
        \centering
        \includegraphics[width = 0.95\textwidth]{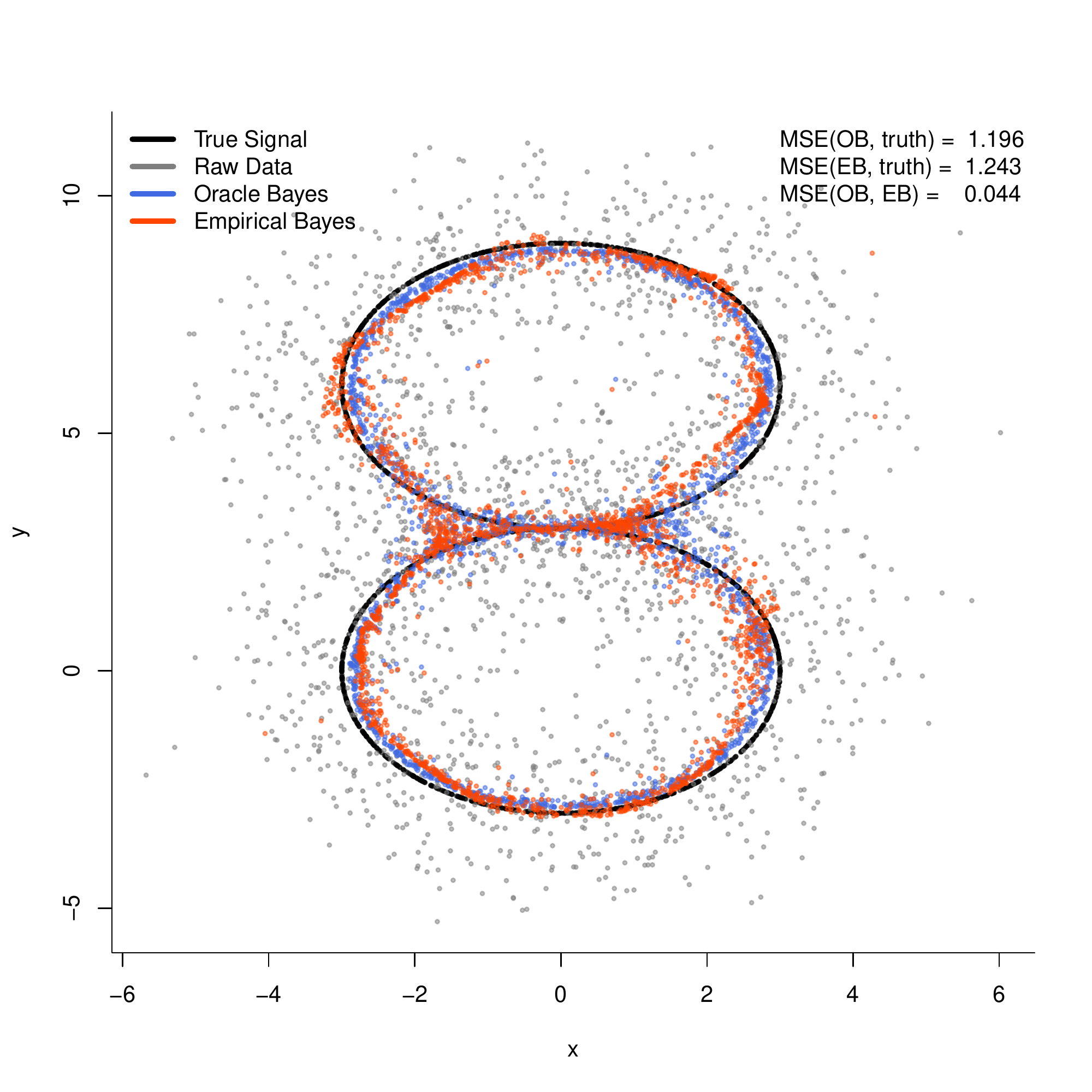}
        \caption{\textbf{Digit $8$:} $n = 1000$. Half of $\{\theta_i\}_{i=1}^n$ are drawn uniformly at random from each of the circles of radii $3$ cnetered at $(0,0)$ and $(0,6)$ respectively.}
    \end{subfigure}%
    ~ 
    \centering
    \begin{subfigure}[t]{0.45\textwidth}
        \centering
        \includegraphics[width = 0.95\textwidth]{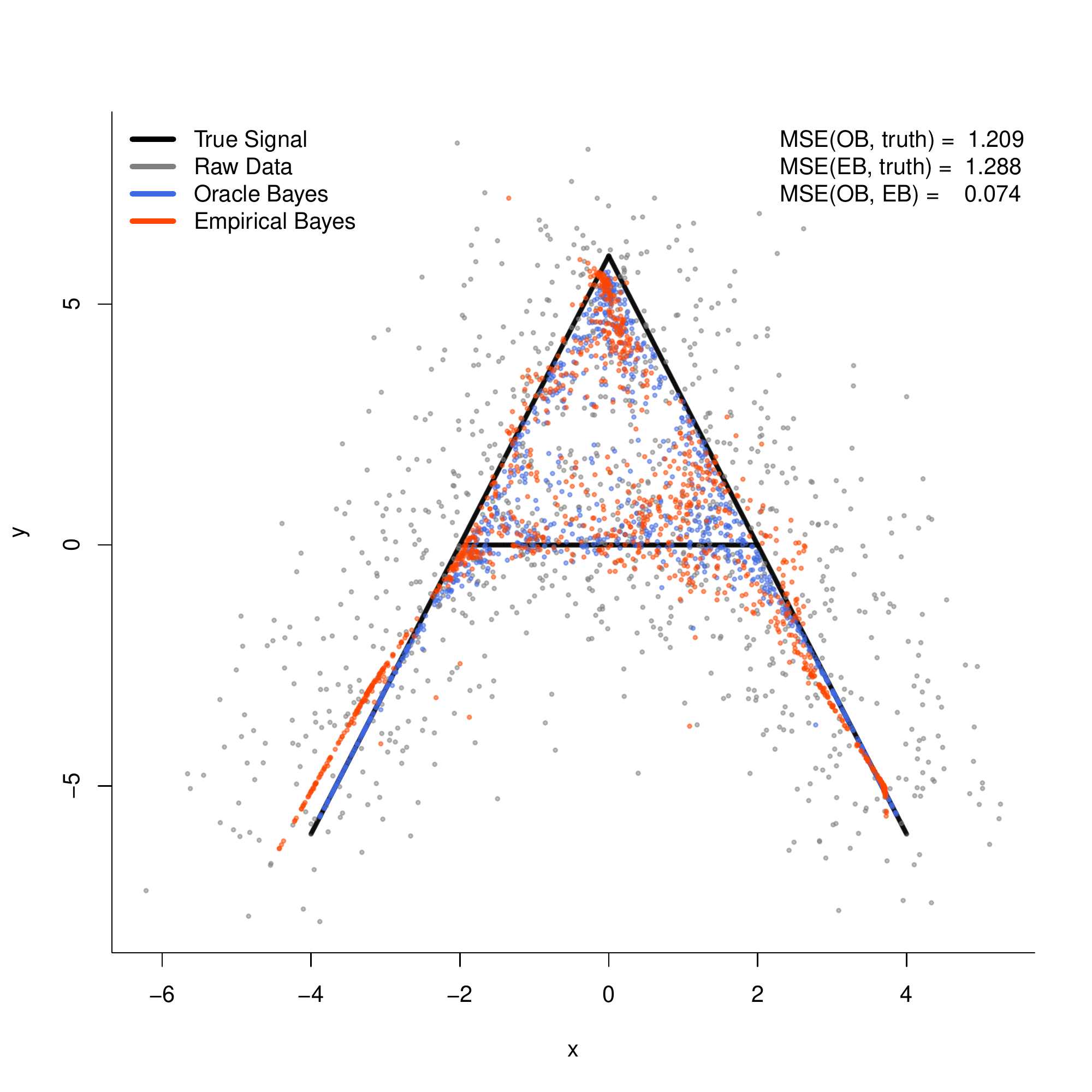}
        \caption{\textbf{Letter A:} $n = 1000$. A fifth of $\{\theta_i\}_{i=1}^n$ are drawn uniformly at random from each of the line segments joining the points $(-4,-6), (-2,0), (0,6), (2,0)$ and $(4,6)$ so as to form the letter A.}
    \end{subfigure}%
    ~ 
    \caption{Illustrations of denoising using the Empirical Bayes
      estimates \eqref{es.int}}
    \label{fig:denoising_illustrations}
\end{figure}
In each of the four subfigures in Figure
\ref{fig:denoising_illustrations}, we generate $n$ vectors $\theta_1,
\dots, \theta_n$ from a bounded region in $\R^d$  for $d = 2$: they
are generated from two concentric circles in the first subfigure,
a triangle in the second subfigure, the digit 8 in the third subfigure
and the uppercase letter A in the last subfigure. In each
of these cases, the empirical measure $\bar{G}_n$ is supported on a
bounded region so that Corollary \ref{dna} yields the near parametric
rate $1/n$ up to logarithmic multiplicative factors in $n$ for every
NPMLE. In each of 
the subfigures in Figure \ref{fig:denoising_illustrations}, we plot
the true parameter values $\theta_1, \dots, \theta_n$ in black, the
data $X_1, \dots, X_n$ (generated independently according to $X_i \sim
N(\theta_i, I_2)$) are plotted in gray, the Oracle Bayes estimates
$\hat{\theta}_1^*, \dots, \hat{\theta}_n^*$ are plotted in blue while
the estimates $\hat{\theta}_1, \dots, \hat{\theta}_n$ are plotted in
red. The mean squared discrepancies: 
\begin{equation*}
  \frac{1}{n} \sum_{i=1}^n \norm{\hat{\theta}_i^* - \theta_i}^2, ~~~
  \frac{1}{n} \sum_{i=1}^n \norm{\hat{\theta}_i - \theta_i}^2 \text{
    and } ~~~
  \frac{1}{n} \sum_{i=1}^n \norm{\hat{\theta}_i^* - \hat{\theta}_i}^2  
\end{equation*}
are given in each figure in the legend at the upper-right corner.
Note that the third MSE is much smaller than the other two in each
subfigure. 

As can be observed from Figure \ref{fig:denoising_illustrations}, the
Empirical Bayes estimates \eqref{es.int} approximate their targets
\eqref{ob} quite well. The most noteworthy fact is that the estimates
\eqref{es.int} do not require any knowlege of the underlying structure
in $\bar{G}_n$, for instance, concentric circles, or triangle or a
letter of the alphabet etc. We should also note here that the noise
distribution here is completely known to be $N(0, I_d)$ which implies,
in particular, that there is no unknown scale parameter representing
the noise variance.  

We have also done numerical simulations for illustrating the denoising
performance of $\hat{\theta}_1, \dots, \hat{\theta}_n$ in the case
when $\theta_1, \dots, \theta_n$ have a clustering structure. Due to
space constraints, these results have been moved to Section
\ref{desim} of the technical appendix. 

\newpage

\appendix

\title{Supplement to ``On the nonparametric maximum likelihood estimator for
  Gaussian location mixture densities with application to
  Gaussian denoising''}

\vspace{0.2in}

This technical appendix contains proofs of all results in the main
paper. Some observations on the heteroscedastic Gaussian denoising
problem are also given in this appendix. The proofs for results in
Section \ref{hela} are given in Section \ref{hela.pf} while the proofs
for Section \ref{gadeno} are in Section \ref{gadeno.pf}. Section
\ref{hete} contains some results and remarks on the heteroscedastic
problem. Metric entropy results for multivariate Gaussian  
location mixture densities play a crucial rule in the proofs of the
main results; these results are stated and proved in Section
\ref{mmps}. Section \ref{zhafo} contains the statement and proof for a
crucial ingredient for the proof of the main denoising
theorem. Finally, additional technical results needed in the proofs of
the main results are collected in Section \ref{auxre} together with
their proofs while additional simulation results are in Section
\ref{desim}.     . 
 
\section{Proofs of results in Section \ref{hela}} \label{hela.pf}
The following notation will be used in the proofs in the sequel. 

\noindent \textbf{1.} For $x \in \R^d$ and $a > 0$, let $B(x, a) :=
\{u \in \R^d: \|u - x\| \leq a\}$ denote the closed ball of radius $a$
centered at $x$.     

\noindent \textbf{2.} For a subset $S \subseteq \R^d$ and $a > 0$, we denote the set $S^a$
by 
\begin{equation}\label{saf}
  S^a := \cup_{x \in S} B(x, a) = \left\{y : \dos(y) \leq a \right\}
\end{equation}
where $\dos(\cdot)$ is defined as in \eqref{ds}. 

\noindent \textbf{3.} For a compact subset $S$ of $\R^d$ and $\epsilon > 0$, we
denote the $\epsilon$-covering number of $S$ in the usual Euclidean
distance by $N(\epsilon, S)$ i.e., $N(\epsilon, S)$ stands for the
smallest number of closed balls of radius $\epsilon$ whose union
contains $S$. 

\noindent \textbf{4.} Given a pseudometric $\varrho$ on $\M$, let $N(\epsilon, \M, \varrho)$
  denote the $\epsilon$-covering number of $\M$ under the pseudometric
  $\varrho$ by $N(\epsilon, \M, \varrho)$ i.e., $N(\epsilon, \M,
  \varrho)$ denotes the smallest positive integer $N$ for which there
  exist densities $f_1, \dots, f_N \in \M$ satisfying  
  \begin{equation*}
    \sup_{f \in \M} \inf_{1 \leq i \leq N} \varrho(f, f_i) \leq
    \epsilon. 
  \end{equation*}
  In the proof below, we will be concerned with $N(\epsilon, \M,
  \varrho)$ for the following choice of $\varrho$. For a compact set
  $S$, let $\norm{\cdot}_{\infty, S}$ denote the pseudonorm on $\M$
  defined by   
\begin{equation}\label{sups}
\|f\|_{\infty, S} := \sup_{x \in S} |f(x)|. 
\end{equation}
This pseudonorm naturally induces a pseudometric on $\M$ given by
$\varrho(f, g) := \norm{f - g}_{\infty, S}$. The covering number for
this pseudometric will be denoted by $N(\epsilon, \M,
\norm{\cdot}_{\infty, S})$. In the proofs for the results in Section
\ref{gadeno}, we will need to deal with covering numbers for other
pseudometrics $\varrho$ on $\M$ as well. These pseudometrics will be
introduced in Section \ref{gadeno.pf}.  

With the above notation in place, we are now ready to give the proof
of Theorem \ref{dens}. This proof uses additional ingredients which
are proved in later sections. Arguably the most important ingredient
for the proof of this theorem is a bound on the covering numbers
$N(\epsilon, \M, \norm{\cdot}_{\infty,S})$ which is stated as
inequality \eqref{mm.eq_f} in Theorem \ref{mm}. Other ingredients include inequality
  \eqref{Eq:tailExpect} (which is a consequence of Lemma
  \ref{tailmom}) and a standard fact (Lemma \ref{volm}) giving a 
  volumetric upper bound for Euclidean covering numbers.

\subsection{Proof of Theorem \ref{dens}} 
\begin{proof}[Proof of Theorem \ref{dens}]
We shall prove inequalities \eqref{dens.eq} and \eqref{dens.1} under
the assumption that the sample size $n$ satisfies  
\begin{equation}\label{ncon}
         n \geq \max \left(\exp\left(\frac{d+1}{2} \right), \frac{1}{2}(2
         \pi)^{(d-1)/2} \right). 
\end{equation}
If \eqref{ncon} is not satisfied, then $\epsilon_n(M, S, \bar{G}_n)$ (and 
 $\epsilon_n(M, S, \bar{G}_n)/\min(1 - \alpha, \beta)$ which is larger than
 $\epsilon_n(M, S, \bar{G}_n)$) will 
be bounded from below by a positive constant $\kappa_d$. We can then
therefore choose $C_d$ in \eqref{dens.eq} and \eqref{dens.1} large 
enough so that $\epsilon_n(M, S, \bar{G}_n) \sqrt{C_d} > \sqrt{2
  \min(1 - \alpha,   \beta)}$. Because the Hellinger distance
$\hel(\hat{f}_n, f_{\bar{G}_n})$ is always bounded from above by
$\sqrt{2}$, the probability on the left hand side of \eqref{dens.eq}
will then equal zero so that \eqref{dens.eq} holds
trivially. Inequality \eqref{dens.1}  will also be trivial because its
right hand side will then be larger than 2. 

Let us therefore fix $n$ satisfying \eqref{ncon}. Fix a positive
sequence $\{\gamma_n\}$ and assume that $\hat{f}_n$ satisfies  
\begin{equation}\label{bahu}
  \prod_{i=1}^n \frac{\hat{f}_n(X_i)}{f_{\bar{G}_n}(X_i)} \geq \exp
  \left((\beta - \alpha) n \gamma_n^2 \right) \qt{for some $0 < \beta
    \leq \alpha < 1$}. 
\end{equation}
We shall then bound the probability  
\begin{equation*}
  \P \{\hel(\hat{f}_n, f_{\bar{G}_n}) \geq t \gamma_n\}  \qt{for $t
    \geq 1$}. 
\end{equation*}
Fix a non-empty compact set $S \subseteq \R^d$ and $M \geq \sqrt{10
  \log n}$. We shall work with the set $S^M$ (defined as in
\eqref{saf}) and the pseudometric given by the pseudonorm
$\norm{\cdot}_{\infty, S^M}$ (defined as in \eqref{sups}). 

Let $\eta := 1/n^2$ and let $\{h_1, \dots, h_N\} \subseteq
\M$ denote an $\eta$-covering set of $\M$ in the pseudometric given by 
$\|\cdot\|_{\infty, S^M}$ where $N = N(\eta, \M, \|\cdot\|_{\infty,
  S^M})$ i.e.,   
\begin{equation*}
  \sup_{h \in \M} \inf_{1 \leq j \leq N} \|h - h_j\|_{\infty, S^M} \leq
  \eta.  
\end{equation*}
Inequality \eqref{mm.eq_f} in Theorem \ref{mm} gives an upper bound
for $N$ that will be crucially used in this proof. 

Let $J$ denote the set of all $j \in \{1, \dots, N\}$ for which there
exists a density $h_{0j} \in \M$ satisfying  
\begin{equation*}
	\|h_{0j} - h_j\|_{\infty, S^M} \leq \eta ~~ \text{ and } ~~
	\hel(h_{0j}, f_{\bar{G}_n}) \geq t \gamma_n. 
\end{equation*}
Because $h_1, \dots, h_N$ cover $\M$, there will exist $1 \leq j
\le N$ such that $\|h_j - \hat{f}_n \|_{\infty, S^M} \le \eta$. If
$\hel(\hat{f}_n, f_{\bar{G}_n}) \geq t \gamma_n$, then $j \in J$ 
and consequently
\begin{equation}\label{tpp1}
	\|\hat{f}_n - h_{0j} \|_{\infty, S^M} \leq 2 \eta. 
\end{equation}
We now define a function $v :=  v_{S, M} : \R^d \rightarrow (0,
\infty)$ via    
\begin{equation} \label{vdef}
v(x) := \begin{cases}
\eta &\text{ if } x \in S^M \\
\eta\left( \frac{M}{\mathfrak{d}_S(x)} \right)^{d+1} &\text{ otherwise }
\end{cases}
\end{equation}
where $\mathfrak{d}_S: \R^d \rightarrow [0, \infty)$ is defined as in
\eqref{ds}.  

Inequality \eqref{tpp1} clearly implies that $\hat{f}_n(X_i) \leq
h_{0j}(X_i) + 2 \eta = h_{0j}(X_i) + 2 v(X_i)$ whenever $X_i \in S^M$
which allows us to write   
   \begin{equation*}
     \prod_{i=1}^n \hat{f}_n(X_i) \leq \prod_{i: X_i \in S^M} 
                                    \left\{ h_{0j}(X_i)
                                    + 2 v(X_i)\right\} \prod_{i : X_i
                                    \notin S^M} (2\pi)^{-d/2} 
   \end{equation*}
   where we used the bound $\hat{f}_n(X_i) \leq \sup_{x} \hat{f}_n(x)
   \leq (2 \pi)^{-d/2}$ for $X_i \notin S^M$ (the bound $\sup_x f(x)
   \leq (2 \pi)^{-d/2}$ holds for every $f \in \M$ as can easily be
   seen). From here, we deduce 
   \begin{align*}
     \prod_{i=1}^n \hat{f}_n(X_i) &\leq \prod_{i=1}^n  
                                    \left\{ h_{0j}(X_i)
                                    + 2 v(X_i)\right\} \prod_{i : X_i
                                    \notin S^M}
                                    \frac{(2\pi)^{-d/2}}{h_{0j}(X_i) +
                                    2 v(X_i)} \\
&\leq \prod_{i=1}^n  
                                    \left\{ h_{0j}(X_i)
                                    + 2 v(X_i)\right\} \prod_{i : X_i
                                    \notin S^M}
                                    \frac{(2\pi)^{-d/2}}{2 v(X_i)}
   \end{align*}
We have therefore proved that the inequality
  \begin{equation*}
    \prod_{i=1}^n \frac{\hat{f}_n(X_i)}{f_{\bar{G}_n}(X_i)} \leq
    \max_{j \in J} \prod_{i=1}^n \frac{h_{0j}(X_i) + 2
      v(X_i)}{f_{\bar{G}_n}(X_i)} \prod_{i : X_i \notin S^M}
    \frac{(2\pi)^{-d/2}}{2 v(X_i)}
  \end{equation*}
holds on the event $\hel(\hat{f}_n, f_{\bar{G}_n}) \geq t
\gamma_n$. Because $\hat{f}_n$ satisfies \eqref{bahu}, we obtain  
\begin{align}
\P \left( \hel(\hat{f}_n, f_{\bar{G}_n}) \geq t \gamma_n \right)  &\leq
                                                                   \P
                                                                   \bigg\{\max_{j
                                                                   \in
                                                                   J}
                                                                   \prod_{i=1}^n
                                                                   \frac{h_{0j}(X_i)
                                                                   + 2
                                                                   v(X_i)}{f_{\bar{G}_n}(X_i)}
                                                                   \prod_{i
                                                                   :
                                                                   X_i
                                                                   \notin
                                                                   S^M}
                                                                   \frac{(2
                                                                   \pi)^{-d/2}}{2v(X_i)}
  \nonumber \\
                                                                   &\geq
                                                                   \exp
                                                                   \left((\beta
                                                                   -
                                                                   \alpha)
                                                                   n
                                                                   t^2 \gamma_n^2
                                                                    \right)
                                                                   \bigg\}
  \nonumber \\
	&\leq \P \bigg\{\max_{j \in J} \prod_{i=1}^n \frac{h_{0j}(X_i)
          + 2 v(X_i)}{f_{\bar{G}_n}(X_i)} \geq e^{-\alpha nt^2\gamma_n^2}
          \bigg\} \label{mur} \\
&+ \P \bigg\{ \prod_{i : X_i \notin S^M} \frac{(2
          \pi)^{-d/2}}{2v(X_i)} \geq e^{\beta nt^2\gamma_n^2}
      \bigg\}. \nonumber  
\end{align}
We shall bound the two probabilities above separately. For the first
probability:  
\begin{equation*}
  P_{\RN{1}} := \P \left\{\max_{j \in J} \prod_{i=1}^n \frac{h_{0j}(X_i) + 2
  v(X_i)}{f_{\bar{G}_n}(X_i)} \geq e^{-\alpha nt^2\gamma_n^2}
\right\}, 
\end{equation*}
we write
\begin{align*}
  P_{\RN{1}}
  &\leq \sum_{j \in J} \P \left\{\prod_{i=1}^n 
	\frac{h_{0j}(X_i) + 2 v(X_i)}{f_{\bar{G}_n}(X_i)} \geq
    e^{-\alpha nt^2\gamma_n^2} \right\}   \\ 
	&\leq e^{\alpha nt^2\gamma_n^2/2}\sum_{j \in J} \E \prod_{i=1}^n
  \sqrt{\frac{h_{0j}(X_i) + 2 v(X_i)}{f_{\bar{G}_n}(X_i)}} \\ &=
                                                                e^{\alpha
                                                                nt^2\gamma_n^2/2}
                                                                \sum_{j
                                                                \in J}
                                                                \prod_{i=1}^n 
        \E \sqrt{\frac{h_{0j}(X_i) + 2 v(X_i)}{f_{\bar{G}_n}(X_i)}}. 
\end{align*}
	Now for each fixed $j \in J$, we have 
	\begin{align*}
	\prod_{i=1}^n \E \sqrt{\frac{h_{0j}(X_i) + 2v(X_i)}{f_{\bar{G}_n}(X_i)}} 
	&= \exp \left(\sum_{i=1}^n \log \E \sqrt{\frac{h_{0j}(X_i) + 2v(X_i)}{f_{\bar{G}_n}(X_i)}} \right) \\
	&\leq \exp \left(\sum_{i=1}^n \E \sqrt{\frac{h_{0j}(X_i) + 2v(X_i)}{f_{\bar{G}_n}(X_i)}} - n \right) \\
	&\leq \exp \left(\sum_{i=1}^n \int \sqrt{\frac{h_{0j} +
              2v}{f_{\bar{G}_n}}} f_{G_i}  - n \right) \\
 &= \exp \left(n
          \int\sqrt{\left(h_{0j} + 2v \right) f_{\bar{G}_n}}  - n
        \right).   
	\end{align*}
	Because of $\sqrt{\alpha + \beta} \leq \sqrt{\alpha} +
        \sqrt{\beta}$ and the Cauchy-Schwarz inequality (along with
        $\int f_{\bar{G}_n} = 1$), we obtain
        \begin{align*}
          \int \sqrt{\left(h_{0j} + 2 v \right) f_{\bar{G}_n}} &\le
        \int \sqrt{h_{0j} f_{\bar{G}_n}} + \sqrt{2} \int \sqrt{v
          f_{\bar{G}_n}} \\  
    &\leq \int \sqrt{h_{0j} f_{\bar{G}_n}} + \sqrt{2} \sqrt{\int v}  \\ &=
    1 - \frac{1}{2} \shel(h_{0j}, f_{\bar{G}_n}) + \sqrt{2} \sqrt{\int
      v}. 
        \end{align*} 
  We now use Lemma \ref{Lemma:integralV} which gives an upper bound on
  $\int v$. This (along with the fact that $\hel(h_{0j}, f_{\bar{G}_n})
  \geq t \gamma_n$) allows us to deduce: 
  \begin{equation*}
          \int \sqrt{\left(h_{0j} + 2 v(X_i) \right) f_{\bar{G}_n}}
          \leq 1 - \frac{t^2}{2} \gamma_n^2 + C_d \sqrt{2 \eta
            \text{Vol} (S^M)}. 
  \end{equation*}
 We have therefore proved that 
 \begin{align}
P_{\RN{1}} &\leq \exp \left(\frac{\alpha}{2} n t^2 \gamma^2_n +
              \log |J| - \frac{1}{2} n t^2 \gamma_n^2 + n C_d
              \sqrt{\eta \text{Vol}(S^M)} \right) \nonumber \\
&\leq \exp \left(\frac{\alpha - 1}{2} n t^2 \gamma_n^2 + \log N +
  C_d \sqrt{\text{Vol}(S^M)} \right)    \label{Eq:Prob1}
 \end{align}
because $\eta := n^{-2}$ and $|J| \le N$ (as $J \subseteq \{1, \dots,
N\}$). 

We now use the upper bound on $N$ from inequality \eqref{mm.eq_f} in
Theorem \ref{mm}. Because $\eta  = 1/n^2$ and $n \geq 2$, the quantity
$a$ appearing in Theorem \ref{mm} satisfies 
\begin{equation*}
  a = \sqrt{2 \log \frac{2 \sqrt{2 \pi}}{(2 \pi)^{d/2} \eta}} =
  \sqrt{2 \log \frac{2 \sqrt{2 \pi}}{(2 \pi)^{d/2}} + 4 \log n} \leq
  \sqrt{6 \log n}. 
\end{equation*}
Also because of \eqref{ncon}, we have $2n \geq (2 \pi)^{(d-1)/2}$ so
that 
\begin{align*}
  a = \sqrt{2 \log \frac{2 \sqrt{2 \pi}}{(2 \pi)^{d/2} \eta}} &=
  \sqrt{2 \log \frac{2 \sqrt{2 \pi}}{(2 \pi)^{d/2}} + 4 \log n}  \\ &\geq
  \sqrt{2 \log (1/n) + 4 \log n} = \sqrt{2 \log n}. 
\end{align*}
Thus Theorem \ref{mm} gives  
\begin{equation*}
  \log N \leq C_d N(a, (S^M)^a) (\log n)^{d+1} \leq C_d N(\sqrt{2 \log n},
  S^{M + \sqrt{6 \log n}}) (\log n)^{d+1}. 
\end{equation*}
Using Lemma \ref{volm} to bound the Euclidean covering number
appearing in the right hand side above, we deduce that  
\begin{align*}
  N(\sqrt{2 \log n}, S^{M + \sqrt{6 \log n}}) &\leq C_d (\sqrt{2 \log
    n})^{-d} \text{Vol} (S^{M + \sqrt{6 \log n} + \sqrt{2 \log n}/2})
  \\
&\leq C_d (\log n)^{-d/2} \text{Vol}(S^{M + \sqrt{10 \log n}}) \\ &\leq
  C_d(\log n)^{-d/2} \text{Vol}(S^{2 M}) 
\end{align*}
as $M \geq \sqrt{10 \log n}$. This gives $\log N 
\leq C_d (\log n)^{(d/2) + 1} \text{Vol}(S^{2M})$. Using this bound
for $\log N$ in \eqref{Eq:Prob1}, we obtain
\begin{align}\label{feq1}
P_{\RN{1}} &\leq \exp \bigg(\frac{\alpha - 1}{2} n t^2 \gamma_n^2 
  + C_d (\log n)^{(d/2) + 1} \text{Vol}(S^{2 M}) +
  C_d \sqrt{\text{Vol}(S^M)} \bigg) \nonumber.
\end{align}
We shall now bound the second probability in \eqref{mur}: 
\begin{equation*}
  P_{\RN{2}} :=   \P \left\{\prod_{i : X_i \notin S^M} \frac{(2
      \pi)^{-d/2}}{2 v(X_i)} \geq e^{\beta n t^2 \gamma_n^2}\right\}. 
\end{equation*}
First observe, by Markov's inequality, that 
\begin{equation*}
  P_{\RN{2}} \leq \exp \left(-
  \frac{\beta n t^2 \gamma_n^2}{2 \log n} \right) \E \left(\prod_{i :
    X_i \notin S^M} \frac{(2 \pi)^{-d/2}}{2 v(X_i)} \right)^{1/(2 \log
  n)} 
\end{equation*}
The expectation above is bounded as (recall the formula for
$v(\cdot)$ from \eqref{vdef})
\begin{align*}
  \E \left(\prod_{i :
    X_i \notin S^M} \frac{(2 \pi)^{-d/2}}{2 v(X_i)} \right)^{\frac{1}{2 \log
  n}}  &\leq \E \left(\prod_{i :
    X_i \notin S^M} \frac{1}{v(X_i)} \right)^{\frac{1}{2 \log n}}  \\
&= \E \left(\prod_{i : X_i \notin S^M} \frac{\mathfrak{d}_S(X_i)}{M
  \eta^{\frac{1}{d+1}}} \right)^{\frac{d+1}{2 \log n}}  \\
&= \E \left[\prod_{i=1}^n \left(\frac{\mathfrak{d}_S(X_i)}{M \eta^{\frac{1}{d+1}}}
  \right)^{I\{\mathfrak{d}_S(X_i) \geq M\}}  \right]^{\frac{d+1}{2 \log n}}  
\end{align*}
The above term will be controlled below by using inequality
\eqref{Eq:tailExpect} (which is a consequence of Lemma \ref{tailmom})
with   
\begin{equation}\label{simch}
  a := \frac{1}{M \eta^{1/(d+1)}} ~~~~~ \text{ and } ~~~~~ \lambda := \frac{d+1}{2 \log n} 
\end{equation}
to obtain 
\begin{align}\label{sim1}
  \E \left[\prod_{i=1}^n \left(\frac{\mathfrak{d}_S(X_i)}{M \eta^{\frac{1}{d+1}}}
  \right)^{I\{\mathfrak{d}_S(X_i) \geq M\}}  \right]^{\frac{d+1}{2 \log n}}  &\leq
\exp \bigg\{C_d a^{\lambda} M^{\lambda
    + d - 2} \\
&+ (aM)^{\lambda} n
  \left(\frac{2\mu_p(\mathfrak{d}_S, \bar{G}_n)}{M}\right)^p \bigg\}. \nonumber
\end{align}
We need to assume here that 
\begin{equation*}
  \log n \geq \frac{d+1}{2 \min(1, p)} 
\end{equation*}
to ensure that $\lambda \leq \min(1, p)$ as required for inequality
\eqref{Eq:tailExpect}. This is satisfied as long as $p \geq (d+1)/(2
\log n)$ because under the assumption \eqref{ncon}, we have $\log n
\geq \frac{d+1}{2}$. Thus \eqref{sim1} holds for all $p \geq (d+1)/(2
\log n)$. 

For notational convenience, we write $\mu_p :=
\mu_p(\mathfrak{d}_S, \bar{G}_n)$ in the rest of the proof. With the choices
\eqref{simch} (and $\eta = 1/n^2$), the first term in the exponent of
the right hand side of \eqref{sim1} is calculated as   
\begin{equation*}
  a^{\lambda} M^{\lambda + d - 2} = M^{d-2} \eta^{-\lambda/(d+1)} =
  M^{d-2} n^{1/(\log n)} = e M^{d-2}. 
\end{equation*}
On the other hand, the second term in the exponent in \eqref{sim1}
becomes 
\begin{equation*}
  (aM)^{\lambda} n \left(\frac{2\mu_p(\mathfrak{d}_S, \bar{G}_n)}{M}\right)^p = e n
  \left(\frac{2 \mu_p}{M} \right)^{p}. 
\end{equation*}
Therefore the second probability in \eqref{mur} satisfies the
inequality: 
\begin{equation*}
   P_{\RN{2}} \leq \exp \left(-
  \frac{\beta n t^2 \gamma_n^2}{2 \log n} + C_d M^{d-2} + en
  \left(\frac{2 \mu_p}{M} \right)^{p} \right). 
\end{equation*}
This is true for all $p \geq (d+1)/(2 \log n)$ so we can also write
\begin{equation*}
   P_{\RN{2}} \leq \exp \left(-
  \frac{\beta n t^2 \gamma_n^2}{2 \log n} + C_d M^{d-2} + en
 \inf_{p \geq (d+1)/(2 \log n)} \left(\frac{2 \mu_p}{M} \right)^{p}
\right).  
\end{equation*}
We have proved therefore that for every $t > 0$ 
\begin{align*}
  \P \left\{\hel(\hat{f}_n, f_{\bar{G}_n}) \geq t \gamma_n \right\}
  &\leq \exp \left(\frac{\alpha - 1}{2} n t^2 \gamma_n^2 +
     C_d (\log n)^{(d/2) + 1} \text{Vol}(S^{2 M}) \right. \\ & \left. +
  C_d \sqrt{\text{Vol}(S^M)} \right) + \exp \left(-
  \frac{\beta n t^2 \gamma_n^2}{2 \log n} + C_d M^{d-2} \right. \\ &\left. + en
  \inf_{p \geq (d+1)/(2 \log n)} \left(\frac{2 \mu_p}{M} \right)^{p} \right). 
\end{align*}
We now note that  $\vol(S^{M}) \leq \vol(S^{2M}) \leq C_d M^d
\vol(S^1)$ which follows from inequality \eqref{sms1} in Lemma
\ref{volm}. This, along with the definition of $\epsilon^2_n(M, S,
\bar{G}_n)$ in \eqref{si.ra}, gives that
\begin{equation*}
  \max \left((\log n)^{\frac{d}{2} + 1} \vol(S^{2M}), \sqrt{\vol(S^M)},
    M^d, n (\log n)
  \inf_{p \geq \frac{d+1}{2 \log n}} \left(\frac{2 \mu_p}{M} \right)^{p}
\right) 
\end{equation*}
is bounded from above by $C_d n  \epsilon_n^2(M, S, \bar{G}_n)$. 
As a result, 
\begin{align*}
  \P \left\{\hel(\hat{f}_n, f_{\bar{G}_n}) \geq t \gamma_n \right\} 
  &\leq \exp \left(\frac{\alpha - 1}{2} n t^2 \gamma_n^2 +
     C_d n \epsilon_n^2(M, S, \bar{G}_n) \right) \\ &+ \exp \left(-
  \frac{\beta n t^2 \gamma_n^2}{2 \log n} + C_d \frac{n \epsilon_n^2(M, S,
                                                      \bar{G}_n)}{\log n}
                                                      \right).  
\end{align*}
Now suppose that 
\begin{equation}\label{gmde}
  \gamma_n^2 = C_d' \frac{\epsilon^2_n(M, S, \bar{G}_n)}{\min(1-\alpha,
    \beta)} 
\end{equation}
for some $C_d' \geq 4 C_d$. We deduce then that, for every $t \geq 1$, 
\begin{align}
  \P \left\{\hel(\hat{f}_n, f_{\bar{G}_n}) \geq t \gamma_n \right\}
  &\leq \exp \left(- \frac{1-\alpha}{2} n t^2 \gamma_n^2 +
    \frac{1-\alpha}{4} n \gamma_n^2 \right)  \nonumber \\ &+ \exp \left(-
    \frac{\beta}{2 \log n} n t^2 \gamma_n^2 + \frac{\beta}{4 \log n} n
    \gamma_n^2 \right) \nonumber  \\
&\leq 2 \exp \left(- \frac{\min((1-\alpha), \beta)}{4 \log n} n t^2
   \gamma_n^2 \right) \label{kpp}
\end{align}
Observe now that (because $M \geq \sqrt{10 \log n}$) 
\begin{equation*}
  \epsilon_n^2(M, S, \bar{G}_n) \geq \vol(S^1) \frac{M^d}{n} \left(\sqrt{\log n} \right)^{d+2} \geq \vol(B(0, 1)) \frac{(\log n)^2}{n}
\end{equation*}
so that we can choose the constant $C_d'$ such that 
\begin{equation*}
n \min(1 - \alpha, \beta) \gamma_n^2 \geq C_d' n \epsilon_n^2(M, S, \bar{G}_n)
\geq 4 (\log n)^2. 
\end{equation*}
This gives, via \eqref{kpp}, 
\begin{equation*}
  \P \left\{\hel(\hat{f}_n, f_{\bar{G}_n}) \geq t \gamma_n \right\} 
  \leq 2 n^{-t^2}. 
\end{equation*}
We have therefore proved the above inequality for $\gamma_n$ as chosen
in \eqref{gmde} (provided $C_d'$ is chosen sufficiently large) for
every estimator $\hat{f}_n$ satisfying \eqref{bahu}. This completes
the proof of \eqref{dens.eq}. 

For \eqref{dens.1}, we multiply both sides of \eqref{dens.eq} by $t$ 
and then integrate from $t = 1$ to $t = \infty$ to obtain
\begin{equation*}
  \E \bigg(\frac{\shel(\hat{f}_n, f_{\bar{G}_n}) \min(1 - \alpha,
      \beta)}{C_d \epsilon_n^2(M, S, \bar{G}_n)} \bigg) \leq 1 + 4
  \int_1^{\infty} t n^{-t^2} \leq 1 + \frac{2}{n \log n} \leq 4
\end{equation*}
which proves \eqref{dens.1} and completes the proof of Theorem
\ref{dens}.  
\end{proof}

\subsection{Proof of Corollary \ref{dza}} 
\begin{proof}[Proof of Corollary \ref{dza}]
  To prove \eqref{dza1.eq}, assume that $\bar{G}_n$ is supported on a
  compact set $S$. We then apply Theorem \ref{dens} to this $S$ and $M
  = \sqrt{10 \log n}$. Because $\bar{G}_n$ is supported on $S$, we
  have $\mu_p(\dos, \bar{G}_n) = 0$ for every $p > 0$ so that
  $\epsilon_n^2(M, S, \bar{G}_n)$ (defined in \eqref{si.ra}) becomes 
  \begin{equation*}
    \epsilon_n^2(M, S, \bar{G}_n) = \vol(S^1) \frac{M^d}{n} \left(\sqrt{\log n}
    \right)^{d+2} = (10)^{d/2}\frac{\vol(S^1)}{n}
  \left(\log n \right)^{d + 1}. 
  \end{equation*}
  Inequality \eqref{dza1.eq} then immediately follows from Theorem
  \ref{dens}. 

  We next prove \eqref{dza2.eq}  assuming the condition
  \eqref{mmc}. Let
  \begin{equation}\label{mex}
    M := 4 K (e \log n)^{1/\alpha}. 
  \end{equation}
   This quantity $M \geq \sqrt{10 \log n}$ because $K \geq 1$ and
   $\alpha \leq 2$. We shall apply \eqref{dens.1} with this $M$. Let  
\begin{equation*}
  T_2(M, S, \bar{G}_n) := (\log n) \inf_{p \geq \frac{d+1}{2 \log n}} \bigg(\frac{
   2 \mu_p(\dos, \bar{G}_n)}{M} \bigg)^p
\end{equation*}
and note that this is the second term on the right hand side of
\eqref{si.ra} in the definition of $\epsilon_n^2(M, S,
\bar{G}_n)$. The infimum over $p$ above is 
easily seen to be achieved at $p = (M/(2K))^{\alpha} (1/e)$. By the
expression \eqref{mex} for $M$, it is easy to see that $p \geq 
(d+1)/(2 \log n)$ provided 
\begin{equation}\label{nalp}
  n \geq \exp \left(\sqrt{(d+1)/2} \right). 
\end{equation}
We then deduce that 
\begin{equation*}
  T_2(M, S, \bar{G}_n) \leq (\log n) \exp \left(\frac{-1}{\alpha e}
    \left(\frac{M}{2 K} \right)^{\alpha} \right).  
\end{equation*}
It follows from here that $T_2(M, S, \bar{G}_n) \leq (\log n)/n$ because $M \geq
(4 K) (e \log n)^{1/\alpha} \geq (2 K ) (\alpha e \log
n)^{1/\alpha}$. Thus 
\begin{align}
  \epsilon^2_n(M, S, \bar{G}_n) &= \vol(S^1) \frac{M^d}{n} \left(\sqrt{\log n}
                       \right)^{d+2} + T_2(M, S, \bar{G}_n) \nonumber \\
&\leq \vol(S^1) \frac{(4 K e^{1/\alpha})^d}{n} (\log n)^{d/\alpha}
  \left(\sqrt{\log n} \right)^{d+2} + \frac{\log n}{n} \label{tobeusl}
\end{align}
and hence \eqref{dza2.eq} readily follows as a consequence of Theorem
\ref{dens}. When the assumption \eqref{nalp} does not hold, inequality
\eqref{dza2.eq} becomes trivially true when $C_d$ is chosen
sufficiently large. 

We now turn to \eqref{dza3.eq}. Assume that $S$ is such that
$\mu_p(\dos, \bar{G}_n) \leq \mu$ for fixed $\mu > 0$ and $p > 0$. Then Theorem
\ref{dens} gives
\begin{align*}
  \E \shel(\hat{f}_n, f_{\bar{G}_n}) &\leq C_d \inf_{M \geq \sqrt{10
  \log n}} \epsilon_n^2(M, S, \bar{G}_n) \\
&= C_d \inf_{M \geq \sqrt{10 \log n}} \left( \mathrm{Vol}(S^1) \frac{M^d}{n} \left(\sqrt{\log n}
  \right)^{d+2}  \right. \\ & \left.+  \left(\log n \right) \left(\frac{2 \mu}{M}
  \right)^{p}\right)
\end{align*}
where we assumed that $n$ is large enough so that $p \geq (d+1)/(2
\log n)$. Taking 
\begin{equation*}
  M = \left(\sqrt{\log n} \right)^{-d/(p+d)} \left(\frac{n
      \mu^p}{\vol(S^1)} \right)^{1/(p+d)} 
\end{equation*}
results in \eqref{dza3.eq}. When $n$ is large enough, $M$ chosen as
above exceeds $\sqrt{10 \log n}$. For smaller $n$, the inequality
\eqref{dza3.eq} trivially holds provided $C_{d, \mu, p}$ is chosen
large enough. 
\end{proof}

\subsection{Proof of Lemma \ref{trilo1}} 
The following uses standard ideas involving Assouad's lemma (see,
for example, \citet[Chapter 2]{Tsybakovbook}.   

\begin{proof}[Proof of Lemma
  \ref{trilo1}] 
Fix $\delta > 0$ and $M > 0$. Let $a_1, \dots, a_k$ and $b_1,\dots,
b_k$ be points in $\R^d$ such that 
\begin{equation}\label{cnd1}
  \min \bigg(\min_{i \neq j}\norm{a_i - a_j},  \min_{i \neq
      j}\norm{b_i - b_j}, \min_{i \neq j}\norm{a_i - b_j} \bigg) \geq
  M
\end{equation}
and such that
\begin{equation}\label{cnd2}
  \norm{a_i - b_i} = \delta \qt{for every $1 \leq i \leq k$}. 
\end{equation}
Now for every $\tau \in \{0, 1\}^k$, let 
\begin{equation*}
  f_{\tau}(x) = \frac{1}{k} \sum_{i=1}^k \phi_d(x - a_i (1 - \tau_i) 
  - b_i \tau_i)
\end{equation*}
where $\phi_d(\cdot)$ is the standard normal density on
$\R^d$. Clearly $f_{\tau} \in \M_k$ for every $\tau \in \{0,
1\}^k$. We shall now employ Assouad's lemma which gives
\begin{equation*}
  \Rs(\M_k) \geq \frac{k}{8} \min_{\tau \neq \tau'}
  \frac{\hel^2(f_{\tau}, f_{\tau'})}{\Upsilon(\tau, \tau')}
  \min_{\Upsilon(\tau, \tau') = 1} \bigg(1 - \norm{P_{f_{\tau}} -
      P_{f_{\tau'}}}_{TV} \bigg)
\end{equation*}
where $\Upsilon(\tau, \tau') := \sum_{i=1}^k I\{\tau_j \ne \tau_j'\}$
denotes Hamming distance and $P_{f}$ (for $f \in \M$) denotes the
joint distribution of $X_1, \dots, X_n$ which are independently
distributed according to $f$. 

We now fix $\tau \neq \tau' \in \{0, 1\}^k$ and bound
$\hel^2(f_{\tau}, f_{\tau'})$ from below. For simplicity, let $f =
f_{\tau}$ and $g = f_{\tau'}$. Also, for $i = 1, \dots, k$, let 
\begin{align*}
  f_i(x) := \phi_d(x - a_i (1 - \tau_i) - b_i \tau_i) ~~
  \text{ and } ~~  g_i(x) := \phi_d(x - a_i (1 - \tau_i') - b_i
\tau_i')
\end{align*}
so that $f = \sum_{i=1}^k f_i/k$ and $g = \sum_{i=1}^k g_i/k$. This
gives 
\begin{align*}
 \frac{1}{2} \hel^2(f, g) &=  1 - \int  \sqrt{f(x) g(x)} dx \\
&= 1 - \int
                            \sqrt{\frac{1}{k^2} \sum_{i, j} f_i(x)
                            g_j(x)} dx \geq 1 - \frac{1}{k} \sum_{i,
                            j} \int \sqrt{f_i(x) g_j(x)} dx .
\end{align*}
Because $f_i$ and $g_j$ are normal densities, by a straightforward
computation, we obtain
\begin{equation*}
  \int \sqrt{f_i(x) g_j(x)} dx = \exp \left(- \norm{a_i (1
      - \tau_i) + b_i \tau_i - a_j(1 - \tau_j')  - b_j \tau_j'}^2/8
  \right) 
\end{equation*}
so that by \eqref{cnd1} and \eqref{cnd2}, we obtain that 
\begin{equation*}
  \int \sqrt{f_i(x) g_j(x)} dx = I\{\tau_i = \tau_i'\} + I\{\tau_i
  \neq \tau_i'\} e^{-\delta^2/8} \qt{for $i = j$}
\end{equation*}
and
\begin{equation*}
  \int \sqrt{f_i(x) g_j(x)} dx \leq e^{-M^2/8} \qt{for $i \neq j$}. 
\end{equation*}
As a result, we obtain
\begin{align}
  \frac{1}{2} \hel^2(f_{\tau}, f_{\tau'}) &= 1 - \frac{1}{k}
                                            \sum_{i=1}^k \int
                                            \sqrt{f_i(x) g_i(x)} dx -
                                            \frac{1}{k} \sum_{i \neq
                                            j} \int \sqrt{f_i(x)
                                            g_j(x)} dx  \nonumber \\
&\geq 1 - \frac{1}{k} \sum_{i=1}^k I\{\tau_i = \tau_i'\} -
  \frac{e^{-\delta^2/8}}{k} \Upsilon(\tau, \tau') - \frac{k^2 - k}{k}
  e^{-M^2/8} \nonumber \\
&= \frac{1}{k} \Upsilon(\tau, \tau') \left( 1  - e^{-\delta^2/8}
  \right) - (k-1) e^{-M^2/8} \label{fp1}
\end{align}
for every $\tau \neq \tau' \in \{0, 1\}^k$. Now let us fix $\tau,
\tau'$ with $\Upsilon(\tau, \tau') = 1$ and bound from above the total
variation distance between $P_{f_{\tau}}$ and $P_{f_{\tau'}}$. Without
loss of generality, we can assume that $\tau_1 \neq \tau_1'$ and that
$\tau_i = \tau_i'$ for $i \geq 2$. Below $D(P_{f_{\tau}} ||
P_{f_{\tau'}})$ denotes the Kullback-Leibler divergence between
$P_{f_{\tau}}$ and $P_{f_{\tau'}}$. Also $D(f_{\tau} || f_{\tau'})$
and $\chi^2(f_{\tau}, f_{\tau'})$ denote the Kullback-Leibler
divergence and chi-squared divergence between the densities $f_{\tau}$
and $f_{\tau'}$ respectively. By Pinsker's inequality and the fact
that $D(f_{\tau} || f_{\tau'}) \leq \chi^2(f_{\tau}, f_{\tau'})$, we
obtain
\begin{align*}
  \norm{P_{f_{\tau}} - P_{f_{\tau'}}}_{TV} \leq \sqrt{\frac{1}{2}                                     D(P_{f_{\tau}} ||
                                             P_{f_{\tau'}})} = \sqrt{\frac{n}{2} D({f_{\tau}} ||
                                             {f_{\tau'}})} \leq \sqrt{\frac{n}{2} \chi^2({f_{\tau}} ||
                                             {f_{\tau'}})}. 
\end{align*}
Further
\begin{align*}
  \chi^2({f_{\tau}} || {f_{\tau'}}) &= \int \frac{(f_{\tau}(x) -
                                       f_{\tau'}(x))^2}{f_{\tau'}(x)}
                                       dx \\
&= \int
  \frac{\left(\phi_d(x - a_1(1 - \tau_1) - b_1 \tau_1)  - \phi_d(x -
  a_1(1 - \tau_1') - b_1 \tau_1')\right)^2}{k^2 f_{\tau'}(x)} dx \\
&\leq  \int
  \frac{\left(\phi_d(x - a_1(1 - \tau_1) - b_1 \tau_1)  - \phi_d(x -
  a_1(1 - \tau_1') - b_1 \tau_1')\right)^2}{k \phi_d(x - a_1(1 -
  \tau_1') - b_1 \tau_1')} dx. 
\end{align*}
By a routine calculation, it now follows that 
\begin{align*}
  \chi^2(f_{\tau} || f_{\tau'}) &\leq \frac{1}{k} \left\{ \exp \left(\norm{a_1
  (1 - \tau_1) + b_1 \tau_1 - a_1(1 - \tau_1') - b_1 \tau_1'}^2
  \right) - 1 \right\} \\
&= \frac{1}{k} \left\{ \exp \left(\norm{a_1 - b_1}^2 \right) - 1
  \right\} = \frac{1}{k}  \left(e^{\delta^2} - 1 \right). 
\end{align*}
We have therefore proved that 
\begin{align}\label{fp2}
  \norm{P_{f_{\tau}} - P_{f_{\tau'}}}_{TV} \leq \sqrt{\frac{n}{2k}
  \left(e^{\delta^2} - 1 \right)} 
\end{align}
for every $\tau, \tau' \in \{0, 1\}^k$ with $\Upsilon(\tau, \tau') =
1$. Combining \eqref{fp1} and \eqref{fp2}, we obtain 
\begin{align*}
  \Rs(\M_k) \geq \frac{k}{4} \left(\frac{1}{k} \left(1 -
  e^{-\delta^2/8} \right) - \frac{(k-1)}{\Upsilon(\tau, \tau')}
  e^{-M^2/8} \right) \left(1 - \sqrt{\frac{n}{2k}
  \left(e^{\delta^2} - 1 \right)} \right). 
\end{align*}
This inequality holds for every $\delta > 0$ and $M > 0$. So we can
let $M$ tend to $\infty$ to deduce 
\begin{align*}
  \Rs(\M_k) \geq \frac{1}{4} \left(1 - e^{-\delta^2/8} \right) \left(1 - \sqrt{\frac{n}{2k}
  \left(e^{\delta^2} - 1 \right)} \right)
\end{align*}
for every $\delta > 0$. The inequalities $1 - e^{-t} \geq t/2$ and
$e^t - 1 \leq 2t$ for $0 \leq t \leq 1$ imply that  
\begin{align*}
  \Rs(\M_k) \geq \frac{\delta^2}{64} \left(1 - \sqrt{\frac{n}{k}}
  \delta \right) \qt{for every $0 \leq \delta \le 1$}. 
\end{align*}
The choice $\delta = \sqrt{k/4n}$ now proves \eqref{trilo1.eq}. 
\end{proof}

\subsection{Proof of Theorem \ref{amix}} 

\begin{proof}[Proof of Theorem \ref{amix}]
Note that 
\begin{equation*}
  h^*(x) = \sum_{j=1}^k w_j \phi_d(x; \mu_j, \Sigma_j) = \sum_{j=1}^k
  w_j \det(\Sigma_j^{-1/2}) \phi_d \left(\Sigma_j^{-1/2} (x - \mu_j)
  \right) 
\end{equation*}
where $\phi_d(z) := (2 \pi)^{-d/2} \exp \left(-\norm{z}^2/2\right)$
denotes the standard $d$-dimensional normal density. It is then easy
to see that $X_1, \dots, X_n$ (where $X_i = Y_i/\sigma_{\min}$) are
independent observations having the density $f^*$ where 
\begin{align*}
  f^*(x) &= \sigma_{\min}^d h^*(\sigma_{\min} x) \\
&= \sum_{j=1}^k w_j 
  \left[\det \left(\sigma_{\min}^{-2}\Sigma_j \right)^{-1/2}
  \right] \phi_d \left(\left\{\sigma^{-2}_{\min}\Sigma_j
    \right\}^{-1} \left(x - \sigma^{-1}_{\min}\mu_j \right)
  \right). 
\end{align*}
This means that $f^*$ is the density of the normal mixture: 
\begin{equation*}
  \sum_{j=1}^k w_j N \left(\sigma_{\min}^{-1} \mu_j,
    \sigma_{\min}^{-2} \Sigma_j \right) 
\end{equation*}
where $N(\mu, \Sigma)$ denotes the multivariate normal distribution
with mean vector $\mu$ and covariance matrix $\Sigma$. It follows from
here that $f^*$ equals $f_{G^*}$ (in the notation \eqref{glmd}) where
$G^*$ is the distribution of the normal mixture
\begin{equation*}
  \sum_{j=1}^k w_j N \left(\sigma_{\min}^{-1} \mu_j,
    \sigma_{\min}^{-2} \Sigma_j  - I_d\right) 
\end{equation*}
where $I_d$ is the $d \times d$ identity matrix. 

We can now use Corollary \ref{dza} to bound $\shel(\hat{f}_n, f^*)$
(note that $\hat{f}_n$ is an NPMLE based on $X_1, \dots,
X_n$). Specifically we shall use inequality \eqref{dza2.eq} with 
\begin{equation*}
  S := \left\{\sigma_{\min}^{-1} \mu_1, \dots, \sigma_{\min}^{-1}
    \mu_k \right\}. 
\end{equation*}
In order to verify \eqref{mmc}, observe first that $\bar{G}_n$ in
Corollary \ref{dza} is $G^*$ since $X_1, \dots, X_n$ are i.i.d
$f_{G^*}$ and that 
\begin{equation*}
  \dos(\theta) = \min_{1 \leq i \leq k} \norm{\sigma_{\min}^{-1} \mu_i
  - \theta} 
\end{equation*}
As a result, for every $p \geq 1$ and $Z \sim N(0, I_d)$, we have
\begin{align*}
  \mu_p(\dos, \bar{G}_n) &\leq \left(\E \max_{1 \leq j \leq k}
    \norm{\left(\sigma_{\min}^{-2} \Sigma_j - I_d \right)^{1/2} Z }^p
  \right)^{1/p} \\
&\leq \sqrt{\frac{\sigma_{\max}^2}{\sigma_{\min}^2} - 
    1} \left( \E \norm{Z}^p \right)^{1/p} \leq C_d \tau \sqrt{p}. 
\end{align*}
Thus \eqref{mmc} holds with $K := C_d \max(1, \tau)$ and $\alpha = 2$
and inequality \eqref{dza2.eq} then gives
\begin{equation*}
  \E \shel(\hat{f}_n, f^*) \leq C_d \frac{\vol(S^1)}{n} \left(\max(1,
    \tau) \right)^d  (\log n)^{d+1}
\end{equation*}
As $S$ is a finite set of cardinality $k$, we have $\vol(S^1)
\leq k C_d$ so that 
\begin{equation*}
  \E \shel(\hat{f}_n, f^*) \leq C_d \left(\frac{k}{n} \right) \left(\max(1,
    \tau) \right)^d (\log n)^{d+1}. 
\end{equation*}
We now use the fact that the Hellinger distance is invariant under
scale transformations which implies that $\hel(\hat{f}_n, f^*) =
\hel(\hat{h}_n, h^*)$. This proves inequality \eqref{amix.eq}.
\end{proof}

\section{Proofs of Results in Section \ref{gadeno}}\label{gadeno.pf}

\subsection{Proof of Theorem \ref{theorem:denoising_theorem}}
Theorem \ref{theorem:denoising_theorem} is a special of Theorem
\ref{rgende} (indeed, taking $\Sigma_i = I_d$ for each $i$ in Theorem
\ref{rgende} leads to Theorem
\ref{theorem:denoising_theorem}). Therefore, the proof of Theorem
\ref{theorem:denoising_theorem}  follows from the proof of Theorem
\ref{rgende} which is given in Subsection \ref{pfsec.rgende}. 

\subsection{Proof of Corollary \ref{dna}} 
The idea is to choose $M$ and $S$ appropriately under each of the
assumptions on $\bar{G}_n$ and then to appropriately bound
$\epsilon_n(M, S)$. The necessary work for this is already done in
Corollary \ref{dza} from which Corollary \ref{dna} immediately
follows. 

\subsection{Proof of Proposition \ref{pkde}} 
The assumption \eqref{kga} implies that the empirical measure
$\bar{G}_n$ of $\theta_1, \dots, \theta_n$ is supported on 
\begin{equation*}
  S := \cup_{j=1}^k B(a_j, R) \qt{where $B(a_j, R) := \left\{x \in
      \R^d : \norm{x - a_j} \leq R \right\}$}. 
\end{equation*}
We can therefore apply inequality \eqref{dna1.eq} in Corollary
\ref{dna} to bound $\Rr_n(\hat{\theta}, \hat{\theta}^*)$. The
conclusion \eqref{pkde.eq} then immediately follows from
\eqref{dna1.eq} because
\begin{equation*}
  \vol(S^1) \leq \sum_{j=1}^k \vol(B(a_j, 1+R)) \leq C_d k (1 + R)^d. 
\end{equation*}

\subsection{Proof of Lemma \ref{denlo}}
The proof of Lemma \ref{denlo} uses Assouad's lemma (see,
for example, \citet[Chapter 2]{Tsybakovbook} as well as Lemma
\ref{obg} (stated and proved in Section \ref{auxre}).  
\begin{proof}[Proof of Lemma \ref{denlo}]
  Fix $k$ and $n$ with $1 \leq k \leq n$. Also fix $\delta > 0$ and $M
  \geq 2$. Let $a_1, \dots, a_k$ and 
  $b_1, \dots, b_k$ be points in $\R^d$ such that 
\begin{equation}\label{cd1}
  \min \left(\min_{i \neq j}\norm{a_i - a_j},  \min_{i \neq
      j}\norm{b_i - b_j}, \min_{i \neq j}\norm{a_i - b_j} \right) \geq
  M
\end{equation}
and such that
\begin{equation}\label{cd2}
  \norm{a_i - b_i} = \delta \qt{for every $1 \leq i \leq k$}. 
\end{equation}
We now define a partition $S_1, \dots, S_k, S_{k+1}$ of $\{1, \dots,
n\}$ via  
  \begin{equation*}
    S_i := \left\{(i-1)m + 1, \dots, i m \right\} \qt{for $i = 1,
      \dots, k$}
  \end{equation*}
  and  $S_{k+1} := \{km + 1, \dots, n\}$ where $m := [n/k]$ ( for $x >
  0$, we define $[x]$ as usual to be the largest integer that is
  smaller than or equal to $x$). Note that the cardinality of $S_j$
  equals $m$ for $i = 1, \dots, k$ and that $S_{k+1}$ will be empty if
  $n$ is a multiple of $k$.  

 Now for every $\tau \in \{0, 1\}^k$, we define $n$ vectors
 $\theta_1(\tau), \dots, \theta_n(\tau)$ in $\R^d$ via 
 \begin{equation*}
   \theta_i(\tau) := (1 - \tau_j) a_j + \tau_j b_j \qt{provided $i \in
     S_j$ for some $1 \leq j \leq k$}
 \end{equation*}
and for $i \in S_{k+1}$, we take $\theta_i(\tau) := a_1$. 

Let $\Theta(\tau)$ denote the collection of all $n$-tuples
$(\theta_1(\tau), \dots, \theta_n(\tau))$ as $\tau$ ranges over
$\{0,1\}^k$. It is easy to see that $\Theta(\tau) \subseteq \Theta_{n,
d, k}$ so that 
\begin{align*}
  \Rs^*(\Theta_{n, d, k}) \geq \Rs^*(\Theta(\tau)) := \inf_{\tilde{\theta}_1, \dots,
    \tilde{\theta}_n} \sup_{(\theta_1, \dots, \theta_n) \in
    \Theta(\tau)} \E \left[\frac{1}{n} \sum_{i=1}^n
      \norm{\tilde{\theta}_i - 
    \hat{\theta}_i^*}^2 \right]. 
\end{align*}
The elementary inequality $\norm{a - b}^2 \geq \norm{a}^2/2 -
\norm{b}^2$ for vectors $a, b \in \R^d$ gives 
\begin{equation*}
  \frac{1}{n} \sum_{i=1}^n \norm{\tilde{\theta}_i -
    \hat{\theta}_i^*}^2 \geq \frac{1}{2n} \sum_{i=1}^n
  \norm{\tilde{\theta}_i - \theta_i}^2 - \frac{1}{n} \sum_{i=1}^n
  \norm{\hat{\theta}_i^* - \theta_i}^2
\end{equation*}
for every $\theta_1, \dots, \theta_n$ and estimators $\tilde{\theta}_1
\dots, \tilde{\theta}_n$. As a result, we deduce that
\begin{equation}\label{ctr}
  \Rs^*(\Theta(\tau)) \geq \breve{\Rs}(\Theta(\tau)) -
  \sup_{(\theta_1, \dots, \theta_n) \in \Theta(\tau)} \E \left[ \frac{1}{n} \sum_{i=1}^n
  \norm{\hat{\theta}_i^* - \theta_i}^2 \right] 
\end{equation}
where 
\begin{equation*}
  \breve{\Rs}(\Theta(\tau)) := \inf_{\tilde{\theta}_1, \dots,
    \tilde{\theta}_n} \sup_{(\theta_1, \dots, \theta_n) \in
    \Theta(\tau)} \E \left[\frac{1}{n} \sum_{i=1}^n
      \norm{\tilde{\theta}_i - \theta_i}^2 \right]. 
\end{equation*}
We first bound $\breve{\Rs}(\Theta(\tau))$ from below via Assouad's
lemma. For $\tau, \tau' \in \{0, 1\}^k$, let 
\begin{equation*}
  \los(\tau, \tau') := \frac{1}{n} \sum_{i=1}^n \norm{\theta_i(\tau) -
  \theta_i(\tau')}^2. 
\end{equation*}
Also let $P_{\tau}$ denote the joint distribution of the independent
random variables $X_1, \dots, X_n$ with $X_i \sim N(\theta_i(\tau),
I_d)$ for $i = 1, \dots, n$. Assouad's lemma then gives
\begin{equation}\label{gass}
  \breve{\Rs}(\Theta(\tau)) \geq \frac{k}{8} \min_{\tau \neq \tau'}
  \frac{\los(\tau, \tau')}{\Upsilon(\tau, \tau')} \min_{\Upsilon(\tau,
    \tau') = 1} \left(1 - \norm{P_{\tau} - P_{\tau'}}_{TV} \right)
\end{equation}
where $\Upsilon(\tau, \tau') := \sum_{j=1}^k I\{\tau_j \neq
\tau_j'\}$ is the Hamming distance and $\norm{P_{\tau} -
  P_{\tau'}}_{TV}$ denotes the variation distance between $P_{\tau}$
and $P_{\tau'}$. We now bound the terms appearing in the right hand
side of \eqref{gass}.  For $\tau, \tau' \in \{0, 1\}^k$, observe that
\begin{align}
  \los(\tau, \tau') &= \frac{1}{n} \sum_{j=1}^k \sum_{i : i \in S_j}
  \norm{a_j - b_j}^2 I \{\tau_j \neq \tau'_j \} \nonumber \\
&= \frac{1}{n}
  \sum_{j=1}^k |S_j| \norm{a_j - b_j}^2 I \{\tau_j \neq \tau'_j \}  =
  \frac{m \delta^2}{n} \Upsilon(\tau, \tau') \label{chn1}
\end{align}
where $|S_j|$ denotes the cardinality of $S_j$. We have used above the
fact that $|S_j| = m$ for $1 \leq j \leq k$ and \eqref{cd2}. 

To bound the last term in \eqref{gass}, we use Pinsker's inequality
(below $D$ stands for Kullback-Leibler divergence) to obtain
\begin{equation*}
  \norm{P_{\tau} - P_{\tau'}}_{TV} \leq \sqrt{\frac{1}{2} D(P_{\tau} ||
  P_{\tau'})} =\frac{1}{2}\sqrt{\sum_{i=1}^n \norm{\theta_i(\tau) -
    \theta_i(\tau')}^2} = \frac{1}{2} \sqrt{n \los(\tau, \tau')}. 
\end{equation*}
Thus, from \eqref{chn1}, we deduce that for $\Upsilon(\tau, \tau') =
1$, 
\begin{equation*}
  \norm{P_{\tau} - P_{\tau'}}_{TV} \leq \frac{1}{2} \sqrt{m
    \delta^2}. 
\end{equation*}
Inequality \eqref{gass} thus gives
\begin{equation}\label{gsy}
  \breve{\Rs}(\Theta(\tau)) \geq \frac{k m \delta^2}{8 n} \left(1 -
    \frac{\sqrt{m \delta^2}}{2} \right).  
\end{equation}
To bound the second term in \eqref{ctr}, we use Lemma \ref{obg} which
gives that for every $\theta_1, \dots, \theta_n \in \Theta(\tau)$, we have
\begin{equation*}
 \E \left[ \frac{1}{n} \sum_{i=1}^n
  \norm{\hat{\theta}_i^* - \theta_i}^2 \right]  \leq \frac{k}{2 \sqrt{2 \pi}} \sum_{j, l
    : j \neq l} \left(p_j + p_l \right) \norm{c_j - c_l}  \exp
  \left(-\frac{1}{8} \norm{c_j - c_l}^2 \right)
\end{equation*}
where $c_1, \dots, c_{k+1}$ denote the distinct elements from $\theta_1,
\dots, \theta_n$ and $p_j, j = 1, \dots, k+1$ are nonnegative real
numbers summing to one. Now each $c_j$ equals either $a_j$ or $b_j$ and
hence, by \eqref{cd1}, we have $\norm{c_j - c_l} \geq M$ for every $j
\neq l$. As $x \mapsto x e^{-x^2/8}$ is decreasing for $x > 2$ and $M
> 2$, we deduce that
\begin{equation}\label{sdn}
 \E \left[ \frac{1}{n} \sum_{i=1}^n
  \norm{\hat{\theta}_i^* - \theta_i}^2 \right]  \leq \frac{k}{2
  \sqrt{2 \pi}}  M e^{-M^2/8} \sum_{j, l
    : j \neq l} \left(p_j + p_l \right) \leq \frac{k}{\sqrt{2 \pi}}
  M e^{-M^2/8}. 
\end{equation}
We obtain therefore from \eqref{ctr}, \eqref{gsy} and \eqref{sdn},
that
\begin{equation*}
  \Rs^*(\Theta_{n, d, k}) \geq \frac{k m \delta^2}{8 n} \left(1 -
    \frac{\sqrt{m \delta^2}}{2} \right) - \frac{k}{\sqrt{2 \pi}}
  M e^{-M^2/8}. 
\end{equation*}
The left hand side above does not depend on $M$ so we can let $M
\rightarrow \infty$ to obtain
\begin{equation*}
  \Rs^*(\Theta_{n, d, k}) \geq \frac{k m \delta^2}{8 n} \left(1 -
    \frac{\sqrt{m \delta^2}}{2} \right). 
\end{equation*}
We now make the choice $\delta := 1/\sqrt{m}$ to obtain
$\Rs^*(\Theta_{n, d, k}) \geq k/(16 n)$ which proves Lemma
\ref{denlo}.   
\end{proof}

\subsection{Proof of Theorem \ref{rgende}}\label{pfsec.rgende} 
The proof of Theorem \ref{rgende} is similar to
\citet[Proof of Theorem 5]{jiang2009general}. It uses ingredients that
are proved in Section \ref{mmps}, Section \ref{zhafo} and Section
\ref{auxre}. More precisely, crucial roles are played by the metric
entropy results of Section \ref{mmps} (specifically Corollary
\ref{rgnt}) and Theorem \ref{fcc} in Section \ref{zhafo} which relates
the denoising error to Hellinger distance (thereby allowing the
application of Theorem \ref{dens}). Additionally, Lemma \ref{tailmom}
and Lemma \ref{dco} from Section \ref{proids} as well as Lemma
\ref{p1}, Lemma \ref{devs} and Lemma \ref{volm} from Section
\ref{auxre} will also be used.   

Basically, the following proof bounds $\Rr_n(\hat{\theta},
\hat{\theta}^*)$ in terms of five quantities $\zeta_{1n}^2, \dots,
\zeta_{5n}^2$. The additional $(\log n)^{\max(d, 3)}$ factor in
Theorem \ref{rgende} (compared to Theorem \ref{dens}) comes from the
bounds used for the terms involving $\zeta_{4n}^2$ and
$\zeta_{5n}^2$. 

The notation described at the beginning of Section \ref{hela.pf} will
be followed in this section as well.  
\begin{proof}[Proof of Theorem \ref{rgende}]
  The goal is to bound 
\begin{align*}
  \Rr_n(\hat{\theta}, \breve{\theta}^*) &= \E \left(\frac{1}{n} \sum_{i=1}^n
  \|\hat{\theta}_i - \breve{\theta}_i^* \|^2 \right) \\
&=  \E \left(\frac{1}{n} \sum_{i=1}^n \norm{X_i + \frac{\nabla 
    \hat{f}_n(X_i)}{\hat{f}_n(X_i)} - X_i -  \frac{\nabla
   f_{\bar{G}^0_n}(X_i)}{f_{\bar{G}^0_n}(X_i)} }^2 \right)  
\end{align*} 
Let us now introduce the following notation. Let $\bx$ denote the 
$d \times n$ matrix whose columns are the observed data vectors $X_1,
\dots, X_n$. For a density $f \in \M$, let $T_f(\bx)$ denote the $d
\times n$ matrix whose $i^{th}$ column is given by the $d \times 1$
vector: 
\begin{equation*}
  X_i + \frac{\nabla f(X_i)}{f(X_i)} \qt{for $i = 1, \dots, n$}. 
\end{equation*}
With this notation, we can clearly rewrite $\Rr_n(\hat{\theta},
\hat{\theta}^*)$ as 
\begin{equation*}
  \Rr_n(\hat{\theta}, \breve{\theta}^*) = \E \left(\frac{1}{n}
    \norm{T_{\hat{f}_n}(\bx) - T_{f_{\bar{G}^0_n}}(\bx)}_F^2 \right) 
\end{equation*}
where $\norm{\cdot}_F$ denotes the usual Frobenius norm for matrices. 

To bound the above, we first observe that since $\hat{f}_n$ is an
NPMLE defined as in \eqref{kw}, it follows from the general maximum
likelihood theorem (see, for example, \citet[Theorem
2.1]{bohning2000computer}) that  
\begin{equation}\label{kk}
  \frac{1}{n} \sum_{i=1}^n \frac{\phi_d(X_i - \theta)}{\hat{f}_n(X_i)}
  \leq 1 
\end{equation}
for every $\theta \in \R^d$. Taking $\theta = X_i$ in the above
inequality, we deduce that 
\begin{equation*}
 1 \geq \frac{\phi_d(X_i - \theta)}{n \hat{f}_n(X_i)} =
 \frac{\phi_d(0)}{n \hat{f}_n(X_i)} 
\end{equation*}
so that $\hat{f}_n(X_i) \geq \phi_d(0)/n = (2
\pi)^{-d/2} n^{-1}$. Since this is true for each $i = 1, \dots, n$,
this means that 
\begin{equation}\label{rob}
  \min_{1 \le i \le n} \hat{f}_n(X_i) \geq \rho_n := \frac{(2\pi)^{-d/2}}{n}. 
\end{equation}
As a result, $\hat{f}_n(X_i) = \max(\hat{f}_n(X_i), \rho_n)$ for each
$i$ so that $T_{\hat{f}_n}(\bx) = T_{\hat{f}_n}(\bx, \rho_n)$ where
for $f \in \M$ and $\rho > 0$, we define $T_f(\bx, \rho)$ to be the $d
\times n$ matrix whose $i^{th}$ column is given by the $d \times 1$
vector: 
\begin{equation*}
  X_i + \frac{\nabla f(X_i)}{\max(f(X_i), \rho)} \qt{for $i = 1,
    \dots, n$}. 
\end{equation*}
This gives 
\begin{equation*}
  \Rr_n(\hat{\theta}, \hat{\theta}^*) = \E \left(\frac{1}{n}
    \norm{T_{\hat{f}_n}(\bx, \rho_n) - T_{f_{\bar{G}^0_n}}(\bx)}_F^2
  \right). 
\end{equation*}
A difficulty in dealing with the expectation on the right hand side 
above comes from the fact that $\hat{f}_n$ is random. This is handled
by covering the random $\hat{f}_n$ by an $\epsilon$-net for a 
specific $\epsilon$ in the following way. First fix a compact set $S
\subseteq \R^d$ and $M \geq \sqrt{10 \log n}$. Note that by Theorem
\ref{dens} (specifically inequality \eqref{dens.eq} applied to $\alpha
= \beta = 0.5$ and $t = 1$), we deduce that the following inequality
holds with probability at least $1 - (2/n)$: 
\begin{equation}\label{eve}
  \hel(\hat{f}_n, f_{\bar{G}_n^0}) \leq \tilde{C}_d \epsilon_n(M, S,
  \bar{G}_n^0).  
\end{equation}
Here $\tilde{C}_d$ is a positive constant depending on $d$ alone and
$\epsilon_n(M, S, \bar{G}^0_n)$ is defined as in \eqref{si.ra}.  Note
that Theorem \ref{dens} is indeed applicable here as $X_1, \dots, X_n$
are independent random vectors with  
\begin{equation*}
  X_i \sim N(\theta_i, \Sigma_i) = f_{G_i^0}  
\end{equation*}
where $G_i^0$ is the $N(\theta_i, \Sigma_i - I_d)$ and $\bar{G}^0_n$ is
the average of $G_i^0$ over $i = 1 \dots, n$.  

Let $E_n$ denote the event that \eqref{eve} holds.  We now obtain a
covering of  
\begin{equation}\label{spp}
\{f \in \M : \hel(f, f_{\bar{G}^0_n}) \leq \tilde{C}_d \epsilon_n(M,
S, \bar{G}_n^0)\}   
\end{equation}
under the pseudometric given by
\begin{equation}\label{psm}
  \norm{f - g}_{S^M, \nabla}^{\rho_n} := \sup_{x \in S^M}
  \norm{\frac{\nabla f(x)}{\max(f(x), \rho_n)} -
    \frac{\nabla g(x)}{\max(g(x), \rho_n)}} 
\end{equation}
where $S^M := \{x \in \R^d : \dos(x) \leq M\}$. We have proved
covering number bounds under this pseudometric in Corollary \ref{rgnt}
which will be used in this proof. Let $f_{G_1}, \dots, f_{G_N}$ denote a
maximal subset of \eqref{spp} such that for every $i \neq j$, we have 
\begin{equation}\label{pake}
  \norm{f_{G_i} - f_{G_j}}_{S^M, \nabla}^{\rho_n} \geq 2 \eta^*
\end{equation}
where $\eta^*$ is defined in terms of 
\begin{equation}\label{ets}
  \eta^* :=
  \left(\frac{1}{\rho_n} + \sqrt{\frac{1}{\rho_n^2} \log \frac{1}{(2
        \pi)^d \rho_n^2}} \right) \eta ~~~ \text{ and } ~~~ \eta :=
  \frac{\rho_n}{n}. 
\end{equation}
By the usual relation between packing and covering numbers, the
integer $N$ is then bounded from above by $N(\eta^*, \M, 
\norm{\cdot}_{S^MN, \nabla}^{\rho_n})$ which is bounded in Corollary
\ref{rgnt}. Specifically, Corollary \ref{rgnt} (applied to $S^M$)
gives   
\begin{equation*}
 \log N \leq C_d N(a, (S^M)^a) |\log \eta|^{d+1} \leq C_d N(a, S^{M+ a})
 (\log n)^{d+1} 
\end{equation*}
where
\begin{equation}\label{adef}
a := \sqrt{2 \log (2 \sqrt{2 \pi} n^2)}.   
\end{equation}
This further implies (via the use of inequality \eqref{volm.eq} in
Lemma \ref{volm} to bound  
$N(a, S^{M+a})$ as $N(a, S^{M+a}) \leq C_d a^{-d} \vol(S^{M +
  (3a/2)})$) that 
\begin{equation*}
 \log N \leq C_d (\log n)^{d+1} a^{-d} \vol(S^{M + (3a/2)}) \leq C_d (\log
 n)^{(d/2) + 1} \vol(S^{M + (3a/2)}). 
\end{equation*}
Using \eqref{sms1} in Lemma \ref{volm} to bound $\vol(S^{M + (3a/2)})$
in terms of $\vol(S^1)$ (and the fact that $a \leq C \sqrt{10 \log n}
\leq C M$), we obtain 
\begin{equation}\label{nbo}
  \log N \leq C_d \vol(S^1) M^d (\log n)^{(d/2) + 1}. 
\end{equation}
Also because $f_{G_1}, \dots, f_{G_N}$ is a maximal subset of
\eqref{spp} satisfying \eqref{pake}, we have 
\begin{equation}\label{hbb}
 \max_{1 \leq j \leq N} \hel(f_{G_j}, f_{\bar{G}^0_n}) \leq \tilde{C}_d
 \epsilon_n(M, S, \bar{G}_n^0) 
\end{equation}
and, on the event $E_n$, 
\begin{equation}\label{rco}
  \min_{1 \leq j \leq N} \norm{\hat{f}_n - f_{G_j}}_{S^M,
    \nabla}^{\rho_n}\leq 2 \eta^*.  
\end{equation}
We are now ready to bound the risk $\Rr_n(\hat{\theta},
\breve{\theta}^*)$. The strategy is to break down the risk into various
terms involving the densities $f_{G_1}, \dots, f_{G_N}$. 

\textbf{Breakdown of the risk}: The risk 
\begin{equation*}
  \Rr_n(\hat{\theta}, \breve{\theta}^*) = \E \left(\frac{1}{n}
    \norm{T_{\hat{f}_n}(\bx, \rho_n) - T_{f_{\bar{G}^0_n}}(\bx)}_F^2 \right) 
\end{equation*}
will be broken down via the inequality: 
\begin{align}
\|T_{\hat{f}_n}(\mathbf{X}, \rho_n) - T_{f_{\bar{G}_n^0}}(\mathbf{X}) \|_F &\leq
                                                                   \|T_{\hat{f}_n}(\mathbf{X},
                                                                           \rho_n)
                                                                   -
                                                                   T_{f_{\bar{G}^0_n}}(\mathbf{X},
                                                                   \rho_n)
                                                                   \|_F \\
                                                                   &+
                                                                   \|T_{f_{\bar{G}^0_n}}(\mathbf{X},
                                                                   \rho_n)
                                                                   -
                                                                   T_{f_{\bar{G}^0_n}}(\mathbf{X})
                                                                   \|_F \nonumber
  \\ 
&\leq (\zeta_{1n} + \zeta_{2n} + \zeta_{3n} + \zeta_{4n}) + \zeta_{5n} \label{bkd}
\end{align}
where
\begin{align*}
\zeta_{1n} &:= \|T_{\hat{f}_n}(\mathbf{X}, \rho_n) - T_{f_{\bar{G}^0_n}}(\mathbf{X}, \rho_n) \|_F I(E^c_n) \\
\zeta_{2n} &:= \left( \|T_{\hat{f}_n}(\mathbf{X}, \rho_n) -
             T_{f_{\bar{G}^0_n}}(\mathbf{X}, \rho_n) \|_F \right. \\ & \left. - \max_{1 \leq j
	\leq N}   \|T_{f_{G_j}}(\mathbf{X}, \rho_n) - T_{f_{\bar{G}^0_n}}(\mathbf{X}, \rho_n) \|_F  \right)_+ I(E_n) \\
\zeta_{3n} &:= \max_{1 \leq j \leq N} \left(\|T_{f_{G_j}}(\mathbf{X},
             \rho_n) - 
T_{f_{\bar{G}_n^0}}(\mathbf{X}, \rho_n) \|_F \right. \\ & \left. - \E \|T_{f_{G_j}}(\mathbf{X}, \rho_n) - T_{f_{\bar{G}^0_n}}(\mathbf{X}, \rho_n) \|_F \right)_+ \\
\zeta_{4n} &:= \max_{1 \leq j \leq N} \E \|T_{f_{G_j}}(\mathbf{X}, \rho_n) - T_{f_{\bar{G}^0_n}}(\mathbf{X}, \rho_n) \|_F \\
\zeta_{5n} &:= \|T_{f_{\bar{G}^0_n}}(\mathbf{X}) - T_{f_{\bar{G}^0_n}}(\mathbf{X}, \rho_n) \|_F 
\end{align*} 
With the elementary inequality $(a_1 + \dots + a_5)^2  
\leq 5(a_1^2 + \dots + a_5^2)$, inequality \eqref{bkd} gives
\begin{equation*}
  \Rr_n(\hat{\theta}, \hat{\theta}^*) \leq 5 \sum_{i=1}^5 \frac{\E
    \zeta_{in}^2}{n}. 
\end{equation*}
The proof of Theorem \ref{rgende} will 
be completed below by showing the existence of a positive constant
$C_d$ such that, for every $i = 1, \dots, 5$, 
\begin{equation}\label{mnts}
\begin{split}
  \E \zeta_{in}^2 &\leq C_d \sigma^2_{\max} n \epsilon_n^2(M, S, \bar{G}_n^0)
  \left(\sqrt{\log n} \right)^{\max(d-2, 6)}  \\
&= C_d  \sigma^2_{\max}\left(\mathrm{Vol}(S^1) M^d \left(\sqrt{\log n}
  \right)^{d+2} \right. \\ & \left.+ n \left(\log n \right) \inf_{p \geq
    \frac{d+1}{2 \log n}} \left(\frac{2 \mu_p(\mathfrak{d}_S, \bar{G}_n^0)}{M}
  \right)^{p} \right) \left(\sqrt{\log n} \right)^{\max(d-2, 6)}. 
\end{split}
\end{equation}
It may be noted that $\zeta_{4n}$ is non-random so that the
expectation above can be removed for $i = 4$. Every other $\zeta_{in}$
is random. We will actually prove
\eqref{mnts} without the multiplicative factor of $\sigma^2_{\max}$
for $i \neq 3$; the factor of $\sigma^2_{\max}$ only appears for $i =
3$ (note that $\sigma^2_{\max} \geq 1$ because of the assumption that
$\Sigma_i \gtrsim I_d$).  

\textbf{Bounding $\E \zeta_{1n}^2$}: We write
\begin{align*}
  \E \zeta_{1n}^2 &= \E \left(\|T_{\hat{f}_n}(\mathbf{X}, \rho_n) -
                    T_{f_{\bar{G}^0_n}}(\mathbf{X}, \rho_n) \|^2_F
                    I(E^c_n) \right) \\ &= \sum_{i=1}^n \E
                    \left(\norm{\frac{\nabla
                    \hat{f}_n(X_i)}{\max(\hat{f}_n(X_i), \rho_n)} - \frac{\nabla
                    {f}_{\bar{G}^0_n}(X_i)}{\max({f}_{\bar{G}^0_n}(X_i),
                    \rho_n)}}^2 I(E_n^c)\right) . 
\end{align*}
Inequality \eqref{p1.neq} in Lemma \ref{p1} now gives
\begin{align}\label{p1cq}
\norm{\frac{\nabla
                    \hat{f}_n(X_i)}{\max(\hat{f}_n(X_i), \rho_n)} - \frac{\nabla
                    {f}_{\bar{G}^0_n}(X_i)}{\max({f}_{\bar{G}^0_n}(X_i),
                    \rho_n)}}^2 \leq 4 \log \frac{(2
  \pi)^d}{\rho_n^2}
\end{align}
provided $\rho_n \leq (2 \pi)^{-d/2} e^{-1/2}$ which is equivalent to
$n \geq \sqrt{e}$ and hence holds for all $n \geq
2$. This gives (note that $\P(E_n^c) \leq 2/n$) 
\begin{align*}
  \E \zeta_{1n}^2 &\leq 4 n \left( \log \frac{(2
  \pi)^d}{\rho_n^2} \right) \P(E_n^c) \\
&\le 8 \left( \log \frac{(2
  \pi)^d}{\rho_n^2} \right) \leq C_d \log n \leq C_d \vol(S^1) M^d
  (\sqrt{\log n})^{d+2}
\end{align*}
which proves \eqref{mnts} for $i = 1$. 

\textbf{Bounding $\E \zeta_{2n}^2$}: For this, we write
\begin{align*}
  \zeta_{2n}^2 &\leq \min_{1 \leq j \leq N} \norm{T_{\hat{f}_n}(\bx,
  \rho_n) - T_{f_{G_j}}(\bx, \rho_n)}_F^2 I(E_n) \\
&= \min_{1 \leq j \le N} \sum_{i=1}^n \norm{\frac{\nabla
                    \hat{f}_n(X_i)}{\max(\hat{f}_n(X_i), \rho_n)} - \frac{\nabla
                    {f}_{{G}_j}(X_i)}{\max({f}_{{G}_j}(X_i),
                    \rho_n)}}^2 I(E_n) \\
&\leq \min_{1 \leq j \leq N} \left(\norm{\hat{f}_n - f_{G_j}}_{S^M,
  \nabla}^{\rho_n} \right)^2 \left(\sum_{i=1}^n I\{X_i \in S^M\}
  \right) I(E_n) \\
&+ \left(4 \log \frac{(2 \pi)^d}{\rho_n^2} \right)
  \left(\sum_{i=1}^n I\{X_i \notin S^M\} \right) I(E_n). 
\end{align*}
where we have used the notation \eqref{psm} in the first term above
and the inequality \eqref{p1cq} in the second term. We can simplify
the above bound as 
\begin{align*}
  \zeta_{2n}^2 &\leq n \left( \min_{1 \leq j \leq N} \norm{\hat{f}_n - f_{G_j}}_{S^M,
  \nabla}^{\rho_n} \right)^2 I(E_n) \\ &+ \left(4 \log \frac{(2
  \pi)^d}{\rho_n^2} \right) \left(\sum_{i=1}^n I\{X_i \notin S^M\}
  \right).
\end{align*}
Inequality \eqref{rco} and the expression \eqref{ets} for $\eta^*$ now
give 
\begin{align*}
  \E\zeta_{2n}^2 &\leq \frac{4}{n} \left(1 + \sqrt{\log \frac{1}{(2
  \pi)^d \rho_n^2}} \right)^2+ \left(4 \log \frac{(2
  \pi)^d}{\rho_n^2} \right) \left(\sum_{i=1}^n \P\{X_i \notin S^M\}
  \right) \\
&\leq C_d \frac{\log n}{n} +  C_d (\log n) \left(\sum_{i=1}^n \P\{X_i \notin S^M\}
  \right). 
\end{align*}
To control the second term above, we use inequality
\eqref{Eq:tailProb} (which is a consequence of Lemma
\ref{tailmom}). Note that $\P\{X_i \notin S^M\} \leq \P\{\dos(X_i)
\geq M\}$. Inequality \eqref{Eq:tailProb} therefore gives 
\begin{align*}
  \E\zeta_{2n}^2 \le C_d \frac{\log n}{n} +  C_d (\log n) M^{d-2} +
  C_d (n \log n) \inf_{p \geq \frac{d+1}{2 \log n}} \left(\frac{2
  \mu_p(\dos, \bar{G}_n^0)}{M} \right)^p. 
\end{align*}
This proves \eqref{mnts} for $i = 2$ (note that $(\log n) M^{d-2} \leq
M^d$ as $M \geq \sqrt{10 \log n}$). 

\textbf{Bounding $\zeta_{3n}^2$}: Here Lemma \ref{devs} and the bound
\eqref{nbo} will be crucially used. Let us first write $\zeta_{3n} :=
\max_{1 \leq j \leq N} \zeta_{3n.j}$  where 
\begin{align*}
  \zeta_{3n.j} &:= \left(\|T_{f_{G_j}}(\mathbf{X}, \rho_n) -
T_{f_{\bar{G}^0_n}}(\mathbf{X}, \rho_n) \|_F \right. \\ & \left. - \E
\|T_{f_{G_j}}(\mathbf{X}, \rho_n) - T_{f_{\bar{G}^0_n}}(\mathbf{X},
\rho_n) \|_F \right)_+. 
\end{align*}
Lemma \ref{devs} (applied with $f_1 := f_{G_j}$ and $f_2 :=
f_{\bar{G}_n^0}$) then gives  
\begin{equation*}
  \P \left\{\zeta_{3n.j} \geq x \right\} \leq \exp \left(\frac{-x^2}{8
      \sigma^2_{\max} L^4(\rho_n)} \right), 
\end{equation*}
for every $1 \leq j \leq N$ and $x > 0$, where 
\begin{equation}\label{lrn}
  L(\rho_n) = \sqrt{\log \frac{1}{(2 \pi)^d \rho_n^2}} = \sqrt{\log n}.
\end{equation}
By the union bound, we have 
\begin{align*}
  \P \left\{\zeta_{3n} \geq x\right\} \leq N \exp \left(\frac{-x^2}{8
  \sigma^2_{\max} L^4(\rho_n)} \right) \qt{for every $x > 0$} 
\end{align*}
so that, for every $x_0 > 0$, 
\begin{align*}
  \E \zeta_{3n}^2 &\leq \int_0^{\infty} \P \left\{\zeta_{3n} \geq
    \sqrt{x}  \right\} dx \\
& \le x_0 + \int_{x_0}^{\infty} N \exp
  \left(\frac{-x}{8 \sigma^2_{\max} L^4(\rho_n)} \right) dx \\ &= x_0 + 8
  N \sigma^2_{\max} L^4(\rho_n) \exp \left(\frac{-x_0}{8
  \sigma^2_{\max} L^4(\rho_n)} \right).  
\end{align*}
Minimizing the above bound over $x_0 > 0$, we deduce that 
\begin{equation*}
  \E \zeta_{3n}^2 \le 8 \sigma^2_{\max} L^4(\rho_n) \log \left(e N
  \right).  
\end{equation*}
The bound \eqref{nbo} (along with \eqref{lrn}) then gives 
\begin{align*}
  \E \zeta_{3n}^2 &\leq C_d \sigma^2_{\max}  \vol(S^1) M^d (\sqrt{\log
                    n})^{d+6} \\ & \leq C_d \sigma^2_{\max}
                                     \vol(S^1) M^d (\sqrt{\log n})^{d+2} 
  (\log n)^2
\end{align*}
which proves \eqref{mnts} for $i = 3$.  

\textbf{Bounding $\zeta_{4n}^2$}: To bound the non-random quantity
$\zeta_{4n}^2$, we only need to bound  
\begin{equation*}
  \Gamma^2_j := \E \norm{T_{f_{G_j}}(\bx, \rho_n) - T_{f_{\bar{G}^0_n}}(\bx,
    \rho_n)}_F^2 
\end{equation*}
for each $1 \leq j \leq N$. We can clearly write
\begin{align*}
  \Gamma^2_j &= \sum_{i=1}^n \E \norm{\frac{\nabla
      f_{G_j}(X_i)}{\max(f_{G_j}(X_i), \rho_n)} - \frac{\nabla
      f_{\bar{G}^0_n}(X_i)}{\max(f_{\bar{G}^0_n}(X_i), \rho_n)}}^2 \\
&= n \int
  \norm{\frac{\nabla
      f_{G_j}(x)}{\max(f_{G_j}(x), \rho_n)} - \frac{\nabla
      f_{\bar{G}^0_n}(x)}{\max(f_{\bar{G}^0_n}(x), \rho_n)}}^2
  f_{\bar{G}^0_n}(x) dx.  
\end{align*}
The above term can be bounded by a direct application of Theorem
\ref{fcc} which furnishes a bound in terms of $\hel(f_{G_j},
f_{\bar{G}_n^0})$. Indeed, because $n \geq 2$, we have $\rho_n \leq (2
\pi)^{-d/2} e^{-1/2}$ so that Theorem \ref{fcc} applies (with $G =
G_j$ and $G_0 = \bar{G}^0_n$) and we obtain
\begin{align*}
 \frac{1}{n} \Gamma_j^2 &\leq C_d \max \left\{ \left(\log \frac{(2
       \pi)^{-d/2}}{\rho_n} \right)^3, |\log \hel(f_{G_j}, f_{\bar{G}_n^0})|
 \right\} \hel^2(f_{G_j}, f_{\bar{G}^0_n})  \\
&= C_d \max \left\{ \left(\log n\right)^3, |\log \hel(f_{G_j}, f_{\bar{G}^0_n})|
 \right\} \hel^2(f_{G_j}, f_{\bar{G}^0_n}). 
\end{align*}
We now use that $\hel(f_{G_j}, f_{\bar{G}^0_n})$ is bounded
from above by $\tilde{C}_d \epsilon_n(M, S, \bar{G}_n)$ (see \eqref{hbb}). We 
then work with two cases. If $\tilde{C}_d \epsilon_n(M, S) \leq
e^{-1/2}$ , then using the fact that $h \mapsto h^2 |\log h|$ is
increasing on $(0, e^{-1/2}]$, we have 
\begin{equation*}
  \frac{1}{n} \Gamma_j^2 \leq C_d \tilde{C}_d^2 \max \left\{ \left(\log n\right)^3,
    \left|\log (\tilde{C}_d \epsilon_n(M, S, \bar{G}^0_n) ) \right|
  \right\} \epsilon_n^2(M, S, \bar{G}_n^0). 
\end{equation*}
The trivial observation $\epsilon_n(M, S, \bar{G}_n^0) \geq K_d/n$ for a constant
$K_d$ now gives
\begin{equation}\label{fobo}
\Gamma_j^2 \le n C_d (\log n)^3 \epsilon_n^2(M, S, \bar{G}_n^0). 
\end{equation}
On the other hand when $\tilde{C}_d \epsilon_n(M, S, \bar{G}_n^0) >
e^{-1/2}$, then we can simply bound $|\log \hel(f_{G_j}, f_{\bar{G}^0_n})|
\hel^2(f_{G_j}, f_{\bar{G}^0_n})$ by a constant (the function $h \mapsto
h^2 |\log h|$ is bounded on $h \in (0, 2]$) so that the inequality
\eqref{fobo} still holds. The bound in the right hand side of
\eqref{fobo} does not depend on $j$ so that it is an upper bound for
$\zeta_{4n}^2$ as well. This proves \eqref{mnts} for $i = 4$. 

\textbf{Bounding $\E \zeta_{5n}^2$}: We write 
\begin{align*}
  \E \zeta_{5n}^2 &= \E \norm{T_{f_{\bar{G}^0_n}}(\bx) -
                    T_{f_{\bar{G}^0_n}}(\bx, \rho_n)}_F^2 \\
&= \sum_{i=1}^n \E \norm{\frac{\nabla
  f_{\bar{G}^0_n}(X_i)}{f_{\bar{G}^0_n}(X_i)} - \frac{\nabla
  f_{\bar{G}^0_n}(X_i)}{\max(f_{\bar{G}^0_n}(X_i), \rho_n )} }^2  \\
&= n \int \norm{\frac{\nabla
  f_{\bar{G}^0_n}(x)}{f_{\bar{G}^0_n}(x)} - \frac{\nabla
  f_{\bar{G}^0_n}(x)}{\max(f_{\bar{G}^0_n}(x), \rho_n )}}^2
  f_{\bar{G}^0_n}(x) dx \\
&= n\int \left(1 - \frac{f_{\bar{G}^0_n}}{\max \left(f_{\bar{G}^0_n} ,
  \rho \right)} \right)^2
  \frac{\norm{\nabla f_{\bar{G}^0_n}}^2}{f_{\bar{G}^0_n}} = n
  \Delta(\bar{G}^0_n, \rho_n) 
\end{align*}
where we define 
\begin{align*}
  \Delta(G, \rho) := \int \left(1 - \frac{f_{{G}}}{\max \left(f_{{G}}
  , \rho \right)} \right)^2
  \frac{\norm{\nabla f_{{G}}}^2}{f_{{G}}} 
\end{align*}
for probability measures $G$ on $\R^d$ and $\rho > 0$. We now use
Lemma \ref{dco} to bound $\Delta(\bar{G}^0_n, \rho_n)$.  Specifically,
inequality \eqref{dco.eq} in Lemma \ref{dco} applied to the compact
set $S^M$ gives 
\begin{align}\label{rud}
  \Delta(\bar{G}^0_n, \rho_n) \leq C_d N \left(\frac{4}{L(\rho_n)}, S^M
  \right) L^d(\rho_n) \rho_n + d ~ \bar{G}^0_n((S^M)^c). 
\end{align}
The first term above is bounded using Lemma \ref{volm} as follows
(note that $\rho_n = (2 \pi)^{-d/2}/n$ and $L(\rho_n) = \sqrt{\log n}$
as shown in \eqref{lrn}): 
\begin{align*}
  N \left(\frac{4}{L(\rho_n)}, S^M \right) L^d(\rho_n) \rho_n &= N \left(\frac{4}{                                                                \sqrt{\log
                                                                n}},
                                                                S^M \right)
                                                                (\log
                                                                n)^{d/2}
                                                                \frac{(2
                                                                \pi)^{-d/2}}{n}
  \\
\left(\text{using inequality
  \eqref{volm.eq}} \right) &\leq C_d (4/ \sqrt{\log n})^{-d} \vol((S^M)^{2/ \sqrt{\log n}})
  \frac{(\log n)^{d/2}}{n} \\
&\le \frac{C_d}{n} (\log n)^d  \vol(S^{M + 2/ \sqrt{\log n}}) \\ 
\left(\text{using inequality \eqref{sms1}} \right) &\le \frac{C_d}{n} (\log n)^d \vol(S^1) \left(1 + \frac{M}{4} + \frac{1}{2
  \sqrt{\log n}} 
  \right)^d \\
&\le \frac{C_d}{n} (\log n)^d M^d \vol(S^1). 
\end{align*}
For the second term in \eqref{rud}, note that
\begin{equation*}
 \bar{G}^0_n((S^M)^c) \le \int I \{\dos(\theta) \geq M\} d
 \bar{G}^0_n(\theta) \leq \inf_{p \geq \frac{d+1}{2 
\log n}}  \left(\frac{2 \mu_p(\dos, \bar{G}_n^0)}{M} \right)^p. 
\end{equation*}
We have therefore proved that 
\begin{align*}
  \E \zeta_{5n}^2 &\leq n \Delta(\bar{G}^0_n, \rho_n) \\ &\leq C_d \left\{ (\log n)^d
  M^d \vol(S^1) + n \inf_{p \geq \frac{d+1}{2 
\log n}}  \left(\frac{2 \mu_p(\dos, \bar{G}_n^0)}{M} \right)^p
                                                           \right\} \\
  &\leq C_d \left\{ (\sqrt{\log n})^{d+2}
  M^d \vol(S^1) + n \inf_{p \geq \frac{d+1}{2 
\log n}}  \left(\frac{2 \mu_p(\dos, \bar{G}_n^0)}{M} \right)^p
                                                           \right\}
    (\sqrt{\log n})^{d-2}
\end{align*}
which evidently implies \eqref{mnts} for $i = 5$. The proof of Theorem
\ref{rgende} is now complete.   
\end{proof}

\section{Heteroscedastic Gaussian Denoising}\label{hete}
In this subsection, we provide more details about the heteroscedastic
setting and Theorem \ref{rgende} and explain why the Oracle estimator
$\breve{\theta}_i^*$ (defined in \eqref{es.bre}) is not the best
separable estimator. We first state the following corollary of Theorem
\ref{rgende} which can be seen as the analogue of Proposition
\ref{pkde}  for the heteroscedastic setting. 

\begin{proposition}\label{pkde.h}
  Consider the same setting and notation as in Theorem
  \ref{rgende}. Suppose that $\theta_1, \dots,
  \theta_n$ satisfy \eqref{kga} for some $a_1, \dots, a_k \in \R^d$
  and $R \geq 0$. Then $\Rr_n(\hat \theta, \breve \theta^*)$ is
  bounded from above by
  \begin{equation*}
 C_d \sigma^2_{\max}
    (\max(1, \tau))^d \left(1 + R \right)^d
    \left(\frac{k}{n} \right) \left(\sqrt{\log n} \right)^{\max(3d, 2d+8)} 
    \end{equation*}
where $\tau := \sqrt{\sigma^2_{\max} - 1}$ and $C_d$ is a constant
depending only on $d$. 
\end{proposition}
The above result has similar interpretation to Proposition
\ref{pkde}: when the unknown $\theta_1, \dots, \theta_n$ can be clustered into $k$
groups, then $\hat{\theta}_1, \dots, \hat{\theta}_n$ estimate
$\breve{\theta}_1^*, \dots, \breve{\theta}_n^*$ in squared error loss
with accuracy $k/n$ up to logarithmic multiplicative factors in $n$
(assuming that $\sigma^2_{\max}$ is bounded from above by a
constant). The key here is to realize that the estimator does not use
any knowledge of $k$ and is tuning-free (it only requires
$\Sigma_i \gtrsim I_d$ for each $i$). 

Now we shall explain why the quantities $\breve{\theta}_1^*, \dots,
\breve{\theta}_n^*$ do not give the best separable estimator in the
heteroscedastic setting. This is mainly the reason why Theorem
\ref{rgende} is of somewhat limited interest in the heteroscedastic
situation. Note first that in the homoscedastic case when $\Sigma_i
= I_d$ for each $i = 1, \dots, n$, the probability measure
$\bar{G}_n^0$ (defined in \eqref{gengeb}) is exactly equal to the
empirical measure corresponding 
to $\theta_1, \dots, \theta_n$ and, consequently, we have 
\begin{equation*}
  \breve{\theta}_i^* = \hat{\theta}_i^* \qt{for every $i = 1, \dots n$}
\end{equation*}
where $\hat{\theta}_1^*, \dots, \hat{\theta}_n^*$ are as defined in
\eqref{es.int}. Also, in this homoscedastic case, as remarked in
Section \ref{sec_intro}, $\breve{\theta}_i^* = \hat{\theta}_i^*$ has
the property 
of being equal to $T^*(X_i)$ where $T^*$ is the best separable estimator in
the sense of minimizing \eqref{basri} over all functions $T: \R^d
\rightarrow \R^d$. This observation makes $\breve{\theta}_1^*, \dots,
\breve{\theta}_n^*$ an ideal target for estimation in the
homoscedastic setting. 

Now let us get to the heteroscedastic setting where $X_1, \dots, X_n$
are independent satisfying \eqref{hetdi}. In this setting, the best
separable estimator is specified in the next result. Recall that
separable estimators of $\theta_1, \dots, \theta_n$ are estimators of
the form $T(X_1), \dots, T(X_n)$ where $T: \R^d \rightarrow \R^d$ is a
deterministic function. The best separable estimator is then given by
$T^*(X_1), \dots, T^*(X_n)$ where $T^*$ minimizes 
\begin{equation}\label{jri}
\Upsilon(T) := E \left[  \frac{1}{n} \sum_{i=1}^n \norm{T(X_i) - \theta_i}^2 \right]
\end{equation}
over all possible functions $T: \R^d \rightarrow \R^d$. 

\begin{lemma}\label{bsep}
  Consider the problem  of estimating $\theta_1, \dots, \theta_n$ from
  independent observations $X_1, \dots, X_n$ with $X_i \sim
  N(\theta_i, \Sigma_i)$. Suppose that $\Sigma_1, \dots, \Sigma_n$ are
  unknown. Then the best separable estimator for $\theta_1, \dots,
  \theta_n$ is given by $T^*(X_1), \dots,T^*(X_n)$ where
   \begin{align}\label{bsep.eq}
     T^*(x) := \frac{\frac{1}{n} \sum_{j=1}^n \theta_j
    \phi_d(x, \theta_j, \Sigma_j)}{\frac{1}{n} \sum_{j=1}^n
    \phi_d(x, \theta_j, \Sigma_j)}
   \end{align}
   where $\phi_d(x, \mu, \Sigma)$ denotes the $d$-variate normal
   density (evaluated at $x$) with mean vector $\mu$ and covariance
   matrix $\Sigma$. 
\end{lemma}
Note that the best separable estimator $T^*(x)$ given by
\eqref{bsep.eq} can also be written as
\begin{equation*}
  T^*(x) := \E (\theta | X = x) ~~ \text{where } (\theta,
    \Sigma) \sim \bar{G}_n^* \text{ and } X|(\theta, \Sigma) \sim
    N(\theta, \Sigma).  
\end{equation*}
where $\bar{G}_n^*$ is the empirical measure corresponding to
$(\theta_1, \Sigma_1), \dots, (\theta_n, \Sigma_n)$. In other words,
$\bar{G}_n^*$ is a discrete prior on $(\theta, \Sigma)$ which takes
the value $(\theta_i, \Sigma_i)$ with probability $1/n$. The best
separable estimator is then given by $\hat{\theta}^*_i := T^*(X_i)$ and
this also has the alternative expression: 
\begin{equation}\label{hetor}
  \hat{\theta}^*_i := \E (\theta | X = X_i) ~~ \text{where } (\theta,
    \Sigma) \sim \bar{G}_n^* \text{ and } X|(\theta, \Sigma) \sim
    N(\theta, \Sigma).  
\end{equation}
The above expression should be compared with the expression
\eqref{hor} for $\breve{\theta}_i^*$. We now argue
that, under  heteroscedasticity, $\breve{\theta}_i^*$ can be quite far
from $\hat{\theta}_i^*$ (note that they are equal in the homoscedastic
setting). The following lemma provides bounds on the discrepancy
between $\breve{\theta}_i^*$ and $\hat{\theta}_i^*$. 

\begin{lemma}\label{bodisc}
  Consider the same setting and notation as in Theorem \ref{rgende}
  and let $\hat{\theta}_1^*, \dots, \hat{\theta}_n^*$ be as in
  \eqref{hetor}. Then 
  \begin{equation}
    \label{bodisc.eq}
    \Rr_n(\hat \theta^*, \breve{\theta}^*) := \E \left[\frac{1}{n}
          \sum_{i=1}^n \|\breve{\theta}^*_i - \hat{\theta}_i^* \|^2
        \right] \leq d \sigma^2_{\max} \left(1 -
          \frac{1}{\sigma^2_{\max}} \right)^2. 
  \end{equation}
Further, if $\theta_1 = \dots = \theta_n = \theta_0$ for some vector
$\theta_0 \in \R^d$ and if $\Sigma_1 = \dots = \Sigma_n =
\sigma^2_{\max} I_d$, then 
\begin{equation}
  \label{bodisc.peq}
    \Rr_n(\hat \theta^*, \breve{\theta}^*) := \E \left[\frac{1}{n}
          \sum_{i=1}^n \|\breve{\theta}^*_i - \hat{\theta}_i^* \|^2
        \right] = d \sigma^2_{\max} \left(1 -
          \frac{1}{\sigma^2_{\max}} \right)^2. 
\end{equation}
\end{lemma}

Lemma \ref{bodisc} (specifically \eqref{bodisc.peq}) implies that
$\Rr_n(\hat{\theta}^*, \breve{\theta}^*)$ can be as large as $d
\sigma^2_{\max} (1 - \sigma^{-2}_{\max})^2$ for certain configurations
of $\theta_1, \Sigma_1, \dots, \theta_n, \Sigma_n$ while, in general,
it is always less than or equal to this quantity. Consequently, for
$\hat{\theta}_1, \dots, \hat{\theta}_n$ to consistently estimate
$\hat{\theta}_1^*, \dots, \hat{\theta}_n^*$, it is necessary that
$\sigma_{\max}^2$ be close to $1$. In fact, combining Theorem
\ref{rgende} (and Proposition \ref{pkde.h}) with inequality
\eqref{bodisc.eq}, we obtain the following result on the discrepancy
between $\hat{\theta}_i$ and 
$\hat{\theta}_i^*$. 

\begin{theorem}\label{fire}
    Consider the same setting and notation as in Theorem \ref{rgende}
  and let $\hat{\theta}_1^*, \dots, \hat{\theta}_n^*$ be as in
  \eqref{hetor}. 
  \begin{enumerate}
  \item There exists a positive constant $C_d$ (depending only on $d$)
    such that for every non-empty compact set $S \subseteq \R^d$ and
    $M \geq \sqrt{10 \log n}$, we have 
   \begin{equation}
    \label{bodisc.eq}
    \Rr_n(\hat \theta, \hat \theta^*) \leq C_d \sigma^2_{\max}
    \left[\epsilon_n^2(M, S, \bar{G}_n^0) (\sqrt{\log n})^{\max(d-2, 6)} +  \left(1 -
          \frac{1}{\sigma^2_{\max}} \right)^2 \right]. 
  \end{equation}
   \item Suppose $\theta_1, \dots, \theta_n$ satisfy \eqref{kga} for
     some $a_1, \dots, a_k \in \R^d$ and $R \geq 0$. Then $\Rr_n(\hat
     \theta, \hat \theta^*)$ is  bounded from above by
  \begin{align*}
 C_d \sigma^2_{\max}
  & \left[  (\max(1, \tau))^d \left(1 + R \right)^d
    \left(\frac{k}{n} \right) \left(\sqrt{\log n} \right)^{\max(3d, 2d+8)} \right. \\
&\left. + \left(1 -
          \frac{1}{\sigma^2_{\max}} \right)^2 \right] \qt{where $\tau
  := \sqrt{\sigma^2_{\max} - 1}$}. 
    \end{align*}
  \end{enumerate}
\end{theorem}

Theorem \ref{fire} implies that $\hat{\theta}_1, \dots,
\hat{\theta}_n$ perform well as an approximation of the best
separable estimator in the heteroscedastic setting (under the
assumption that $\Sigma_i - I_d$ is positive semi-definite) when
$\sigma^2_{\max}$ is close to 1 (i.e., when the heteroscedasticity is
mild and we have near-homoscedasticity with a known variance
lower bound). On
the other hand, when $\sigma^2_{\max}$ is not close to $1$,
$\hat{\theta}_1, \dots, \hat{\theta}_n$ will not work for
approximating the best separable rule. For example, if
$\sigma_{\max}^2$ differs from 1 by a constant, then the risk
$\Rr_n(\hat{\theta}, \hat{\theta}^*)$ is also a constant and will
remain a constant irrespective of $n$.  In this case, $\hat{\theta}_1,
\dots, \hat{\theta}_n$ will only
provide a good approximation for $\breve{\theta}_1^*, \dots, 
\breve{\theta}_n^*$. This is a price (of needing to be in a
near-homoscedastic setting with a known variance lower bound) that the
estimator  
$\hat{\theta}_1, \dots, \hat{\theta}_n$ pays for the property of being
able to work for a wide variety of structures of $\theta_1, \dots,
\theta_n$. 

\subsection{Proofs of Results in Section \ref{hete}} 
\subsection{Proof of Proposition \ref{pkde.h}} 
Using Theorem \ref{rgende}, it is clear that to prove Proposition
\ref{pkde.h}, it is enough to show the existence of a compact set 
$S \subseteq \R^d$ and $M \geq \sqrt{10 \log n}$ such that 
\begin{equation}\label{talaca}
  \epsilon_n^2(M, S, \bar{G}_n^0) \leq C_d (1 + R)^d
  \left(\frac{k}{n}\right) (\max(1, \tau))^d (\log n)^{d+1} 
\end{equation}
where $\tau := \sqrt{\sigma^2_{\max} - 1}$. We shall take  
\begin{equation*}
  S := \cup_{j=1}^k B(a_j, R) \qt{where $B(a_j, R) := \left\{x \in
      \R^d : \|x - a_j\| \leq R \right\}$}. 
\end{equation*}
Note then that for every $i = 1, \dots, n$ and $\theta \in \R^d$, we have
\begin{equation*}
  \dos(\theta)  \leq \min_{1 \leq j \leq k} \inf_{x \in B(a_j,
    R)}\|\theta - x\| \leq \|\theta - \theta_i\|
\end{equation*}
because, by \eqref{kga}, there exists $1 \le j \leq k$ for which
$\theta_i \in B(a_j, R)$. The above inequality implies (recall that
$G_i^0$ is the $N(\theta_i, \Sigma_i - I_d)$ distribution which is
taken to be the Dirac probability measure concentrated at $\theta_i$
when $\Sigma_i = I_d$) that for every $p \geq 1$, we have
\begin{align*}
  \mu_p(\dos, \bar{G}_n^0) &\leq \left(\frac{1}{n} \sum_{i=1}^n \int
                             (\dos(\theta))^p d G_i^0(\theta)
                             \right)^{1/p} \\
&\leq \left(\frac{1}{n} \sum_{i=1}^n \int
                             \|\theta - \theta_i\|^p d G_i^0(\theta)
                             \right)^{1/p} \\
&= \left(\frac{1}{n} \sum_{i=1}^n \E\| \left(\Sigma_i - I_d
  \right)^{1/2} Z \|^p \right)^{1/p} \\
&= \left( \frac{1}{n} \sum_{i=1}^n \E \left(Z^T (\Sigma_i - I_d) Z
  \right)^{p/2} \right)^{1/p} \\
&\leq \sqrt{\sigma^2_{\max} - 1} \left(\E \|Z\|^p \right)^{1/p} \leq
  C_d \sqrt{\sigma^2_{\max} - 1} \sqrt{p}. 
\end{align*}
Let $\tau := \sqrt{\sigma_{\max}^2 - 1}$ so that the above calculation
gives $\mu_p(\dos, \bar{G}_n^0) \leq C_d \max(\tau, 1) \sqrt{p}$ for
some $C_d \geq 1$ (if $C_d < 1$, simply take $C_d =1$). We now use
inequality \eqref{tobeusl} in the proof of Corollary \ref{dza} with
$\alpha = 2$ and $K = C_d \max(\tau, 1)$ to obtain
\begin{equation*}
\inf_{M \geq \sqrt{10 \log n}}  \epsilon_n^2(M, S, \bar{G}_n^0) \leq
C_d \vol(S^1) (\max(1, \tau))^d \left(\frac{1}{n} \right) (\log
n)^{d+1}. 
\end{equation*}
Note that the second term $(\log n)/n$ in \eqref{tobeusl} is dropped
above because it is dominated by the first term. Now we simply use the
inequality: 
\begin{equation*}
  \vol(S^1) \leq \sum_{j=1}^k \vol(B(a_j, 1 + R)) \leq C_d k (1 + R)^d
\end{equation*}
to deduce \eqref{talaca}. This completes the proof of Proposition
\ref{pkde.h}. 

\subsection{Proof of Lemma \ref{bodisc}} 
From the definition of $\hat{\theta}_i^*$ in \eqref{hetor}, it is
clear that 
\begin{equation}\label{fulb}
  \hat{\theta}_i^* = \frac{\frac{1}{n} \sum_{j=1}^n \theta_j
    \phi_d(X_i, \theta_j, \Sigma_j)}{\frac{1}{n} \sum_{j=1}^n
    \phi_d(X_i, \theta_j, \Sigma_j)} \qt{for $i = 1, \dots, n$}
\end{equation}
where $\phi_d(\cdot, \mu, \Sigma)$ denotes the $d$-variate normal
density with mean vector $\mu$ and covariance matrix $\Sigma$. 

On the other hand, $\breve{\theta}_i^*$ is given by 
\begin{equation*}
  \breve{\theta}_i^* = X_i + \frac{\nabla
    f_{\bar{G}_n^0}(X_i)}{f_{\bar{G}_n^0}(X_i)}. 
\end{equation*}
Let us write the gradient in the  right hand side above more
explicitly. First note that by the expression for $\bar{G}_n^0$ in
\eqref{gengeb}, it is clear that 
\begin{equation*}
  f_{\bar{G}_n^0}(x) = \frac{1}{n}\sum_{j=1}^n \phi_d(x, \theta_j,
  \Sigma_j). 
\end{equation*}
Note that the denominator in the right hand side of \eqref{fulb} is
 $f_{\bar{G}_n^0}(X_i)$. Differentiating the above with respect to
 $x$, we obtain
\begin{equation*}
\breve{\theta}^*_i = X_i + \frac{1}{f_{\bar{G}_n^0}(X_i)} \left[ \frac{1}{n}
  \sum_{j=1}^n \Sigma_j^{-1} (\theta_j - X_i) \phi_d(X_i, \theta_j,
  \Sigma_j) \right]. 
\end{equation*}
As a result, 
\begin{equation*}
J_i := \hat \theta^*_i - \breve \theta^*_i = \frac{1}{f_{\bar{G}_n^0}(X_i)}
\left[ \frac{1}{n} \sum_{j=1}^n (I_d - \Sigma_j^{-1}) (\theta_j - X_i)
  \phi_d(X_i, \theta_j, \Sigma_j) \right]. 
\end{equation*}
We will now represent the above formula for $J_i$ in an alternative
form. Suppose $\E^*$ denotes expectation under the joint distribution
of $(\theta, \Sigma, X)$ given by $(\theta, \Sigma) \sim \bar{G}_n^*$
and $X | (\theta, \Sigma) \sim N(\theta, \Sigma)$. Here $\bar{G}_n^*$
is the empirical measure corresponding to $(\theta_1, \Sigma_1),
\dots, (\theta_n, \Sigma_n)$. In other words, $\bar{G}_n^*$ is a
discrete prior taking the value $(\theta_i, \Sigma_i)$ with
probability $1/n$. Then it is easy to see that 
\begin{align*}
  J_i = \gamma(X_i) \qt{where $\gamma(x) := \E^* \left[(I_d -
  \Sigma^{-1}) (\theta - X) \big| X = x \right]$}. 
\end{align*}
We thus have
\begin{align*}
  \Rr_n(\hat{\theta}^*, \breve{\theta}^*) = \E \left[\frac{1}{n}
  \sum_{i=1}^n \|J_i\|^2 \right] = \E \left[\frac{1}{n}
  \sum_{i=1}^n \|\gamma(X_i)\|^2 \right] . 
\end{align*}
To be clear, the above expectation is with respect to $X_i, i = 1,
\dots, n$ being independently distributed according to $N(\theta_i,
\Sigma_i)$. The key now is to realize that 
\begin{align*}
  \E \left[\frac{1}{n}
  \sum_{i=1}^n \|\gamma(X_i)\|^2 \right] = \int \frac{1}{n}
  \sum_{i=1}^n \|\gamma(x)\|^2 \phi_d(x, \theta_i, \Sigma_i) dx = \E^*
  \|\gamma(X)\|^2. 
\end{align*}
We have therefore seen that 
\begin{align}\label{pcin}
  \Rr_n(\hat{\theta}^*, \breve{\theta}^*) = \E^*
  \|\gamma(X)\|^2. 
\end{align}
We are now ready to prove \eqref{bodisc.eq}. By Jensen's inequality, 
\begin{align*}
 \| \gamma(x)\|^2 \leq \E^* \left[ \|(I_d - \Sigma^{-1}) (\theta - X) \|^2 \big|
  X = x \right]
\end{align*}
so that 
\begin{align*}
  \Rr_n(\hat{\theta}^*, \breve{\theta}^*) &= \E^*
  \|\gamma(X)\|^2 \\ &\leq \E^* \|(I_d - \Sigma^{-1}) (\theta - X)
                       \|^2 \\
&\leq \max_{1 \leq j \leq n} \lambda^2_{\max}(I_d - \Sigma_j^{-1}) \E^*
  \|X - \theta\|^2 \\
&\leq \left(1 - \frac{1}{\sigma^2_{\max}} \right)^2 \E^* \|X -
  \theta\|^2 \\
&= \left(1 - \frac{1}{\sigma^2_{\max}} \right)^2 \E^* \E^* \left( \|X -
  \theta\|^2 \big| (\theta, \Sigma) \right) \\
&=  \left(1 - \frac{1}{\sigma^2_{\max}} \right)^2 \E^*
  \text{tr}(\Sigma) \\ 
&= \left(1 - \frac{1}{\sigma^2_{\max}} \right)^2\frac{1}{n}
  \sum_{j=1}^n \text{tr}(\Sigma_j) \leq d  \sigma_{\max}^2  \left(1 -
  \frac{1}{\sigma^2_{\max}} \right)^2  
\end{align*}
where $\text{tr}(\cdot)$ above denotes matrix trace. This proves
\eqref{bodisc.eq}. 

To prove \eqref{bodisc.peq}, let us start with \eqref{pcin}. Under the
assumption that $\theta_1 = \dots = \theta_n = \theta_0$ and $\Sigma_1
= \dots = \Sigma_n = \sigma_{\max}^2 I_d$, the probability measure
$\bar{G}_n^*$ becomes a Dirac measure concentrated at $(\theta_0,
\sigma_{\max}^2 I_d)$. As a result, the posterior distribution of
$(\theta, \Sigma)$ given $X$ is also concentrated at
$(\theta_0,\sigma_{\max}^2 I_d)$ so that 
\begin{align*}
  \gamma(x) = (I_d - \sigma^{-2}_{\max} I_d) (\theta_0 - x) = \left(1
  - \frac{1}{\sigma^2_{\max}} \right) (\theta_0 - x)
\end{align*}
and thus, from \eqref{pcin}, 
\begin{align*}
  \Rr_n(\hat{\theta}^*, \breve{\theta}^*) = \left(1 -
  \frac{1}{\sigma^2_{\max}} \right)^2 \E^* \|\theta_0 - X\|^2. 
\end{align*}
Note finally that under the distribution underlying $\E^*$, the random
vector $X$ has the $N(\theta_0, \sigma^2_{\max} I_d)$
distribution. This immediately gives \eqref{bodisc.peq}. 

\subsection{Proof of Theorem \ref{fire}} 
By the elementary inequality $(a + b)^2 \leq 2 a^2 + 2 b^2$, we get
\begin{align*}
  \Rr_n(\hat{\theta}, \hat{\theta}^*) \leq 2 \Rr_n(\hat{\theta},
  \breve{\theta}^*) + 2 \Rr_n(\breve{\theta}^*, \hat{\theta}^*). 
\end{align*}
We then use inequality \eqref{bodisc.eq} to bound the second term
above. For the first term, we use Theorem \ref{rgende} to prove the
first assertion of Theorem \ref{fire} and Proposition \ref{pkde.h} to
prove the second assertion of Theorem \ref{fire}. This finishes the
proof of Theorem \ref{fire}. 

\subsection{Proof of Lemma \ref{bsep}}
  Let $\phi_d(\dot, \mu, \Sigma)$ be the $d$-variate normal density
  with mean vector $\mu$ and covariance matrix $\Sigma$. Then, as $X_i
  \sim N(\theta_i, \Sigma_i)$, it is easy to see that (recall the
  definition of the risk $\Upsilon(T)$ from \eqref{jri})
  \begin{align*}
\Upsilon(T) &= \int \frac{1}{n} \sum_{i=1}^n \norm{T(x) - \theta_i}^2
              \phi_d(x, \theta_i, \Sigma_i) dx \\
&= \int \frac{1}{n} \sum_{i=1}^n \norm{T(x) - T^*(x) + T^*(x) - \theta_i}^2
              \phi_d(x, \theta_i, \Sigma_i) dx \\
&= \int \frac{1}{n} \sum_{i=1}^n \norm{T(x) - T^*(x)}^2 \phi_d(x,
  \theta_i, \Sigma_i) dx + \Upsilon(T^*) \\ &+ \int \frac{2}{n}
                                              \sum_{i=1}^n (T(x) -
                                              T^*(x)) (T^*(x) -
                                              \theta_i) \phi_d(x,
                                              \theta_i, \Sigma_i) dx
    \\
&\geq \Upsilon(T^*) + \int \frac{2}{n} (T(x) -
                                              T^*(x))
                                              \sum_{i=1}^n (T^*(x) -
                                              \theta_i) \phi_d(x,
                                              \theta_i, \Sigma_i) dx. 
  \end{align*}
Now by the definition of $T^*$ in \eqref{bsep.eq}, it is easy to see
that 
\begin{align*}
  \sum_{i=1}^n (T^*(x) - \theta_i) \phi_d(x, \theta_i, \Sigma_i)  = 0
  \qt{for every $x \in \R^d$}. 
\end{align*}
This proves $\Upsilon(T) \geq \Upsilon(T^*)$ for every $T$ thereby
completing the proof of Lemma \ref{bsep}.

\section{Main Metric Entropy Results and Proofs}\label{mmps}
The goal of this section is to prove our main metric entropy result,
Theorem \ref{mm}, and its following corollary which involves a
pseudometric that is defined in terms of both $f(x)$ and $\nabla
f(x)$. This corollary was used in the proof of Theorem \ref{rgende}.    

\begin{corollary}\label{rgnt}
For a compact set $S \in \R^d$ and $\rho > 0$, define the
pseudometric: 
\begin{equation}\label{tfrm}
  \norm{f - g}_{S, \nabla}^{\rho} := \sup_{x \in S} \norm{\frac{\nabla
    f(x)}{\max(f(x), \rho)} - \frac{\nabla g(x)}{\max(g(x), \rho)}} 
\end{equation}
for functions $f : \R^d \rightarrow \R$  which are bounded on $S$ and
whose derivatives are bounded on $S$. Let the $\epsilon$-covering number
of $\M$ in the pseudometric given by \eqref{tfrm} be denoted  by
$N(\epsilon, \M,  \norm{\cdot}_{S,  \nabla}^{\rho})$. Then there
exists a positive constant $C_d$ depending on $d$ alone such 
that for every $\rho > 0$,  $0 < \eta \leq \frac{2 \sqrt{2\pi}}{(2\pi)^{d/2}
    \sqrt{e}}$ and compact subset $S \subseteq \R^d$, we have 
  \begin{equation}\label{rgnt.eq}
    \log N(\eta^*, \M, \norm{\cdot}_{S, \nabla}^{\rho}) \leq C_d N(a,
    S^a) |\log \eta|^{d+1}
  \end{equation}
 where $a$ is defined as in \eqref{eq:define_a} and 
 \begin{equation}\label{etst}
   \eta^* :=  \left(\frac{1}{\rho} + \sqrt{\frac{1}{\rho^2} \log \frac{1}{(2 \pi)^d
         \rho^2}} \right) \eta.  
 \end{equation}
\end{corollary}

\begin{remark}\label{premet}
	When $d = 1$ and $S = [-M, M]$, we have
	\begin{equation*}
	N(a, S^a) \leq C \max\left\{ \frac{M}{\sqrt{|\log \eta|}} , 1 \right\}
	\end{equation*} 
	so that inequalities \eqref{mm.eq_f}, \eqref{mm.eq_grad_f} and
        \eqref{rgnt.eq} become
	\begin{equation} \label{eq:zhang_entropy_f}
	\log N(\eta, \M, \norm{\cdot}_{[-M,M]}) \leq C  |\log \eta|^2
        \max\left\{ \frac{M}{\sqrt{|\log \eta|}} , 1 \right\}, 
	\end{equation}
	\begin{equation} \label{eq:zhang_entropy_grad_f}
	\log N(\eta, \M, \norm{\cdot}_{[-M,M], \nabla}) \leq C |\log
        \eta|^2 \max\left\{ \frac{M}{\sqrt{|\log \eta|}} , 1
        \right\}, 
	\end{equation}
  and
  \begin{equation}\label{rgnt.Zhang}
    \log N(\eta^*, \M, \norm{\cdot}_{[-M, M], \nabla}^{\rho}) \leq C
    |\log \eta|^2 \max\left\{ \frac{M}{\sqrt{|\log \eta|}} , 1
        \right\} 
  \end{equation}
	respectively. Inequality \eqref{eq:zhang_entropy_f} has
        previously appeared in \citet[Lemma 2]{zhang2009generalized}
        (improving an earlier result of
        \citet{ghosal2007posterior}). Inequality
        \eqref{eq:zhang_entropy_grad_f} does not seem to have been
        stated explicitly previously but is implicit in \citet[Proof
        of Proposition $3$]{jiang2009general}. Inequality
        \eqref{rgnt.Zhang} has previously appeared as
        \citet[Proposition $3$]{jiang2009general}.  Our contribution 
        therefore lies in generalizing these results to multiple
        dimensions and further in allowing $S$ to take the form of any
        compact subset of $\R^d$. 
\end{remark}

The rest of this section is devoted to the proofs of Theorem \ref{mm}
and Corollary \ref{rgnt}.

\subsection{Proof of Theorem \ref{mm}}
The proof of Theorm \ref{mm} is given in
Subsection \ref{mmoo} after stating and proving two ingredients in the
next two subsections. 

\subsubsection{Moment Matching Lemma}
Recall that for $x \in \R^d$ and $a > 0$, we denote the closed
Euclidean ball of radius $a$ centered at $x$ by $B(x, a)$. We also
let $$\ob(x, a) := \{u \in \R^d : \|u - x \| < a\}$$ denote the open
ball of radius $a$ centered at $x$.  
\begin{lemma}\label{lemma:moment_matching}
	Let $G$ and $G'$ be two arbitrary probability measures on $\R^d$. Fix
	$a \ge 1$ and $x \in \R^d$. Let $A$ be a subset of $\R^d$ such that 
	\begin{equation*}
	\ob(x, a) \subseteq A \subseteq B(x, ca)
	\end{equation*}
	for some $c \geq 1$. Suppose that for some $m \geq 1$, we have 
	\begin{equation}\label{mmr}
	\int_A \theta_1^{k_1} \dots \theta_d^{k_d} dG(\theta) = \int_A
        \theta_1^{k_1} \dots \theta_d^{k_d} dG'(\theta)
	\end{equation}
        for every $k_1, \dots, k_d \in \{0, 1, \dots, 2m+1\}$. Then  
	\begin{equation}\label{mml.eq1}
	\left|f_G(x) - f_{G'}(x) \right| \leq \frac{1}{(2 \pi)^{(d+1)/2}}
	\left(\frac{c^2 a^2 e}{2(m+1)} \right)^{m+1} +
	\frac{e^{-a^2/2}}{(2 \pi)^{d/2}}. 
	\end{equation}
	and
	\begin{equation}\label{mml.eq2}
	\norm{\nabla f_G(x) - \nabla f_{G'}(x)}
	\leq \frac{ca}{(2 \pi)^{(d+1)/2}}
	\left(\frac{c^2 a^2 e}{2(m+1)} \right)^{m+1} + \frac{ae^{-a^2/2}}{(2 \pi)^{d/2}}
	\end{equation}
\end{lemma}

\begin{proof}[Proof of Lemma \ref{lemma:moment_matching}]
	First write
	\begin{equation*}
	f_G(x) - f_{G'}(x) = \int \phi_d(x - \theta) \left(G(d\theta)
          - G'(d\theta)\right)
        \end{equation*}
    and 
      \begin{equation*}
\nabla f_G(x) - \nabla f_{G'}(x) = \int (\theta - x) \phi_d(x -
\theta) \left(G(d\theta) - G'(d\theta)\right). 
\end{equation*}
	We split each integral above into two terms by
	restricting their range first over $A$ and then over $A^c$, the
	complement set of $A$: 
	\begin{equation}
         \begin{split}
	f_G(x) - f_{G'}(x) &= \int_A \phi_d(x - \theta) \left(dG(\theta) -
	dG'(\theta) \right) \\ &+ \int_{A^c} \phi_d(x - \theta) \left(dG(\theta) -
	dG'(\theta) \right) \label{eq:mml_f}
       \end{split}
      \end{equation}
     and 
     \begin{equation}
       \begin{split}
	\nabla f_G(x) - \nabla f_{G'}(x) &= \int_A (\theta - x) \phi_d(x - \theta) \left(dG(\theta) -
	dG'(\theta) \right) \\ &+ \int_{A^c} (\theta - x) \phi_d(x - \theta) \left(dG(\theta) -
	dG'(\theta) \right) \label{eq:mml_gradf}	
        \end{split}
	\end{equation}
	Because $A \supseteq \ob(x, a)$, it is clear that
	\begin{equation*}
         \begin{split}
	\sup_{\theta \in A^c} \phi_d(x - \theta) &\leq \sup_{\theta :
          \norm{x - \theta} \geq a} \phi_d(x - \theta) \leq (2
        \pi)^{-d/2}\exp(-a^2/2). 
         \end{split}
       \end{equation*}
and 
\begin{equation*}
   \begin{split}
        	\sup_{\theta \in A^c} \norm{\theta - x} \phi_d(x -
                \theta) &\leq \sup_{\theta : \norm{x - \theta} \geq a}
                \norm{x - \theta} \phi_d(x - \theta) \\ &\leq (2
                \pi)^{-d/2} \sup_{u \geq a} u e^{-u^2/2} = (2
                \pi)^{-d/2} a e^{-a^2/2} 
    \end{split}
	\end{equation*}
	because $a \geq 1$. Therefore the second terms on the right
        hand side on \eqref{eq:mml_f} and \eqref{eq:mml_gradf} are
        respectively bounded in absolute value by the final terms in
        \eqref{mml.eq1} and \eqref{mml.eq2}. It only remains to prove
        the following pair of inequalities 
	\begin{align}
	\left| \int_A \phi_d(x - \theta) \left(dG(\theta) - dG'(\theta) \right) \right| &\leq \frac{1}{(2 \pi)^{(d+1)/2}}
	\left(\frac{c^2 a^2 e}{2(m+1)} \right)^{m+1} \label{rp.27} \\ 
	\norm{\int_A (\theta - x) \phi_d(x - \theta) \left(dG(\theta) - dG'(\theta) \right)} &\leq \frac{ca}{(2 \pi)^{(d+1)/2}}
	\left(\frac{c^2 a^2 e}{2(m+1)} \right)^{m+1} \label{rp.272}	
	\end{align}
	For this, we use Taylor expansion and the moment matching condition
	\eqref{mmr}. Taylor's formula for $e^u$ is 
	\begin{equation*}
	e^{u} = \sum_{i=0}^m \frac{u^i}{i !} + \frac{u^{m+1}}{(m+1)!} e^v
	\end{equation*}
	for every $u$ where $v$ is some real number lying between $0$ and
	$u$. Using this for $u = -t^2/2$, we obtain 
	\begin{equation*}
	\exp(-t^2/2) = \sum_{i=0}^m \frac{(-t^2/2)^i}{i!} + (-1)^{m+1}
	\frac{(t^2/2)^{m+1}}{(m+1)!} e^v
	\end{equation*}
	where $v$ lies between $0$ and $-t^2/2$. Because $e^v \leq 1$, this
	gives 
	\begin{equation*}
	\left|\exp(-t^2/2) - \sum_{i=0}^m \frac{(-t^2/2)^i}{i!}  \right|
	\leq \frac{(t^2/2)^{m+1}}{(m+1)!}. 
	\end{equation*}
	We can therefore write $\phi_d(z) = P_d(z) + R_d(z)$ for every $z \in \R^d$
	where $P_d(z)$ is a polynomial of degree $2m$ in $z$ and $R_d(z)$ is
	a remainder term which satisfies
	\begin{equation*}
	|R_d(z)| \leq \frac{(\|z\|^2/2)^{m+1}}{(2 \pi)^{d/2} (m+1)!}.  
	\end{equation*}
	Using this for $z = x - \theta$, we can write 
	\begin{equation*}
        \begin{split}
	\left|\int_A \phi_d(x - \theta) \left(dG(\theta) - dG'(\theta) \right) \right| &\leq
	\left|\int_A P_d(x - \theta) \left(dG(\theta) - dG'(\theta)
          \right)  \right| \\ &+
	\left|\int_A R_d(x - \theta) \left(dG(\theta) - dG'(\theta) \right)
	\right|  
        \end{split}
	\end{equation*}
	and similarly,
	\begin{align*}
	& \norm{\int_A (\theta - x)\phi_d(x - \theta) \left(dG(\theta)
          - dG'(\theta) \right)} \\ &\leq
	\norm{\int_A (\theta - x)P_d(x - \theta) \left(dG(\theta) -
                                                                                                     dG'(\theta) \right)} \\ &+
	\norm{\int_A (\theta - x)R_d(x - \theta) \left(dG(\theta) -
                                                                                                                               dG'(\theta)
                                                                                                                               \right)}.   
	\end{align*}	
	The first terms in the above two equations are zero because of condition \eqref{mmr} and
	the fact that $P_d(x - \theta)$ is a polynomial in $\theta$ with
	degree $2m$ (implying that for every $j$, $(\theta_j - x_j)P_d(x - \theta)$ is a polynomial of degree $2m+1$). Because $A \subseteq B(x, ca)$, we have $\|x -
	\theta\| \leq ca$ for every $\theta \in A$ so that 
	\begin{equation*}
	|R_d(x -\theta)| \leq \frac{(2 \pi)^{-d/2}}{(m+1)!}
	\left(\frac{\|x - \theta\|^2}{2} \right)^{m+1} \leq \frac{(2
		\pi)^{-d/2}}{(m+1)!} \left(\frac{c^2 a^2}{2} \right)^{m+1}. 
	\end{equation*}
	Stirling's formula $n! \geq \sqrt{2 \pi n} (n/e)^n \geq \sqrt{2 \pi}
	(n/e)^n$ applied to $n = m+1$ yields
	\begin{equation*}
	|R_d(x - \theta)| \leq \frac{1}{(2 \pi)^{(d+1)/2}}
	\left(\frac{c^2 a^2 e}{2(m+1)} \right)^{m+1} \qt{for every
		$\theta \in A$}
	\end{equation*}
	and
	\begin{equation*}
	\norm{\theta - x}|R_d(x - \theta)| \leq \frac{ca}{(2 \pi)^{(d+1)/2}}
	\left(\frac{c^2 a^2 e}{2(m+1)} \right)^{m+1} \qt{for every
		$\theta \in A$}
	\end{equation*}
	which completes the proof. 
\end{proof}

\subsubsection{Approximation by mixtures with discrete mixing measures}
Given any distribution $f_G$, what is a bound on $\ell$ such that we
can approximate $f_G$ by another gaussian mixture $f_{G'}$ where $G'$
is a discrete measure with at most $\ell$ atoms. The following lemma
addresses this question where approximation is in terms of the
pseudometrics $ \sup_{x \in S} \left|f_{G}(x) - f_{G'}(x) \right| $ as
well as $ \sup_{x \in S} \norm{\nabla f_{G}(x) - \nabla f_{G'}(x)}
$. 

Recall that for a subset $S$ of $\R^d$, we write $N(\eta, S)$ to mean
its $\eta$ covering number (defined as the smallest number of closed
balls of radius $\eta$ whose union contains $S$). 

\begin{lemma}\label{lapp}
	Let $G$ be an arbitrary probability measure on $\R^d$ and let $S$
	denote an arbitrary compact subset of $\R^d$. Also let $a \geq 1$. Then
	there exists a discrete probability measure $G'$ that is supported
	on $S^a := \cup_{x \in S} B(x, a)$ and having at most 
	\begin{equation}\label{lde}
	\ell  :=  \left(2 \lfloor (13.5) a^2 \rfloor + 2 \right)^d N(a, S^a) + 1
	\end{equation}
	atoms such that
	\begin{equation} \label{lapp.eq_f}
	\sup_{x \in S} \left|f_{G}(x) - f_{G'}(x) \right| \leq \left(1 +
	\frac{1}{\sqrt{2 \pi}} \right)  (2 \pi)^{-d/2} e^{-a^2/2}  
	\end{equation}
	and
	\begin{equation} \label{lapp.eq_gradf}
	\sup_{x \in S} \norm{\nabla f_{G}(x) - \nabla f_{G'}(x)} \leq \left(a + \frac{3a}{\sqrt{2 \pi}} \right) (2\pi)^{-d/2} e^{-a^2/2}.
	\end{equation}
\end{lemma}

\begin{proof}[Proof of Lemma \ref{lapp}]
	Let $\mathring{S}^a := \cup_{x \in S} \ob(x, a)$ (here $\ob(x,
        a)$ denotes the open ball of radius $a$ centered at $x$) and
        let $L := N(a, \mathring{S}^a)$ denote the $a$-covering 
	number of $\mathring{S}^a$. Note that $ L \leq N(a, S^a)$. Let  $B_1, 
	\dots, B_L$ denote closed balls of radius $a$ whose union contains
	$\mathring{S}^a$. Let $E_1, \dots, E_L$ denote the standard
	disjointification of the sets $B_1, \dots, B_L$ i.e., $E_1 := B_1$
	and $E_i := B_i \setminus \left(\cup_{j < i} B_j \right)$ for $i = 2, \dots, L$. We can
	also ensure that $\cup_{i=1}^L E_i = \mathring{S}^a$ by removing
	the set $\mathring{S}^a \setminus \cup_i E_i$ from each set $E_i$.   
	
	Let $m := \lfloor (13.5) a^2 \rfloor$, suppose that a probability
	measure $G'$ is 
	chosen so that $G$ and $G'$ have the same moments up to order $2m+1$
	on each set $E_i$ for $i = 1, \dots, L$ i.e., 
	\begin{equation}\label{gre}
	\int_{E_i} \theta_j^k dG(\theta) = \int_{E_i} \theta_j^k
	dG'(\theta) 
	\end{equation}
        for $1 \leq j \leq d, 0 \leq k \leq 2m+1$ and $1 \leq i \leq
        L$.  We shall then prove below that inequalities
        \eqref{lapp.eq_f} and \eqref{lapp.eq_gradf} are 
	satisfied. Fix $x \in S$. Because $\ob(x, a)$ is contained in 
	$\mathring{S}^{a}$, the sets $E_1, \dots, E_L$ cover $\ob(x, a)$
	i.e.,   
	\begin{equation*}
	\ob(x, a) \subseteq \cup_{i \in F} E_i
	\end{equation*}
	where $F := \left\{1 \leq i \leq L : E_i \cap \ob(x, a) \neq
	\emptyset \right\}$. Also because the diameter of $E_i \subseteq
	B_i$ is at most $2a$, we deduce that 
	\begin{equation*}
	\ob(x, a) \subseteq \cup_{i \in F} E_i \subseteq B(x, 3a). 
	\end{equation*}
	We now use Lemma \ref{lemma:moment_matching} with $A = \cup_{i \in F} E_i$ and $c = 3$ to deduce that
	\begin{align*}
	|f_G(x) - f_{G'}(x)| &\leq \frac{1}{\sqrt{2\pi}} \frac{1}{(2
		\pi)^{d/2}} \left(\frac{9 a^2 e}{2(m+1)} \right)^{m+1} +
	\frac{e^{-a^2/2}}{(2 \pi)^{d/2}} \\
	\norm{\nabla f_G(x) - \nabla f_{G'}(x)}
	&\leq \frac{3a}{\sqrt{2\pi} (2 \pi)^{d/2}}
	\left(\frac{9 a^2 e}{2(m+1)} \right)^{m+1} + \frac{ae^{-a^2/2}}{(2 \pi)^{d/2}} 
	\end{align*}
	Because $m := \lfloor 13.5 a^2 \rfloor$, we have $m+1 \geq 13.5 a^2$
	and consequently,  
	\begin{equation*}
	\left(\frac{9 a^2 e}{2(m+1)} \right)^{m+1} \leq \left(\frac{e}{3}
	\right)^{m+1} \leq \exp \left(- \frac{m+1}{12} \right) \leq \exp
	\left(\frac{-27 a^2}{24} \right) \leq e^{-a^2/2}
	\end{equation*}
	where we have also used that $(e/3)^6 \leq e^{-1/2}$. This proves both inequalities \eqref{lapp.eq_f} and \eqref{lapp.eq_gradf} . 
	
	It therefore remains to prove that a discrete probability $G'$
	satisfying \eqref{gre} can be chosen with at most $\ell$ atoms where
	$\ell$ is given by \eqref{lde}. This is guaranteed by Caratheodory's
	theorem as argued below. Let $\Ps(\R^d)$ denote the collection of all
	probability measures on $\R^d$ and let 
	\begin{equation*}
         \begin{split}
	T := \left\{ \left(\int \theta_1^{k_1} \dots \theta_d^{k_d} \{\theta \in E_i\}
            dG(\theta) \right. \right. & \left., 0
	\leq k_1, \dots, k_d \leq 2m+1, 1 \leq i \leq L\right)  \\  & \left.: G \in  \Ps(\R^d)
	\right\}. 
         \end{split}
	\end{equation*}
	This set $T$ is clearly a convex subset of $\R^p$ for $p :=
	(2m+2)^dL$. Moreover, it is easy to see that $T$ is simply the convex
	hull of   
	\begin{equation*}
	C := \left\{\left(\theta_1^{k_1} \dots \theta_d^{k_d} \{\theta \in E_i\} , 0
	\leq k_1, \dots, k_d \leq 2 m + 1, 1 \leq i \leq L\right) : \theta \in S^a \right\}. 
	\end{equation*}
	Therefore, by Caratheodory's theorem, every element of $T$ can be
	written as a convex combination of at most $p + 1$ elements of
	$C$. We therefore take $G'$ to be the discrete probability measure supported upon these elements with probabilities given by the weights of this convex combination. Note that the
	number of atoms of $G'$ is bounded from above by $\ell$  given in
	\eqref{lde}. It is also easy to see that $G'$ is supported on
	$S^a$. This completes the proof.  
\end{proof}

\subsubsection{Proof of Theorem \ref{mm}}\label{mmoo}

\begin{proof} Fix $ 0 < \eta \leq \frac{2 \sqrt{2\pi}}{(2\pi)^{d/2}
    \sqrt{e}}$ and define $a$ as in \eqref{eq:define_a}. Note that $a
  \geq 1$. Fix $G \in \G$. According to Lemma \ref{lapp}, 
	there exists a discrete probability measure $G'$ supported on $S^a$
	and having $\ell$ atoms (with $\ell$ as in \eqref{lde}) such that:
	\begin{equation}
	\sup_{x \in S} \left|f_{G}(x) - f_{G'}(x) \right| \leq \left(1 +
	\frac{1}{\sqrt{2 \pi}} \right)  (2 \pi)^{-d/2} e^{-a^2/2}
	\end{equation}
	and
	\begin{equation}
	\sup_{x \in S} \norm{\nabla f_{G}(x) - \nabla f_{G'}(x)} \leq \left(a + \frac{3a}{\sqrt{2 \pi}} \right) (2\pi)^{-d/2} e^{-a^2/2}.
	\end{equation}
	Let $\alpha > 0$ and let $s_1, \dots, s_D$ be an
        $\alpha-$cover of 
	$S^a$ (i.e., $\sup_{s \in S^a} \inf_{1 \leq i \leq D} \|s - s_i \| \leq
	\alpha$) with $D = N(\alpha, S^a)$. Now if $G' = \sum_{i=1}^{\ell} w_i
	\delta_{a_i}$ (for some probability vector $(w_1, \dots, w_{\ell})$ and atoms
	$a_1, \dots, a_{\ell} \in S^a$), then let $G'' := \sum_{i=1}^{\ell}
	w_i \delta_{b_i}$ where $b_i \in \{s_1, \dots, s_D\}$ and $\|a_i -
	b_i\| \leq \alpha$. Then, for every $x \in S$, 
	\begin{align*}
	\left|f_{G'}(x) - f_{G''}(x) \right| &= \left|\sum_{i=1}^{\ell} w_i
	\phi_d(x - a_i) -
	\sum_{i=1}^{\ell} w_i
	\phi_d(x - b_i)  \right| \\
	&\leq \sum_{i=1}^{\ell} w_i \left|\phi_d(x - a_i)  - \phi_d(x - b_i)
	\right| \\
	&\leq \sum_{i=1}^{\ell} w_i \sup_{t} \|\nabla \phi_d(t) \| \alpha \\
	& \leq \alpha \sup_{t} \|\nabla \phi_d(t) \| \\ &= \alpha (2 \pi)^{-d/2}
	\sup_{t} \|t\| e^{-\|t\|^2/2} = \alpha (2 \pi)^{-d/2} e^{-1/2}. 
	\end{align*}
	We shall now bound $ \norm{\nabla f_{G'}(x) - \nabla
          f_{G''}(x)} $ using similar arguments. By the mean value
        theorem, there exists $u_i$ on the line segment joining $x-a_i$
        and $x-b_i$ such that, 
	\begin{equation*}
	\phi_d (x - b_i) = \phi_d (x - a_i) + (a_i - b_i)^{\top} \nabla \phi_d (u_i)
	\end{equation*}
	and consequently
	\begin{equation*}
	x - b_i = u_i + \zeta_i \qt{ for some } \zeta_i \text{ satisfying } \|\zeta_i\| \leq \alpha.
	\end{equation*}
	Similarly,
	\begin{align*}
	\norm{\nabla f_{G'}(x) - \nabla f_{G''}(x)} &=
                                                      \sum_{i=1}^{\ell}
                                                      w_i \norm{
                                                      \nabla \phi_d(x
                                                      - a_i) - \nabla
                                                      \phi_d(x - b_i)
                                                      } \\
&=
                                                      \sum_{i=1}^{\ell}
                                                      w_i \left\| (a_i
                                                      - x) \phi_d (x -
                                                      a_i) - (b_i - x)
                                                      \phi_d (x - b_i)
                                                      \right\| \\ 
	&= \sum_{i=1}^{\ell} w_i \left\| (b_i - a_i) \phi_d (x - a_i) \right.
          \\ & \left.+ (u_i + \zeta_i) \left[ (a_i - b_i)^{\top} \nabla \phi_d (u_i) \right] \right\| \\
	&\leq \alpha \sup_t \phi_d(t) + \alpha \sup_t (\norm{t} + \alpha) \norm{ \nabla \phi_d (t) } \\
	&\leq \alpha \sup_t \phi_d(t) + \alpha \sup_t \norm{t}^2 \phi_d(t) + \alpha^2 \sup_t \norm{t} \phi_d(t) \\
	&=  \frac{\alpha}{(2\pi)^{d/2}} \left[ 1 + \frac{2}{e} + \alpha \frac{1}{\sqrt{e}} \right] 
	\end{align*}
	
If $G''' := \sum_{i=1}^{\ell} w_i' \delta_{b_i}$ for another
	probability vector $w' := (w'_1, \dots, w'_{\ell})$, then clearly
	\begin{align*}
	\left|f_{G'}(x) - f_{G''}(x) \right| &= \left|\sum_{i=1}^{\ell} (w_i
	- w_i') \phi(x - b_i) \right| \leq (2 \pi)^{-d/2}
	\sum_{i=1}^{\ell} |w_i - w_i'|
        \end{align*}
and 
\begin{align*}
	\norm{\nabla f_{G'}(x) - \nabla f_{G''}(x)} &= \norm{\sum_{i=1}^{\ell} (w_i	- w_i') \nabla \phi(x - b_i) } \\
	&\leq \sum_{i=1}^{\ell} |w_i - w_i'| \left[ \sup_{t} \norm{t
          \phi_d(t)} \right] \\ &= (2\pi)^{-d/2} e^{-1/2}
                                  \sum_{i=1}^{\ell} |w_i - w_i'|. 
	\end{align*}
	Therefore if $\sum_{i=1}|w_i - w_i'| \leq v$, then
	\begin{align*}
	\sup_{x \in S}|f_G(x) - f_{G'''}(x)| &\leq \left(1 + \frac{1}{\sqrt{2 \pi}}
	\right) (2 \pi)^{-d/2} e^{-a^2/2} \\ &+ \alpha (2 \pi)^{-d/2}
                                               e^{-1/2} + (2
                                               \pi)^{-d/2} v 
         \end{align*}
and
\begin{align*}
	\sup_{x \in S}\norm{\nabla f_G(x) - \nabla f_{G'''}(x) } &\leq \left(a + \frac{3a}{\sqrt{2 \pi}}
	\right) (2 \pi)^{-d/2} e^{-a^2/2} \\ &+ \alpha (2 \pi)^{-d/2}
                                                                   \left[
                                                                   1 +
                                                                   \frac{2}{e}
                                                                   +
                                                                   \alpha
                                                                   \frac{1}{\sqrt{e}}
                                                                   \right]
                                               \\ &+ (2\pi)^{-d/2} e^{-1/2}v. 
	\end{align*}
	By choosing 
	\begin{equation*}
	v = \alpha = \frac{(2 \pi)^{d/2}}{2 \sqrt{2 \pi}} \eta ~~ \text{ and } ~~ a = \sqrt{2
		\log \frac{2 \sqrt{2 \pi}}{(2 \pi)^{d/2} \eta}} = \sqrt{2 \log \frac{1}{\alpha}}, 
	\end{equation*}
	we obtain
	\begin{align*}
	\sup_{x \in S}|f_G(x) - f_{G'''}(x)| &< \frac{\alpha}{(2\pi)^{d/2}} \left[ 2 + \frac{1}{\sqrt{2 \pi}} + \frac{1}{\sqrt{e}} \right] < \eta \\
	\sup_{x \in S}\norm{\nabla f_G(x) - \nabla f_{G'''}(x) } &\leq \frac{a \alpha}{(2\pi)^{d/2}} \left[ 2 + \frac{3}{\sqrt{2\pi}} + \frac{3}{e} + \frac{1}{\sqrt{e}} \right] < a \eta
	\end{align*}
	where we have noted that $a \geq 1$ and $\alpha \leq e^{-1/2}$.
	
	It only remains to count the number of ways of choosing the discrete
	probability measure $G'''$. The number of ways of choosing the atoms
	of $G'''$ is clearly
	\begin{equation*}
	{D \choose \ell} \leq \frac{D^{\ell}}{\ell !} \leq \left(\frac{D
		e}{\ell} \right)^{\ell}
	\end{equation*}
	where we used that $\ell ! \geq (\ell/e)^{\ell}$, a fact that follows
	from Stirling's formula.

        The probability vector $w' = (w_1', \dots, w_{\ell}')$ can be chosen
	to belong to a $v$-covering set for all $\ell$-dimensional probability
	vectors under the $L^1$ norm. This covering number is well known to be
	at most: $(1 + (2/v))^{\ell}$. Therefore the nunber of ways of
        choosing $G'''$ is bounded from above by:
	\begin{equation*}
	\left[\frac{D e}{\ell} \left(1 + \frac{2}{v} \right) \right]^{\ell}
	= A^{\ell} \qt{where $A := \frac{D e}{\ell} \left(1 + \frac{2}{v}
		\right)$}. 
	\end{equation*}
	We shall bound $A$ below. Below $C_d$ will denote a constant that
	depends on $d$ alone. Because $ v\leq e^{-1/2}$,
	\begin{equation*}
	1 + \frac{2}{v} \leq \left(\frac{1}{\sqrt{e}} +2  \right)
	\frac{1}{v} = \frac{C_d}{\eta}. 
	\end{equation*}
	Also note that from the expression for $\ell$ given in \eqref{lde}, we
	have $\ell \geq N(a, S^a)$ and hence
	\begin{equation*}
	\frac{D}{\ell} \leq \frac{N(\alpha, S^a)}{N(a, S^a)} \leq N(\alpha,
	B(0, a)) \leq \left(1 + \frac{a}{\alpha} \right)^d \leq C_d
	\left(\frac{1}{\eta}  \right)^{3d/2}. 
	\end{equation*}
	where we have used the trivial fact that
	\begin{equation} \label{eq:crude_bound_on_a}
	a = \sqrt{2 \log \frac{1}{\alpha}} \leq \sqrt{ \frac{4}{\alpha}} = C_d \frac{1}{\sqrt{\eta}}.
	\end{equation}
	We thus have 
	\begin{equation*}
	A \leq C_d \eta^{-1 - 3d/2}
	\end{equation*}
	so that,
	\begin{equation*}
	\log N(\eta, \M, \norm{\cdot}_S) \leq \ell \log A \leq C_d \ell \log \frac{1}{\eta}
	\end{equation*}
	which along with the expression \eqref{lde} for $\ell$ proves \eqref{eq:mml_f}. Similarly,
	\begin{equation*}
	\log N(a\eta, \M, \norm{\cdot}_{S, \nabla}) \leq C_d \ell \log
        \frac{1}{\eta} \leq C_d N(a, S^a) | \log \eta |^{d+1}. 
	\end{equation*}
	This implies that
	\begin{equation*}
	\log N(\eta, \M, \norm{\cdot}_{S, \nabla}) \leq C_d N(a, S^a)
        \left| \log \frac{ \eta}{a} \right|^{d+1} \leq C_d N(a, S^a) |
        \log \eta |^{d+1} 
	\end{equation*}
	where the last inequality follows from \eqref{eq:crude_bound_on_a}. This completes the proof of \eqref{eq:mml_gradf} and consequently Theorem \ref{mm}. 
	
\end{proof}

\subsection{Proof of Corollary \ref{rgnt}}
\begin{proof}[Proof of Corollary \ref{rgnt}]
Fix $\rho > 0$,  $0 < \eta \leq \frac{2 \sqrt{2\pi}}{(2\pi)^{d/2}
  \sqrt{e}}$ and compact subset $S \subseteq \R^d$. For $a, b \in \R$,
we shall denote the maximum of $a$ and $b$ by $a \vee b$. Note first
that for every pair of densities $f_G, f_H \in \M$ and $x \in S$, we
have  
\begin{align*}
\norm{ \frac{\nabla f_G(x)}{\rho \vee f_G(x)} - \frac{\nabla
  f_H(x)}{\rho \vee f_H(x)} } &= \left\| \frac{\nabla f_G(x)}{\rho \vee
                                f_G(x)} - \frac{\nabla f_G(x)}{\rho
                                \vee f_H(x)} \right. \\ & \left.+ \frac{\nabla
                                f_G(x)}{\rho \vee f_H(x)} -
                                \frac{\nabla f_H(x)}{\rho \vee f_H(x)}
                                \right\| \\ 
&\leq \frac{\norm{\nabla f_G(x)}}{\rho \vee f_G(x)} \frac{|\rho \vee
  f_G(x) - \rho \vee f_H(x)|}{\rho \vee f_H(x)} \\ &+ \frac{1}{\rho}
  \norm{ \nabla f_G(x) - \nabla f_H(x) }  
\end{align*}
Using inequality \eqref{p1.neq} (in Lemma \ref{p1}) and the fact that
$t \mapsto \rho \vee t$ is $1$-Lipschitz, we deduce from the above
that 
\begin{align*}
  \norm{ \frac{\nabla f_G(x)}{\rho \vee f_G(x)} - \frac{\nabla
  f_H(x)}{\rho \vee f_H(x)} } &\leq \sqrt{\frac{1}{\rho^2} \log \frac{1}{(2 \pi)^d
         \rho^2}} \left|f_G(x) - f_H(x) \right| \\ &+ \frac{1}{\rho}
     \norm{f_G(x) - f_H(x)}. 
\end{align*}
Because this is true for every $x \in S$, we have 
\begin{equation*}
  \norm{f_G - f_H}_{S, \nabla}^{\rho} \leq \sqrt{\frac{1}{\rho^2} \log
    \frac{1}{(2 \pi)^d \rho^2}} \norm{f_G - f_{H}}_{S}+ \frac{1}{\rho}
  \norm{f_G - f_H}_{S, \nabla}.   
\end{equation*}
We thus have 
\begin{equation*}
 \log N(\eta^*, \M, \norm{\cdot}_S) \le 
\log N(\eta/2, \M, \norm{\cdot}_S)
  + \log N(\eta/2, \M, \norm{\cdot}_{S, \nabla})
\end{equation*}
from which \eqref{rgnt.eq} follows. 
\end{proof}

\section{Bounding Bayes Discrepancy via Hellinger
  Distance} \label{zhafo} 
The purpose of this section is to state and prove the following
theorem relating the quantity:
\begin{equation}\label{gaja}
\Gamma(G_0, G, \rho) := \left( \int \norm{\frac{\nabla
      f_{G}}{\max(f_{G}, \rho)} - \frac{\nabla
      f_{G_0}}{\max(f_{G_0}, \rho)}}^2 f_{G_0} 
\right)^{1/2}. 
\end{equation}
for $\rho > 0$ and two probability measures $G_0$ and $G$ on $\R^d$ in
terms of the squared Hellinger distance between $f_G$ and $f_{G_0}$.
This result is crucial for the proof of Theorem \ref{rgende}. 

\begin{theorem}\label{fcc}
There exists a universal positive constant $C$ such that for every
pair of probability measures $G$ and $G_0$ on $\R^d$ and $0 < \rho
\leq (2 \pi)^{-d/2} e^{-1/2}$, the quantity $\Gamma^2(G_0, G, \rho)$
(defined in \eqref{gaja}) is bounded from above by 
\begin{equation} \label{fcc.eq}
C d \max \left\{ \left(\log \frac{(2 \pi)^{-d/2}}{\rho} \right)^3,
  |\log \hel(f_G, f_{G_0})| \right\} \hel^2(f_G, f_{G_0}). 
\end{equation}
\end{theorem}

The above theorem is a generalization of \citet[Theorem
3]{jiang2009general} to the case when $d \geq 1$. Its proof given
below follows \citet[Proof of Theorem 3]{jiang2009general} with
appropriate changes to deal with the $d \geq 1$ case. Lemma \ref{p1}
and Lemma \ref{kfo} from Section \ref{auxre} will be used in this
proof. 

\begin{proof}[Proof of Theorem \ref{fcc}]
For real numbers $a$ and $b$, we denote $\max(a, b)$ by $a \vee
b$. For functions $u : \R^d \rightarrow \R^d$, we let   
\begin{equation*}
\|u \|_0 := \left(\int \norm{u(x)}^2 f_{G_0}(x) dx
\right)^{1/2} 
\end{equation*}
so that 
\begin{align*}
\Gamma(G_0, G, \rho) &= \norm{\frac{\nabla f_G}{f_G \vee \rho} -
                       \frac{\nabla 
f_{G_0}}{f_{G_0} \vee \rho}}_0 \\
&= \left\|\frac{\nabla f_G}{f_G \vee \rho} - \frac{2 \nabla f_G}{f_G
\vee \rho + f_{G_0} \vee \rho} + 2\frac{\nabla f_G - \nabla
f_{G_0}}{f_G \vee \rho + f_{G_0} \vee \rho} \right. \\ & \left. + \frac{2 \nabla
f_{G_0}}{f_G \vee \rho + f_{G_0} \vee \rho} - \frac{\nabla f_{G_0}}{f_{G_0}
\vee \rho} \right\|_0 \\
&\leq 2 \max_{H \in \{G, G_0\}} \norm{\frac{(\nabla f_H) |f_G \vee \rho
- f_{G_0} \vee \rho|}{(f_H \vee \rho) \left(f_G \vee \rho + f_{G_0}
\vee \rho \right)}}_0 \\ &+ 2 \norm{\frac{\nabla f_G - \nabla
f_{G_0}}{f_G \vee \rho + f_{G_0} \vee \rho}}_0,
\end{align*}
where we have used the triangle inequality for $\norm{}_0$ in the
last step. Let us represent the two terms on the right hand side above
by $T_1$ and $T_2$ respectively so that $\Gamma(G_0, G, \rho) \leq T_1
+ T_2$. We shall now bound $T_1$ and $T_2$ separately. For $T_1$, we
use inequality 
\eqref{p1.neq} in Lemma \ref{p1} (note that we have assumed $0 < \rho
\leq (2 \pi)^{-d/2} e^{-1/2}$). This inequality allows us to bound
$T_1$ as follows:   
\begin{align*}
\frac{1}{4}  T_1^2 &= \max_{H \in \{G, G_0\}}\int \frac{\norm{\nabla
                     f_H}^2 \left(f_G \vee \rho - 
f_{G_0} \vee \rho \right)^2}{\left(f_H \vee \rho
\right)^2 \left(f_G \vee \rho + f_{G_0} \vee \rho \right)^2}
f_{G_0} \\
&\leq \left[ \log \frac{(2 \pi)^{-d}}{\rho^2} \right] \int
\frac{\left(f_G \vee \rho - f_{G_0} \vee \rho\right)^2}{\left(f_G
\vee \rho + f_{G_0} \vee \rho \right)^2} f_{G_0}  \\
&\leq \left[ \log \frac{(2 \pi)^{-d}}{\rho^2} \right] \int
\frac{\left(f_G - f_{G_0} \right)^2}{\left(f_G 
\vee \rho + f_{G_0} \vee \rho \right)^2} f_{G_0}  \\
&= \left[ \log \frac{(2 \pi)^{-d}}{\rho^2} \right] \int
\left(\sqrt{f_G} - \sqrt{f_{G_0}} \right)^2  \frac{\left(\sqrt{f_G} + \sqrt{f_{G_0}} \right)^2}{\left(f_G 
\vee \rho + f_{G_0} \vee \rho \right)^2} f_{G_0} \\
&\leq 2 \left[ \log \frac{(2 \pi)^{-d}}{\rho^2} \right] \int
\left(\sqrt{f_G} - \sqrt{f_{G_0}} \right)^2  \frac{\left(f_G + f_{G_0}
  \right) f_{G_0}}{\left(f_G 
\vee \rho + f_{G_0} \vee \rho \right)^2}  \\
&\leq 2 \left[ \log \frac{(2 \pi)^{-d}}{\rho^2} \right] \int
\left(\sqrt{f_G} - \sqrt{f_{G_0}} \right)^2 = 2 \left[ \log \frac{(2
\pi)^{-d}}{\rho^2} \right] \hel^2(f_G, f_{G_0})
\end{align*}
which gives 
\begin{equation}\label{t1b}
T_1 \leq 2 \sqrt{2} \hel(f_G, f_{G_0}) \sqrt{\log \frac{(2
\pi)^{-d}}{\rho^2}}. 
\end{equation}
We now deal with $T_2$ which needs an elaborate
argument. Start by writing 
\begin{equation} \label{t31}
\begin{split}
\frac{1}{4} T_2^2 &= \int \frac{\norm{\nabla f_G - \nabla
  f_{G_0}}^2}{\left(f_G \vee \rho + f_{G_0} \vee
  \rho \right)^2} f_{G_0} \\
&= \int \frac{\norm{\nabla f_G - \nabla
  f_{G_0}}^2}{f_G \vee \rho + f_{G_0} \vee
  \rho}  \left( \frac{f_{G_0}}{f_G \vee \rho + f_{G_0}
\vee \rho} \right) \\ &\leq \int \frac{\norm{\nabla f_G - \nabla
  f_{G_0}}^2}{f_G \vee \rho + f_{G_0} \vee
  \rho} = \sum_{i=1}^d \Delta_{i,1}^2 
\end{split}
\end{equation}
where, for $1 \leq i \leq d$ and $k \geq 0$, 
\begin{equation*}
\Delta^2_{i,k} := \int \frac{\left( \partial_i^k (f_G - f_{G_0})
\right)^2}{f_{G} \vee \rho + f_{G_0} \vee \rho} \qt{with
$\partial_i^k f := \frac{\partial^k}{\partial x_i^k} f$}. 
\end{equation*}
The next task therefore is to bound $\Delta_{i, 1}^2$ from
above. Before dealing with $\Delta_{i, 1}^2$, let us first note that
it is easy to bound $\Delta_{i, 0}$ by the Hellinger distance between
$f_G$ and $f_{G_0}$. Indeed, we can write 
\begin{equation}\label{dio}
\begin{split}
\Delta_{i, 0}^2 &= \int \frac{(f_G - f_{G_0})^2}{f_G \vee \rho +
f_{G_0} \vee \rho} \\ &= \int \left(\sqrt{f_G} - \sqrt{f_{G_0}}
\right)^2 \frac{(\sqrt{f_G} + \sqrt{f_{G_0}})^2}{f_G \vee \rho +  
f_{G_0} \vee \rho}  \\
&\leq 2 \int \left(\sqrt{f_G} - \sqrt{f_{G_0}}
\right)^2 \frac{(f_G + f_{G_0})}{f_G \vee \rho +  
f_{G_0} \vee \rho} \leq 2 \hel^2(f_G, f_{G_0}). 
\end{split}
\end{equation}
A simple upper bound for $\Delta_{i, k}^2$ for general $k \geq 1$ can
be obtained via Lemma \ref{kfo}. Indeed, noting that (as $f_G \vee
\rho + f_{G_0} \vee \rho \geq 2 \rho$) 
\begin{equation*}
\Delta_{i, k}^2 \leq \frac{1}{2 \rho} \int \left(\partial_i^k
\left(f_G - f_{G_0} \right)\right)^2
\end{equation*}
we can apply Lemma \ref{kfo} to deduce that 
\begin{equation}\label{dik}
\Delta_{i, k}^2 \leq \frac{2 (2 \pi)^{-d/2}}{\rho} \left\{a^{2k}
\hel^2(f_G, f_{G_0}) + \sqrt{\frac{2}{\pi}} a^{2k-1} e^{-a^2}
\right\} 
\end{equation}
for every $a \geq \sqrt{2k -1}$. The problem with this bound is the
presence of $\rho$ in the 
denominator. This $\rho$ will be, in applications of Theorem
\ref{fcc}, of the order $n^{-1}$ 
which makes the above bound quite large. The more refined argument
below will get rid of the $\rho$ factor in the denominator. This
argument involves integration by parts for bounding $\Delta_{i,
1}^2$. It will be clear that the use of integration by parts will
result in expressions involving $\Delta_{i, k}^2$ for $k \geq 2$. It
will then become necessary to deal with $\Delta_{i, k}^2$ for $k \geq
2$ even though we are only interested in $\Delta_{i, 1}^2$. Indeed,
integration by parts gives, for $k \geq 1$,  
\begin{align}
\Delta_{i, k}^2 &= - \int \left[ \partial_i^{k-1} \left(f_G - f_{G_0}
\right) \right]  \left[\partial_i^k \left(f_G - f_{G_0} \right)
\right] \partial_i \left( \frac{1}{f_G \vee \rho + f_{G_0} \vee
\rho} \right) \nonumber \\
& - \int \frac{ \left[\partial_i^{k-1} \left(f_G - f_{G_0}
\right) \right] \left[\partial_i^{k+1} \left(f_G - f_{G_0} \right)
\right]}{f_G \vee \rho + f_{G_0} \vee \rho}. \label{ffr}
\end{align}
Note now that, almost surely 
\begin{align*}
\left|\partial_i \left(\frac{1}{f_G \vee \rho + f_{G_0} \vee \rho}
\right) \right| &\leq \frac{|\partial_i f_G| + |\partial_i
f_{G_0}|}{\left(f_G \vee \rho + f_{G_0} \vee \rho \right)^2} \\ 
&\leq \frac{|\partial_i f_G|/(f_G \vee \rho) + |\partial_i
f_{G_0}|/(f_{G_0} \vee \rho)}{f_G \vee \rho + f_{G_0} \vee \rho} \\
&\leq \frac{\norm{\nabla f_
G}/(f_G \vee \rho) + \norm{\nabla f_{G_0}}/(f_{G_0} \vee \rho)}{f_G
\vee \rho + f_{G_0} \vee \rho} \\
&\leq \frac{2}{f_G \vee \rho + f_{G_0} \vee \rho} \sqrt{\log \frac{(2
\pi)^{-d}}{\rho^2}}
\end{align*}
where, in the last inequality, we used \eqref{p1.neq} in Lemma
\ref{p1}. Imputing the above inequality into \eqref{ffr}, we obtain 
\begin{align*}
\Delta_{i, k}^2 &\leq 2 \sqrt{\log \frac{(2 \pi)^{-d}}{\rho^2}}\int
\frac{\left| \partial_i^{k-1} \left(f_G - f_{G_0}
\right) \right|  \left|\partial_i^k \left(f_G - f_{G_0} \right)
\right|}{f_G \vee \rho + f_{G_0} \vee \rho}  \\ &+ \int \frac{ \left|\partial_i^{k-1} \left(f_G - f_{G_0}
\right) \right| \left|\partial_i^{k+1} \left(f_G - f_{G_0} \right)
\right|}{f_G \vee \rho + f_{G_0} \vee \rho}. 
\end{align*}
Applying the Cauchy-Schwarz inequality to each of the two terms on the
right hand side above, we obtain 
\begin{align*}
\Delta_{i, k}^2 &\leq 2 \sqrt{\log \frac{(2 \pi)^{-d}}{\rho^2}}
\sqrt{\int \frac{\left( \partial_i^{k-1} (f_G - f_{G_0})
\right)^2}{f_{G} \vee \rho + f_{G_0} \vee \rho}} \sqrt{\int
\frac{\left( \partial_i^{k} (f_G - f_{G_0})  
\right)^2}{f_{G} \vee \rho + f_{G_0} \vee \rho}} \\
&+ \sqrt{\int \frac{\left( \partial_i^{k-1} (f_G - f_{G_0})
\right)^2}{f_{G} \vee \rho + f_{G_0} \vee \rho}} \sqrt{\int
\frac{\left( \partial_i^{k+1} (f_G - f_{G_0})  
\right)^2}{f_{G} \vee \rho + f_{G_0} \vee \rho}}
\end{align*}
which can be rewritten as 
\begin{equation}\label{kke}
\Delta_{i, k}^2 \leq \Upsilon 
\Delta_{i, k-1} \Delta_{i, k} + \Delta_{i, k-1} \Delta_{i, k+1}
\qt{where $\Upsilon := 2 \sqrt{\log \frac{(2
\pi)^{-d}}{\rho^2}} $}. 
\end{equation}
The strategy to bound $\Delta_{i, 1}$ is now as follows. Divide both
sides of \eqref{kke} by $\Delta_{i, k-1} \Delta_{i, k}$ to get 
\begin{equation}\label{rr}
\frac{\Delta_{i, k}}{\Delta_{i, k-1}} \leq \Upsilon +
\frac{\Delta_{i, k+1}}{\Delta_{i, k}} \qt{for every $k \geq 1$}. 
\end{equation}
Fix an integer $k_0 \ge 1$ and a real number $\beta > 0$. Our bound on
$\Delta_{i, 1}$ will depend on $k_0$ and $\beta$ and the bound will be
optimized for $k_0$ and $\beta$ at the end. 

Suppose first that there exists an integer $1 \leq k \leq k_0$ such
that $\Delta_{i, k+1} \leq \beta \Delta_{i, k}$. Then applying
\eqref{rr} recursively for $1, \dots, k$, we obtain 
\begin{equation*}
\frac{\Delta_{i, 1}}{\Delta_{i, 0}} \leq k \Upsilon + \beta
\end{equation*}
so that, by \eqref{dio}, 
\begin{equation}\label{bb1}
\begin{split}
\Delta_{i, 1} &\leq \left(k \Upsilon + \beta \right) \Delta_{i, 0} \\
&\leq \sqrt{2} \left(k \Upsilon + \beta \right) \hel(f_G, f_{G_0}) \leq
\sqrt{2} \left(k_0 \Upsilon + \beta \right) \hel(f_G, f_{G_0}). 
\end{split}
\end{equation}
Now suppose that $\Delta_{i, k+1} > \beta \Delta_{i, k}$ for every
integer $1 \le k \leq k_0$. In this case, we deduce from \eqref{rr}
that
\begin{equation*}
\frac{\Delta_{i, k}}{\Delta_{i, k-1}} \leq \Upsilon +
\frac{\Delta_{i, k+1}}{\Delta_{i, k}} \leq \left(1 +
\frac{\Upsilon}{\beta} \right) \frac{\Delta_{i, k+1}}{\Delta_{i,
k}} \qt{for every $k = 0, \dots, k_0$}. 
\end{equation*}
A recursive application of this inequality implies that 
\begin{equation*}
\frac{\Delta_{i,1}}{\Delta_{i, 0}} \leq \left(1 +
\frac{\Upsilon}{\beta} \right)^k \frac{\Delta_{i, k+1}}{\Delta_{i,
k}} \qt{for every $k = 0, \dots, k_0$}. 
\end{equation*}
To obtain a bound for $\Delta_{i, 1}/\Delta_{i, 0}$ that depends only
on $\Delta_{i, k_0 + 1}$ and $\Delta_{i, 0}$, one can take the
geometric mean of the above inequality for $k = 0, 1, \dots,
k_0$. This gives 
\begin{align*}
\frac{\Delta_{i, 1}}{\Delta_{i, 0}} &\leq \left(\prod_{k=0}^{k_0}
\left(1 + \frac{\Upsilon}{\beta} \right)^{k}  \frac{\Delta_{i,
k+1}}{\Delta_{i, k}}\right)^{1/(k_0+1)} \\ &= \left(1 +
\frac{\Upsilon}{\beta} \right)^{k_0/2} \Delta_{i,
k_0+1}^{1/(k_0+1)} \Delta_{i, 0}^{-1/(k_0 + 1)}
\end{align*}
which is same as 
\begin{equation*}
\Delta_{i, 1} \leq \left(1 +
\frac{\Upsilon}{\beta} \right)^{k_0/2} \Delta_{i,
k_0+1}^{1/(k_0+1)} \Delta_{i, 0}^{k_0/(k_0 + 1)}
\end{equation*}
Now using \eqref{dio} and the bound \eqref{dik} (with $k = k_0 + 1$),
we obtain 
\begin{equation}\label{bb2}
\begin{split}
\Delta_{i, 1} &\leq \left(1 +
\frac{\Upsilon}{\beta} \right)^{\frac{k_0}{2}} \left(\frac{2 (2
\pi)^{-d/2}}{\rho}   \left[a^{2k_0+2}
\hel^2(f_G, f_{G_0}) \right. \right. \\ 
& \left. \left.+ \sqrt{\frac{2}{\pi}} a^{2k_0+1} e^{-a^2} \right]
\right)^{\frac{1}{2k_0 + 2}} \left(2 \hel^2(f_G, f_{G_0}) \right)^{\frac{k_0}{2k_0
+ 2}}
\end{split}
\end{equation}
for every $a \geq \sqrt{2k_0 + 1}$. The final bound obtained for
$\Delta_{i, 1}$ is the maximum of the right hand side above and the
right hand side of \eqref{bb1}. This bound will need to be optimized
by choosing $k_0$, $\beta$ and $a \geq \sqrt{2k_0 + 1}$
appropriately. $\beta$ will be chosen as $\beta = k_0 \Upsilon$ so
that the bound \eqref{bb1} becomes $2 \sqrt{2} k_0 \Upsilon \hel(f_G,
f_{G_0})$ and the term $(1 + \Upsilon/\beta)^{k_0/2}$ appearing in \eqref{bb2} is
bounded by $\sqrt{e}$. To select $k_0$, the key is to focus on the
term involving $\rho$ in \eqref{bb2} which is
\begin{equation*}
\left( \frac{(2 \pi)^{-d/2}}{\rho} \right)^{1/(2k_0 + 2)} = \exp
\left(\frac{\Upsilon^2}{16(k_0 + 1)} \right). 
\end{equation*}
This suggests taking $k_0$ to be the smallest integer $\geq 1$ such
that $k_0 + 1 \geq \Upsilon^2/8$ so  that the above term is at most
$\sqrt{e}$. Finally $a$ will be taken 
to be 
\begin{equation*}
a := \max \left(\sqrt{2k_0 + 1}, \sqrt{2 \left| \log \hel(f_G, f_{G_0})
\right| } \right) 
\end{equation*}
which will ensure that $e^{-a^2} \leq \hel^2(f_G, f_{G_0})$ and the term
involving $a$ in \eqref{bb2} can then be bounded by
\begin{align*}
&\left(a^{2k_0+2} \hel^2(f_G, f_{G_0}) + \sqrt{\frac{2}{\pi}} a^{2k_0+1}
e^{-a^2} \right)^{\frac{1}{2k_0 + 2}} \\ &\leq a \left(1 + \sqrt{\frac{2}{\pi}}
\right)^{\frac{1}{2k_0 + 2}} \left(\hel(f_G, f_{G_0}) \right)^{\frac{1}{k_0 + 1}} \\
&\leq \left(1 + \sqrt{\frac{2}{\pi}} \right)a \left(\hel(f_G, f_{G_0})
  \right)^{\frac{1}{k_0 + 1}}  \\
&\leq 2a \left(\hel(f_G, f_{G_0})
  \right)^{\frac{1}{k_0 + 1}}.   
\end{align*}
We have therefore proved that the right hand side in \eqref{bb2} is
bounded from above by $2 \sqrt{2} e a \hel(f_G, f_{G_0})$. Because
$\Delta_{i,1}$ is bounded by the maximum of the bounds given by
\eqref{bb1} and \eqref{bb2}, we obtain: 
\begin{align*}
\Delta_{i, 1} &\leq 2 \sqrt{2} \max \left\{ k_0 \Upsilon, e a
\right\} \hel(f_G, f_{G_0})  \\ &\leq 2 \sqrt{2} \max \left\{k_0 \Upsilon, e
\sqrt{2k_0 + 1}, e \sqrt{2 |\log \hel(f_G, f_{G_0})|}
\right\} \hel(f_G, f_{G_0}) . 
\end{align*}
Now because $k_0$ is chosen to be the smallest integer $\geq 1$ such
that $k_0 + 1 \geq \Upsilon^2/8$, we have 
\begin{equation*}
  k_0 \leq 1 + \frac{\Upsilon^2}{8} = \log \frac{e (2
    \pi)^{-d/2}}{\rho} \le \frac{3}{2} \log \frac{(2 \pi)^{-d/2}}{\rho}
\end{equation*}
because $\rho \leq (2 \pi)^{-d/2} e^{-1/2}$. This, along with the
expression for $\Upsilon$, gives 
\begin{equation*}
  \Delta_{i, 1} \le C \max \left\{\left(\log \frac{(2 \pi)^{-d/2}}{\rho}
\right)^{3/2}, \sqrt{|\log \hel(f_G, f_{G_0})|} \right\} \hel(f_G,
f_{G_0})
\end{equation*}
where $C$ is a universal positive constant. Combining with
\eqref{t31}, we deduce that 
\begin{equation*}
T_2^2 \leq C d \max \left\{\left(\log \frac{(2 \pi)^{-d/2}}{\rho}
\right)^{3}, |\log \hel(f_G, f_{G_0})| \right\} \hel^2(f_G, f_{G_0}). 
\end{equation*}
The proof of Theorem \ref{fcc} is now completed by combining the above
inequality with the bound \eqref{t1b} and the fact that $\Gamma(G_0,
G, \rho) \leq T_1 + T_2$ (which implies that $\Gamma^2(G_0, G, \rho)
\leq 2 T_1^2 + 2 T_2^2$).   
\end{proof}

\section{Auxiliary results} \label{auxre}
This section collects various results which were used in the proofs of
the main results of the paper. 

The following lemma generalizes \citet[Lemma A.1]{jiang2009general} to
the case $d \geq 1$. 

\begin{lemma}\label{p1}
Fix a probability measure $G$ on $\R^d$. For every $x \in \R^d$, we
have ($\norm{\cdot}$ denotes the usual Euclidean norm on $\R^d$)  
\begin{equation}\label{p1.eq}
\left( \frac{\norm{\nabla f_G(x)}}{f_G(x)} \right)^2 \leq   tr \left(I_d + \frac{H
f_G(x)}{f_G(x)} \right) \leq \log \frac{(2 \pi)^{-d}}{f_G^2(x)}
\end{equation}
where $\nabla$ and $H$ stand for gradient and Hessian respectively
and $tr$ denotes trace. 

Also for every $x \in \R^d$, we have 
\begin{equation}\label{p1.neq}
\frac{\norm{\nabla f_G(x)}}{\max \left(f_G(x), \rho \right)} \leq
\sqrt{ \log \frac{(2 \pi)^{-d}}{\rho^2}} 
\end{equation}
for $0 < \rho \leq (2 \pi)^{-d/2} e^{-1/2}$ and 
\begin{equation} \label{p1.neq2}
\left( \frac{\norm{\nabla f_G(x)}}{f_G(x)} \right)^2
\frac{f_G(x)}{f_G(x) \vee \rho} \leq \log \frac{(2
  \pi)^{-d}}{\rho^2}  
\end{equation}
for $0 < \rho \leq (2 \pi)^{-d/2} e^{-1}$. 
\end{lemma}

\begin{proof}[Proof of Lemma \ref{p1}]
If $\theta \sim G$ and $X | \theta \sim N(\theta, I_d)$, then it is
easy to verify that, for every $x \in \R^d$, 
\begin{equation*}
\frac{\nabla f_G(x)}{f_G(x)} = \E \left(\theta - X | X = x \right)
\end{equation*}
and 
\begin{equation}\label{hh}
\frac{H f_G(x)}{f_G(x)} = -I_d + \E \left( (\theta - X) (\theta -
X)^T | X = x \right).  
\end{equation}
From here, we can deduce that 
\begin{align*}
I_d + \frac{H f_G(x)}{f_G(x)} &= \E \left( (\theta - X) (\theta -
              X)^T | X = x \right)  \\
&= \left( \E (\theta - X | X = x) \right) \left(\E (\theta - X | X = x)
\right)^T \\ &+ \E \left( (\theta - \E(\theta | X = x)) (\theta - \E (\theta
| X = x))^T  | X = x\right) \\
&= \frac{\nabla f_G(x)}{f_G(x)} \frac{(\nabla f_G(x))^T}{f_G(x)} \\ &+ \E
\left( (\theta - \E(\theta | X = x)) (\theta - \E (\theta
| X = x))^T  | X = x\right)
\end{align*}
and hence
\begin{equation}\label{p11}
I_d + \frac{H f_G(x)}{f_G(x)} \succeq \frac{\nabla f_G(x)}{f_G(x)}
\frac{(\nabla f_G(x))^T}{f_G(x)} 
\end{equation}
where $A \succeq B$ means that $A - B$ is non-negative definite. 

Also from \eqref{hh} and the convexity of $A \mapsto \exp(tr(A)/2)$
($tr(A)$ denotes the trace of the $d \times d$ matrix $A$), we have
\begin{align*}
\exp \left(\frac{1}{2} tr \left(I_d + \frac{H f_G(x)}{f_G(x)}
\right) \right) &= \exp \left(\frac{1}{2} tr \left(\E \left((\theta
- X) (\theta - X)^T | X = x \right) \right) \right) \\
&\leq \E \left( \exp \left(\frac{1}{2} tr (\theta - X) (\theta - X)^T
\right) | X = x \right) \\
&= \E \left( \exp \left(\frac{1}{2} \norm{X - \theta}^2 \right)| X = x
\right) \\ &= \frac{(2 \pi)^{-d/2}}{f_G(x)}
\end{align*}
so that we have 
\begin{equation*}
tr \left(I_d + \frac{H f_G(x)}{f_G(x)} \right) \leq \log \frac{(2
\pi)^{-d}}{f_G^2(x)}. 
\end{equation*}
Combining with \eqref{p11}, we obtain \eqref{p1.eq}. 

To prove \eqref{p1.neq}, note first from \eqref{p1.eq} that 
\begin{align*}
\frac{\norm{\nabla f_G(x)}}{\max(f_G(x), \rho)} &\leq \sqrt{\log
\frac{(2 \pi)^{-d}}{f^2_G(x)} }
\frac{f_G(x)}{\max(f_G(x), \rho)} \\ &= 
\begin{cases} 
\sqrt{\log \frac{(2 \pi)^{-d}}{f^2_G(x)}} \leq \sqrt{\log \frac{(2
\pi)^{-d}}{\rho^2}}     & \text{if } f_G(x) > \rho \\ 
\sqrt{\log \frac{(2 \pi)^{-d}}{f^2_G(x)} } \frac{f_G(x)}{\rho}          & \text{if } f_G(x) \leq \rho
\end{cases}
\end{align*}
The function $v \mapsto v \log \left( (2 \pi)^{-d}/v \right)$ is
non-decreasing on $(0, (2 \pi)^{-d}/e]$ and hence when $f^2_G(x) \leq
\rho^2 \leq (2 \pi)^{-d}/e$, the inequality 
\begin{equation*}
\sqrt{\log \frac{(2 \pi)^{-d}}{f^2_G(x)} } \frac{f_G(x)}{\rho} \leq
\sqrt{\log \frac{(2 \pi)^{-d}}{\rho^2}}
\end{equation*}
holds and this proves \eqref{p1.neq}. 

We now turn to \eqref{p1.neq2}. Whenever $f_G(x) \geq \rho$, note that
\eqref{p1.neq2} follows directly from \eqref{p1.neq}. Thus,
\eqref{p1.neq2} only needs to be established when $f_G(x) < \rho$. In
this case using \eqref{p1.eq},  
\begin{align*}
\left( \frac{\norm{\nabla f_G(x)}}{f_G(x)} \right)^2 \frac{f_G(x)}{
  \max\{f_G(x), \rho \}} &\leq \left( \frac{f_G(x)}{\rho} \right) \log
                           \frac{(2 \pi)^{-d}}{f_G^2(x)} \\
&= 2 \log \frac{(2 \pi)^{-d/2}}{f_G(x)}
\frac{f_G(x)}{\rho} 
\end{align*}
Note that $v \mapsto v \log \left( (2 \pi)^{-d}/v^2
\right)$ is non-decreasing on $(0, (2 \pi)^{-d/2}/e]$. This,
along with $f_G(x) < \rho$, immediately implies \eqref{p1.neq2}. 
\end{proof}

For an infinitely differentiable function $u : \R^d \rightarrow \R$,
$1 \leq i \leq d$ and $k \geq 1$, let $\partial_i^k u : \R^d \rightarrow \R$
denote the function
\begin{equation*}
(\partial_i^k u)(x) := \frac{\partial^k}{\partial x_i^k} u(x). 
\end{equation*}
\begin{lemma}\label{kfo}
For every pair of probability measures $G$ and $G_0$ on $\R^d$, $1
\leq i \leq d$ and $k \geq 1$, we have
\begin{equation}\label{kfo.eq}
\begin{split}
\int \left\{\partial_i^k (f_G(x) - f_{G_0}(x)) \right\}^2 dx &\leq
4(2 \pi)^{-d/2} \inf_{a
\geq \sqrt{2k - 1}} \left\{a^{2k} \hel^2(f_G, f_{G_0}) \right. \\ &\left. +
\sqrt{\frac{2}{\pi}} a^{2k-1} e^{-a^2}  \right\}.  
\end{split}
\end{equation}
\end{lemma}

\begin{proof}[Proof of Lemma \ref{kfo}]
Fix $a \geq \sqrt{2k-1}$ and assume, without loss of generality,
that $i = 1$. Let 
\begin{equation*}
f^*_{G, 1}(u, x_{2}, \dots, x_d)  := \int e^{i u x_1} f_G(x) dx_1  
\end{equation*}
denote the Fourier transform of $f_G$ treated as a function of
$x_1$. The function $f^*_{G_0, 1}$ is defined analogously. For ease
of notation, we shall suppress the dependence of $f_{G, 1}^*(u,
x_2, \dots, x_d)$ (resp. $f_{G_0}^*(u, x_2, \dots, x_d)$) on $x_2,
\dots, x_d$ below and write it simply as $f_{G_1}^*(u)$
(resp. $f_{G_0}^*(u)$). 

For every $x_2, \dots, x_d$, we then have (by Plancherel's identity)
\begin{equation}\label{wi1}
\begin{split}
&2 \pi \int \left\{\partial_1^k (f_G(x) - f_{G_0}(x)) \right\}^2 dx_1 \\ &=
\int u^{2k} \left|f_{G, 1}^*(u)  - f_{G_0,
1}^*(u) \right|^2 du  \\
&\leq a^{2k} \int \left|f_{G, 1}^*(u)  - f_{G_0,
1}^*(u) \right|^2 du \\ &+ \int_{|u| > a} u^{2k} \left|f_{G,
1}^*(u)  - f_{G_0, 1}^*(u) \right|^2 du  \\
&= (2 \pi) a^{2k} \int \left( f_G(x) - f_{G_0}(x) \right)^2 dx_1 \\ &+ \int_{|u| > a} u^{2k} \left|f_{G,
1}^*(u)  - f_{G_0, 1}^*(u) \right|^2 du 
\end{split}
\end{equation}
for every $a > 0$. Also note that for every $u, x_2, \dots, x_d \in
\R$, 
\begin{align*}
f_{G, 1}^*(u) &= \int e^{i u x_1} \left(\int \phi_d(x - \theta)
dG(\theta) \right) dx_1 \\ 
&= \int \left( \int e^{i u x_1} \phi_d(x - \theta)
dx_1 \right) dG(\theta) \\
&= \int (2 \pi)^{-d/2} \left[ \int e^{i u x_1}
e^{-(x_1 -  \theta_1)^2/2} dx_1 \right] \exp \left(-\sum_{j \neq 1} (x_j -
\theta_j)^2/2 \right) dG(\theta) \\
&= (2 \pi)^{-(d-1)/2} \int e^{i u x_1} e^{-u^2/2} \exp \left(-\sum_{j \neq 1} (x_j -
\theta_j)^2/2 \right) dG(\theta)
\end{align*}
so that 
\begin{equation*}
\left|f_{G, 1}^*(u) \right| \leq (2 \pi)^{-(d-1)/2} e^{-u^2/2} \int \exp \left(-\sum_{j \neq 1} (x_j -
\theta_j)^2/2 \right) dG(\theta). 
\end{equation*}
An analogous bound also holds for $|f_{G_0, 1}^*(u)|$. Using these
bounds for $f_{G, 1}^*(u)$ and $f_{G_0, 1}^*(u)$, the second term in
\eqref{wi1} can be bounded from above as 
\begin{equation*}
\begin{split}
& \int_{|u| > a} u^{2k} \left|f_{G,
1}^*(u)  - f_{G_0, 1}^*(u) \right|^2 du \\ &\leq 2 (2 \pi)^{-(d-1)} \int \exp \left(-\sum_{j \neq 1} (x_j -
\theta_j)^2 \right) \left\{ dG(\theta) + dG_0(\theta) \right\}
\int_{|u| > a} u^{2k} e^{-u^2} du
\end{split}
\end{equation*}
Thus integrating both sides of \eqref{wi1} with respect to $x_2,
\dots, x_d$, we deduce that 
\begin{equation*}
\begin{split}
2 \pi \int \left\{\partial_1^k (f_G(x) - f_{G_0}(x)) \right\}^2 dx &\leq (2
\pi) a^{2k} \int (f_G - f_{G_0})^2 \\ &+ 4 (2 \pi)^{-(d-1)/2}
\int_{|u| > a} u^{2k} e^{-u^2} du. 
\end{split}
\end{equation*}
which implies that 
\begin{equation*}
\begin{split}
\int \left\{\partial_1^k (f_G(x) - f_{G_0}(x)) \right\}^2 dx &\leq a^{2k}
\int (f_G - f_{G_0})^2 \\ &+ 8 (2 \pi)^{-(d+1)/2} 
\int_{u > a} u^{2k} e^{-u^2} du. 
\end{split}
\end{equation*}
We now use the integration by parts argument in   \citet[Page
1675]{jiang2009general} which gives 
\begin{equation*}
\int_{u > a} u^{2k} e^{-u^2} du \leq a^{2k-1} e^{-a^2} \qt{provided
$a \geq \sqrt{2k-1}$}. 
\end{equation*}
The proof of Lemma \ref{kfo} is now completed by noting that 
\begin{equation*}
\begin{split}
\int (f_G - f_{G_0})^2 &\leq \int \left(\sqrt{f_G} - \sqrt{f_{G_0}}
\right)^2 \left(\sqrt{f_G} + \sqrt{f_{G_0}} \right)^2 \\ &\leq 4 (2
\pi)^{-d/2} \hel^2(f_G, f_{G_0})
\end{split}
\end{equation*}
where we have used that every Gaussian mixture density $f_G$ is
bounded from above by $(2 \pi)^{-d/2}$. 
\end{proof}

\subsection{Proof of Lemma \ref{tailmom}} 
We now prove Lemma \ref{tailmom} (this lemma was stated in Subsection
\ref{denso}). 

\begin{proof}[Proof of Lemma \ref{tailmom}]
  We write
  \begin{align*}
& \E \left\{ \prod_{i=1}^n \left| a g(X_i) \right|^{I\{g(X_i)\geq
    M\}} \right\}^{\lambda} \\ &= \prod_{i=1}^n \E \left|a g(X_i)
\right|^{\lambda I\{g(X_i) \geq M\}} \\
&\leq \prod_{i=1}^n \left\{1 + a^{\lambda} \E \left[(g(X_i))^{\lambda}
  I\{g(X_i) \geq M\}\right] \right\} \\
&\leq \prod_{i=1}^n \exp \left(a^{\lambda} \E (g(X_i))^{\lambda}
  I\{g(X_i) \geq M\} \right) \\
&= \exp \left(a^{\lambda} \sum_{i=1}^n \E \left[ (g(X_i))^{\lambda} I\{g(X_i)
  \geq M\} \right] \right) \\
&= \exp \left(n a^{\lambda} \int (g(x))^{\lambda} I\{g(x) \geq M\}
  f_{\bar{G}_n}(x) dx \right) = \exp \left(n a^{\lambda} U \right) 
  \end{align*}
where 
\begin{equation*}
  U := \int (g(x))^{\lambda} I\{g(x) \geq M\} f_{\bar{G}_n}(x) dx = \E \left[
  (g(\theta + Z))^{\lambda} I\{g(\theta + Z) \geq M\} \right]
\end{equation*}
with independent random variables $Z \sim N(0, I_d)$ and $\theta \sim
\bar{G}_n$. Because of the $1$-Lipschitz property of $g$, we have
$g(\theta + z) \leq g(\theta) + \|z\|$ so that 
\begin{equation}\label{dng}
  U \leq \E (2 \|Z\|)^{\lambda} I \left\{2 \|Z\| \geq M \right\} + \E
  (2 g(\theta))^{\lambda} I \left\{2 g(\theta) \geq M \right\}. 
\end{equation}
The first term above will be bounded as
\begin{align*}
&\E \left[ (2\|Z\|)^{\lambda} I\{ 2\|Z\| \geq M\}\right] \\ &= M^{\lambda} \E \left[ \left(\frac{\|Z\|}{M/2} \right)^{\lambda} I\{ \|Z\| \geq M/2\}\right] \\
 &\leq M^{\lambda} \E \left[ \left(\frac{\|Z\|}{M/2} \right) I\{ \|Z\| \geq M/2\}\right] \qt{ since } \lambda \leq 1 \\
 &= 2 M^{\lambda - 1} \frac{1}{(2\pi)^{d/2}} \int_{\|x\| \geq M/2} \|x\| e^{-\|x\|^2/2} dx \\
 &\leq C_d M^{\lambda - 1} \int_{r \geq M/2} r e^{-r^2/2} r^{d-1} dr
 \leq C_d M^{\lambda + d - 2} e^{-M^2/8} 
\end{align*}
where the last inequality follows from Lemma \ref{Lemma:GammaTail}.
Because $M \ge \sqrt{8 \log n}$, we have $e^{-M^2/8} \leq 1/n$ and
this gives 
\begin{equation}\label{dng1} 
\E \left[ (2\|Z\|)^{\lambda} I\{ 2\|Z\| \geq M\}\right] \leq
\frac{C_d}{n} M^{\lambda + d - 2}. 
\end{equation}

For the second term in \eqref{dng}, note that (because $\lambda \leq
p$) 
\begin{align}
\E \left[ (2g(\theta))^{\lambda} I\{ 2g(\theta) \geq M\}\right] &=
M^{\lambda} \int_{g(\theta) \geq M/2}
\left(\frac{g(\theta)}{M/2}\right)^{\lambda} G_n(d\theta) \nonumber \\  
&\leq M^{\lambda} \int \left(\frac{g(\theta)}{M/2}\right)^p
G_n(d\theta) = M^{\lambda} \left(\frac{2
    \mu_p(g)}{M}\right)^p \label{dng2}. 
\end{align}
The proof of \eqref{tailmom.eq} is now completed by putting together
inequalities \eqref{dng}, \eqref{dng1} and \eqref{dng2}.  

For \eqref{tailprob.eq}, we first use an argument similar to the above
to write 
	\begin{equation*}
	\frac{1}{n} \sum_{i=1}^n \P\left[ g(X_i) \geq M \right] = \P\left[ g(\theta + Z) \geq M \right]
	\end{equation*}
	where $\theta \sim \bar G_n$ and $Z \sim N(0, I_d)$ are
        independent. Since $g$ is $1$-Lipschitz, $g(\theta + z) \leq
        g(\theta) + \|z\|$. Consequently, 
	\begin{equation*}
	\P\left[ g(\theta + Z) \geq M \right] \leq \P\left[ 2g(\theta)
          \geq M \right] + \P\left[ 2\|Z\| \geq M \right] 
	\end{equation*}
	Applying \eqref{dng1} and \eqref{dng2} with $\lambda = 0$ then
        concludes the proof of \eqref{tailprob.eq}. 
\end{proof}

\begin{remark}
  We shall apply Lemma \ref{tailmom} to the function 
  \begin{equation*}
    \mathfrak{d}_S(x) := \inf_{u \in S} \|x - u\|
  \end{equation*}
  for a fixed subset $S$ of $\R^d$. This function is clearly
  nonnegative and $1$-Lipschitz. Inequality \eqref{tailmom.eq} in
  Lemma \ref{tailmom} then gives the inequality  
\begin{equation} \label{Eq:tailExpect}
\begin{split}
&\E \left\{ \prod_{i=1}^n \left| a \mathfrak{d}_S(X_i) \right|^{I\{\mathfrak{d}_S(X_i)\geq
    M\}} \right\}^{\lambda} \\ &\leq \exp\left\{C_d a^{\lambda} M^{\lambda
    + d - 2} + (aM)^{\lambda} n
  \left(\frac{2\mu_p(\mathfrak{d}_S)}{M}\right)^p\right\}
\end{split}
\end{equation}
for all $a > 0, M \geq \sqrt{8 \log n}$ and $0 < \lambda \leq \min(1,
p)$. 

Further, inequality \eqref{tailprob.eq} for $g = \dos$ gives
	\begin{equation} \label{Eq:tailProb}
	\frac{1}{n} \sum_{i=1}^n \P\left[ \dos(X_i) \geq M \right] \leq
        C_d \frac{M^{d-2}}{n} + \inf_{p \geq \frac{d+1}{2 \log n}}
        \left( \frac{2 \mu_p(\dos)}{M} \right)^p 
	\end{equation}
	for all $M \geq \sqrt{8 \log n}$.

These two inequalities \eqref{Eq:tailExpect} and \eqref{Eq:tailProb}
hold under the same assumptions on $X_1, \dots, X_n$ as in Lemma
\ref{tailmom}. 
\end{remark}

\begin{lemma}\label{devs}
  Fix $\theta_1, \dots, \theta_n \in \R^d$. Suppose $X_1, \dots, X_n$
  are independent random vectors with $X_i \sim N(\theta_i, \Sigma_i)$
  for some covariance matrices $\Sigma_1, \dots, \Sigma_n$ with
  $\Sigma_i \gtrsim I_d$. Suppose $\sigma^2_{\max}$ is such that
  $\max_{1 \leq j \leq k} \lambda_{\max}(\Sigma_j) \leq
  \sigma^2_{\max}$ where $\lambda_{\max}(\Sigma_j)$ denotes the
  largest eigenvalue of $\Sigma_j$. For $f \in \M$ and $\rho > 0$, let
  $T_f(\bx, \rho)$ be defined as in the proof of Theorem
  \ref{rgende} as the $d \times n$ matrix whose
  $i^{th}$ column is given by the $d \times 1$ vector:
  \begin{equation*}
    X_i + \frac{\nabla f(X_i)}{\max(f(X_i), \rho)} \qt{for $i = 1,
      \dots, n$}. 
  \end{equation*}
Then for every $f_1, f_2 \in \M$,  $0 < \rho \leq (2 \pi)^{-d/2} 
e^{-3/2}$ and $x > 0$, we have 
\begin{equation}\label{devs.eq}
\begin{split}
  \P & \left\{\norm{T_{f_1}(\bx, \rho) - T_{f_2}(\bx, \rho)}_F
    \geq \E \norm{T_{f_1}(\bx, \rho) - T_{f_2}(\bx, \rho)}_F +
    x \right\} \\ &\leq \exp \left(\frac{-x^2}{8 \sigma^2_{\max}
      L^4(\rho)} \right) \qt{ where $L(\rho) := \sqrt{\log \frac{1}{(2
        \pi)^d \rho^2}}$}. 
\end{split}
\end{equation}
\end{lemma}

\begin{proof}[Proof of Lemma \ref{devs}]
Fix $f_1, f_2 \in \M$ and let  
  \begin{equation*}
    F(\bx) := \norm{T_{f_1}(\bx, \rho) - T_{f_2}(\bx,
      \rho)}_F. 
  \end{equation*}
  We shall prove that $F(\bx)$, as a function of $\bx$, is Lipschitz
  with constant $2 L^2(\rho)$ under the Frobenius matrix norm on $\bx$ 
  i.e.,  
  \begin{equation}\label{sudm}
   | F(\bx) - F(\by) | \leq 2 L^2(\rho) \norm{\bx - \by}_F. 
  \end{equation}
  Inequality \eqref{devs.eq} would then directly follow from the
  standard concentration inequality for Lipschitz functions of
  Gaussian random vectors (see, for example, \citet[Theorem 
  5.6]{boucheron2013concentration}). Indeed, to see that \eqref{sudm}
  implies \eqref{devs.eq} by Gaussian concentration, observe that if
  $Z_i := \Sigma_i^{-1/2}(X_i - \theta_i) \sim N(0, I_d)$ for $i = 1,
  \dots, n$, then $F(\bx)$ equals the function $G(Z_1, \dots, Z_n)$ of
  $Z_1, \dots, Z_n$ where
  \begin{equation*}
    G(z_1, \dots, z_n) := F(\theta_1 + \Sigma_1^{1/2} z_1, \dots, 
    \theta_n + \Sigma_n^{1/2}z_n). 
  \end{equation*}
  Now \eqref{sudm} implies that 
  \begin{align*}
  &  \left|G(z_1, \dots, z_n) - G(w_1, \dots, w_n) \right| \\ &\leq 2 L^2(\rho)
                                              \sqrt{\sum_{i=1}^n (z_i
                                              - w_i)^T \Sigma_i (z_i -
                                              w_i)} \\
&\leq 2 L^2(\rho) \sqrt{\max_{1 \leq j \leq n}
  \lambda_{\max}(\Sigma_j)} \sqrt{\sum_{i=1}^n \|z_i - w_i\|^2} \\
&\leq 2 L^2(\rho) \sigma_{\max} \sqrt{\sum_{i=1}^n \|z_i - w_i\|^2}
  \end{align*}
  for every $z_1, \dots, z_n, w_1, \dots, w_n$. This implies that
  $F(\bx)$ is a $\left(2L^2(\rho) \sigma_{\max}\right)$-Lipschitz function of
  independent standard random vectors $Z_1, \dots, Z_n$ so that
  \eqref{devs.eq} follows by the
  standard concentration inequality for Lipschitz functions of
  Gaussian random vectors (see, for example, \citet[Theorem 
  5.6]{boucheron2013concentration}).

   It is enough therefore to prove \eqref{sudm}. For this, note first
   that  
  \begin{align*}
    |F(\bx) - F(\by)| &= \left| \norm{T_{f_1}(\bx, \rho) - T_{f_2}(\bx,
                        \rho)}_{F} - \norm{T_{f_1}(\by, \rho) - T_{f_2}(\by,
                        \rho)}_{F} \right| \\
&\leq \norm{T_{f_1}(\bx, \rho) - T_{f_1}(\by, \rho)}_F +
  \norm{T_{f_2}(\bx, \rho) - T_{f_2}(\by,
  \rho)}_F. 
  \end{align*}
Note now that, for every $f \in \M$, 
\begin{equation}\label{kk12}
  \norm{T_f(\bx, \rho) - T_{f}(\by, \rho)}^2_F = \sum_{i=1}^n
  \norm{t_f(X_i, \rho) - t_f(Y_i, \rho)}^2
\end{equation}
where 
\begin{equation*}
  t_f(x, \rho) := x + \frac{\nabla f(x)}{\max(f(x), \rho)}. 
\end{equation*}
To bound $\norm{t_f(X_i, \rho) - t_f(Y_i, \rho)}$, we compute the Jacobian of
the map $x \mapsto t_f(x, \rho)$ as 
\begin{equation*}
  J t_f(x, \rho) = \begin{cases}
I_d + \frac{H f(x)}{\rho} &\text{ if } f(x) < \rho \\ 
I_d + \frac{H f(x)}{f(x)} - \left(\frac{\nabla f(x)}{f(x)} \right)
\left(\frac{\nabla f(x)}{f(x)} \right)^T &\text{ if } f(x) > \rho 
\end{cases}
\end{equation*}
where $\nabla$  and $H$ denote gradient and Hessian respectively. We
shall now argue that 
\begin{equation}\label{pcqi}
0 \preceq J t_f(x, \rho) \preceq L^2(\rho) I_d  
\end{equation}
where $A \preceq B$ means that $B - A$ is a nonnegative definite
matrix. Before proving \eqref{pcqi}, let us first note that
\eqref{pcqi} implies 
\begin{equation*}
  \norm{t_f(x, \rho) - t_f(y, \rho)} \leq L^2(\rho) \norm{x - y} 
\end{equation*}
which further implies, via \eqref{kk12}, that 
\begin{equation*}
  \norm{T_f(\bx, \rho) - T_{f}(\by, \rho)}^2_F \leq L^2(\rho)
  \norm{\bx - \by}_F^2. 
\end{equation*}
Since this inequality holds for every $f \in \M$, it also holds for
both $f_1$ and $f_2$ which gives \eqref{sudm}  and completes the proof
of Lemma \ref{devs}. 

It remains to prove \eqref{pcqi}. For this, we shall use the above
expression for $J t_f(x, \rho)$ as well as inequality \eqref{p1.eq}
from Lemma \ref{p1} and inequality \eqref{p11} from the proof of Lemma
\ref{p1}. First when $f(x) > \rho$, note that 
\begin{equation*}
  J t_f(x, \rho) = I_d + \frac{H f(x)}{f(x)} - \left(\frac{\nabla f(x)}{f(x)} \right)
\left(\frac{\nabla f(x)}{f(x)} \right)^T  
\end{equation*}
which is $\succeq 0$ from \eqref{p11} and, by \eqref{p1.eq}, we get
\begin{equation*}
0 \preceq  J t_f(x, \rho) \preceq I_d + \frac{H f(x)}{f(x)} \preceq tr
\left( I + \frac{H f(x)}{f(x)} \right) I_d  \preceq L^2(f(x)) I_d
\preceq L^2(\rho) I_d
\end{equation*}
where, in the last inequality, we have used that $L(\cdot)$ is a
decreasing function. Here $tr$ denotes trace. This proves \eqref{pcqi}
when $f(x) > \rho$. Now let $f(x) < \rho$. Then  
\begin{equation*}
  J t_f(x, \rho) = I_d + \frac{H f(x)}{\rho} = \left(1 -
    \frac{f(x)}{\rho} \right) I_d + \frac{f(x)}{\rho} \left(I_d +
    \frac{Hf}{f} \right)
\end{equation*}
which is $\succeq 0$ because $f(x) < \rho$ and because of
\eqref{p11}. Also, by \eqref{p1.eq}, 
\begin{align*}
  J t_f(x, \rho) &= \left(1 -
    \frac{f(x)}{\rho} \right) I_d + \frac{f(x)}{\rho} \left(I_d +
    \frac{Hf}{f} \right) \\
&\preceq \left(1 -
    \frac{f(x)}{\rho} \right) I_d + \frac{f(x)}{\rho} I_d tr\left(I_d +
    \frac{Hf}{f} \right) \\
&\preceq \left(1 + \frac{f(x)}{\rho} \left(\log \frac{(2
  \pi)^{-d}}{f^2(x)} -1 \right) \right) I_d \\ &= \left(1 +
  \frac{f(x)}{\rho} \left(L^2(f(x)) -1 \right) \right) I_d  
\end{align*}
The right hand side above is $\preceq L^2(\rho) I_d$  because $t
\mapsto t(L^2(t) - 1)$ is non-decreasing on $t \in (0, (2 \pi)^{-d/2}
e^{-3/2}]$ so that when $f(x) < \rho$, we have 
\begin{align*}
1 + \frac{f(x)}{\rho} \left(L^2(f(x)) -1 \right) \leq L^2(\rho). 
\end{align*}
This proves \eqref{pcqi} which completes the proof of Lemma
\ref{devs}. 
\end{proof}

\begin{lemma} \label{Lemma:GammaTail}
There exists a positive constant $A_d$ depending only on $d$ such that
for every $M \geq 1$ and $d \in \{0, 1, 2, \dots \}$, we have  
\begin{equation} \label{Eq:induction_tail_bound}
I(d): = \int_{r \geq M} r^d e^{-r^2/2} dr \leq A_d M^{d-1}
e^{-M^2/2}. 
\end{equation}
\end{lemma}

\begin{proof}[Proof of Lemma \ref{Lemma:GammaTail}] 
Let $A_0 := 1$, $A_1 := 1$ and define $A_d$ for $d \geq 2$ via the
recursion $A_d := 1 + (d-1) A_{d-2}$. Clearly
\begin{equation*}
  I(0) = \int_{r \geq M} e^{-r^2/2} dr \leq \int_{r \geq M}
  \frac{r}{M} e^{-r^2/2} = M^{-1} e^{-M^2/2}
\end{equation*}
and
\begin{equation*}
  I(1) = \int_{r \geq M} r e^{-r^2/2} dr = e^{-M^2/2} 
\end{equation*}
and thus inequality \eqref{Eq:induction_tail_bound} holds for $d = 0$
and $d = 1$. For $d \geq 2$, integration by parts gives
\begin{equation*}
I(d) = M^{d-1}e^{-M^2/2} + (d-1)I(d-2). 
\end{equation*}
Inequality \eqref{Eq:induction_tail_bound} for $d \geq 2$ now easily
follows by induction on $d$. 
\end{proof}

\begin{lemma} \label{Lemma:integralV}
Let $S$ be a compact subset of $\R^d$. For $\eta, M > 0$, define
\begin{equation} \label{Eq:vFunc}
v(x) := \begin{cases}
\eta &\text{ if } x \in S^M \\
\eta\left( \frac{M}{\mathfrak{d}_S(x)} \right)^{d+1} &\text{ otherwise }
\end{cases}
\end{equation}
Then, for some constant $C_d$ depending only on $d$,
\begin{equation} \label{Eq:integralV}
\int v(x) dx \leq C_d \eta \text{Vol}\left(S^M\right)
\end{equation}
\end{lemma}

\begin{proof}[Proof of Lemma \ref{Lemma:integralV}]
We first write
\begin{equation}
  \label{in1}
\int v(x) dx = \eta \text{Vol}\left(S^M\right) + \eta M^{d+1} \int_{x
  \notin S^M} \frac{1}{\mathfrak{d}_S(x)^{d+1}} dx   
\end{equation}
Let $N$ be the maximal integer such that there exist $u_1, \dots, u_N
\in S$ with 
\begin{equation}\label{ps}
  \min_{i \neq j} \|u_i - u_j\| \geq M/2. 
\end{equation}
The maximality of $N$ implies that $\sup_{u \in S} \min_{1 \leq i \le
  N} \|u - u_i\| \leq M/2$.  As a result, for every $x \in \R^d$, by
triangle inequality, we have
\begin{equation*}
 \mathfrak{d}_S(x) = \min_{u \in S} \|x - u\| \geq \min_{1 \leq i \leq N} \|x -
 u_i\| - \frac{M}{2} 
\end{equation*}
so that
\begin{align}
  \int_{x \notin S^M} \frac{dx}{(\mathfrak{d}_S(x))^{d+1}} &\leq \int_{x \notin
    S^M} \left(\frac{1}{\min_{1 \leq i \leq N} \|x - u_i\| - M/2}
  \right)^{d+1} dx \nonumber \\ 
&\leq \sum_{i=1}^N \int_{x \notin S^M} \left( \frac{1}{\|x - u_i\| -
    M/2} \right)^{d+1} dx \nonumber \\
&\leq \sum_{i=1}^N \int_{\|x - u_i\| \geq M} \left( \frac{1}{\|x - u_i\| -
    M/2} \right)^{d+1} dx \nonumber \\
&= N \int_{\|x\| \geq M} \left(\frac{1}{\|x\| - M/2} \right)^{d+1} dx \nonumber
\\
&= N C_d \int_{M}^{\infty} \left(\frac{1}{r - M/2} \right)^{d+1}
r^{d-1} dr \nonumber \\
&= N C_d \int_{M/2}^{\infty} t^{-d-1} \left(\frac{M}{2} + t
\right)^{d-1} dt \nonumber \\ &\leq N C_d \int_{M/2}^{\infty} t^{-d-1} (2t)^{d-1} dt
= \frac{N C_d 2^d }{M}. \label{in2}
\end{align}
Note now that because of \eqref{ps}, the balls $B(u_i, M/4), i = 1,
\dots, N$ have disjoint interiors and are all contained in
$S^{M/4}$. As a result
\begin{equation}\label{in3}
  N \leq \frac{\text{Vol}(S^{M/4})}{\text{Vol}(B(0, M/4))} \leq C_d
  \frac{\text{Vol}(S^M)}{M^d}. 
\end{equation}
The proof of Lemma \ref{Lemma:integralV} is completed by putting
together inequalities \eqref{in1}, \eqref{in2} and \eqref{in3}. 
\end{proof}

\begin{lemma}\label{volm}
There exists a positive constant $C_d$ such that for every compact set
$K \subseteq \R^d$ and real numbers $\epsilon > 0$ and $M > 0$, we
have 
  \begin{equation}\label{volm.eq}
    N(\epsilon, K) \leq C_d \epsilon^{-d}\text{Vol}(K^{\epsilon/2}) 
  \end{equation}
and 
\begin{equation}\label{sms1}
  \vol(K^{2M}) \leq C_d \vol(K^{\epsilon/2}) \left(1 +
    \frac{M}{\epsilon} \right)^d 
\end{equation}
\end{lemma}

\begin{proof}[Proof of Lemma \ref{volm}]
 Let us first prove \eqref{volm.eq}. Let $a_1, \dots, a_N \in K$ be a
 maximal set of points such that 
  $\min_{i \neq j} \| a_i - a_j\| \geq \epsilon$. Then clearly
  $N(\epsilon, K) \leq N$. The balls $B(a_i, \epsilon/2)$ for $i = 1,
  \dots, N$ have disjoint interiors and are all contained in
  $K^{\epsilon/2}$. As a result
  \begin{equation}\label{khy1}
    N(\epsilon, K) \leq N \leq
    \frac{\text{Vol}(K^{\epsilon/2})}{\text{Vol}(B(0, \epsilon/2))}
  \end{equation}
  from which \eqref{volm.eq} follows. 

To prove \eqref{sms1}, note that the $K$ is contained in the union of
the balls $B(a_i, \epsilon)$ for $i = 1, \dots, N$. This implies that 
\begin{equation*}
  K^{2M} \subseteq \cup_{i=1}^N B(a_i, \epsilon + 2 M) 
\end{equation*}
so that 
\begin{equation*}
  \vol(K^{2M}) \leq N \vol(B(0, \epsilon + 2M)). 
\end{equation*}
Inequality \eqref{khy1} then gives
\begin{equation*}
  \vol(K^{2M}) \leq \frac{\vol(K^{\epsilon/2})}{\vol(B(0,
    \epsilon/2))} \vol(B(0, \epsilon + 2M)) \leq C_d
  \vol(K^{\epsilon/2}) \left(1 + \frac{M}{\epsilon} \right)^d.  
\end{equation*}

\end{proof}

\subsection{Proof of Lemma \ref{dco}} 
In this subsection, we prove Lemma \ref{dco} which was stated in
Subsection \ref{gendo} and used in the proof of Theorem \ref{rgende}. 

\begin{proof}[Proof of Lemma \ref{dco}]
  The proof uses Lemma \ref{cvx}. 
 
  Fix a compact set $S$. Suppose first that $G$ is supported on $S$ so
  that the second term in \eqref{dco.eq} equals 0. 

  We consider two   further special cases. First assume that $S$ is
  contained in a ball of radius $a := 4/L(\rho)$. Without loss of
  generality, we may assume that the ball is centered at the
  origin. Because $G$ is assumed to be supported on $S$, we have
  $\norm{\theta} \leq a$ almost surely under $G$. 

   For $\theta \sim G$ and $X | \theta \sim N(\theta, I_d)$, we can
   write 
   \begin{equation*}
     \frac{\nabla f_G(x)}{f_G(x)} = \E \left( \theta - X | X = x \right)
   \end{equation*}
   so that 
\begin{equation}\label{des1}
  \frac{\norm{\nabla f_G(x)}}{f_G(x)} = \norm{\E \left(\theta - X | X =
     x \right)} \leq \E \left(\norm{\theta - X} | X = x \right) \leq
 \|x\| + a. 
\end{equation}
Note also that
\begin{equation}\label{des2}
(2 \pi)^{-d/2} \exp \left(-\frac{1}{2} \left(\norm{x} + a \right)^2
\right) \leq  f_G(x) \leq (2 \pi)^{-d/2} \exp \left(-\frac{1}{2}
  \left(\norm{x} - a \right)_+^2 \right)  
\end{equation}
because $(\norm{x}  - a)_+ \leq \norm{x - \theta} \leq \norm{x} + a$
whenever $\norm{\theta} \leq a$. This also implies that whenever
$f_G(x) \leq \rho$, we have
\begin{equation*}
  \rho \geq (2 \pi)^{-d/2} \exp \left(-\frac{1}{2} \left(\norm{x} + a \right)^2
\right)
\end{equation*}
which gives 
\begin{equation}\label{des3}
  \norm{x} + a \geq L(\rho) := \sqrt{\log \frac{1}{(2
      \pi)^d \rho^2}}. 
\end{equation}
Putting together \eqref{des1}, \eqref{des2} and \eqref{des3}, we
deduce that 
\begin{align*}
  \Delta(G, \rho) &\leq \int \{f_G \leq \rho\} \left(\frac{\norm{\nabla
      f_G}}{f_G} \right)^2 f_G \\
&\leq \int_{ \{\norm{x} + a \geq
  L(\rho)\}} \left(\norm{x} + a \right)^2 (2 \pi)^{-d/2} \exp
\left(-\frac{1}{2} 
  \left(\norm{x} - a \right)_+^2 \right) dx. 
\end{align*}
Moving to polar coordinates, we deduce 
\begin{equation*}
  \Delta(G, \rho) \leq C_d \int_{(L(\rho) - a)_+}^{\infty} (r + a)^2 
  \exp \left(-(r - a)_+^2/2 \right) r^{d-1} dr. 
\end{equation*}
Note that with $a := 4/L(\rho)$ and $\rho \leq (2
\pi)^{-d/2}/\sqrt{e}$, we have $4 a \leq L(\rho)$ so that 
\begin{equation*}
  \Delta(G, \rho) \leq C_d \int_{L(\rho) - a}^{\infty} (r + a)^2 
  \exp \left(-(r - a)^2/2 \right) r^{d-1} dr. 
\end{equation*}
By a change of variable $r-a \mapsto r$, we obtain
\begin{equation*}
  \Delta(G, \rho) \leq C_d \int_{L(\rho) - 2a}^{\infty} (s + 2a)^2 
  \exp \left(-s^2/2 \right) (s + a)^{d-1} ds.  
\end{equation*}
Because $4a \leq L(\rho)$, we have 
\begin{equation*}
  s+a \leq s + 2a \leq s + L(\rho) - 2a \leq 2s
\end{equation*}
whenever $s \geq L(\rho) - 2a$. Thus
\begin{equation*}
  \Delta(G, \rho) \leq C_d \int_{L(\rho) - 2a}^{\infty} s^{d+1}
  e^{-s^2/2} ds. 
\end{equation*}
By Lemma \ref{Lemma:GammaTail}, we deduce that 
\begin{equation*}
\begin{split}
  \Delta(G, \rho) &\le C_d (L(\rho))^d \exp \left(-\frac{1}{2}
    \left(L(\rho) - 2a \right)^2 \right) \\ &\leq C_d (L(\rho))^d e^{2 a
    L(\rho)} e^{-L^2(\rho)/2} = C_d \rho (L(\rho))^d e^{2 a L(\rho)}. 
\end{split}
\end{equation*}
We now take
\begin{equation*}
  a := \frac{4}{L(\rho)}
\end{equation*}
which gives 
\begin{equation}\label{bc1}
  \Delta(G, \rho) \leq C_d \rho (L(\rho))^d
\end{equation}
whenever $G$ is supported on a set that is contained in a ball of
radius $a = 4/L(\rho)$. 

For the rest of the proof, we shall use Lemma \ref{cvx}. Now suppose
that $G$ is supported on a general compact set $S$. Then, 
for $N := N(a, S)$ (where $a := 4/L(\rho)$), let $E_1, \dots, E_N$
denote a disjoint covering of $S$ such that each $E_i$ is contained in
a ball of radius $a$. We can then write 
\begin{equation*}
  G := \sum_{j=1}^N w_j H_j
\end{equation*}
where $w_j := G(E_j)$ and $H_j$ is the probability measure $G$
conditioned on $H_j$. The bound \eqref{cvx.eq} in Lemma \ref{cvx} then
gives 
\begin{equation*}
  \Delta(G, \rho) \leq \sum_{j=1}^N w_j \Delta(H_j, \rho/w_j). 
\end{equation*}
Because $H_j$ is supported on a ball of radius at most $a$, we can use
\eqref{bc1} on each $H_j$ to deduce that 
\begin{equation}\label{bc2}
  \Delta(G, \rho) \leq C_d \sum_{j=1}^N w_j \frac{\rho}{w_j}
  L^d(\rho/w_j) \leq C_d \rho N(a, S) L^d(\rho). 
\end{equation}

To bound $\Delta(G, \rho)$ for an arbitrary probability measure $G$,
we write
\begin{equation*}
  G = w_1 H_1 + w_2 H_2
\end{equation*}
where $w_1 = G(S) = 1 - w_2$ and $H_1$ and $H_2$ are the 
probability measures obtained by conditioning $G$ on $S$ and $S^c$
respectively. Then clearly $H_1$ is supported on a compact set $S$ so
that the bound \eqref{bc2} can be used for $\Delta(H_2,
\rho/w_2)$. For $\Delta(H_1, \rho/w_1)$, we use the trivial bound $d$
(see the first part of Lemma \ref{cvx}). This gives (via
\eqref{cvx.eq}) 
\begin{equation*}
  \Delta(G, \rho) \leq C_d G(S) N(a, S) L^d(\rho) \rho + d~G(S^c) \leq
  C_d N(a, S) L^d(\rho) \rho + d~G(S^c)
\end{equation*}
which completes the proof of Lemma \ref{dco}. 
\end{proof}

\begin{lemma}\label{cvx}
For a probability measure $G$ on $\R^d$ and $\rho > 0$, let 
\begin{equation*}
  \Delta(G, \rho) := \int \left(1 - \frac{f_G}{\max \left(f_G , \rho \right)}
  \right)^2 \frac{\norm{\nabla f_G}^2}{f_{G}} 
\end{equation*}
The following pair of statements are then true. 
  \begin{enumerate}
  \item For every $G$ and $\rho > 0$, we have $\Delta(G, \rho) \leq
    d$. 
  \item   Suppose $G = \sum_{j=1}^m w_j H_j$ for some probability measures
  $H_1, \dots, H_m$ and weights $w_1, \dots, w_m$. Then 
  \begin{equation}\label{cvx.eq}
    \Delta(G, \rho) \leq \sum_{j=1}^m w_j \Delta \left(H_j, \rho/w_j
    \right). 
  \end{equation}
  \end{enumerate}
\end{lemma}

\begin{proof}[Proof of Lemma \ref{cvx}] 
  To prove that $\Delta(G, \rho) \leq d$, note that if $\theta \sim G$ and
  $X | \theta \sim N(\theta, I_d)$, then
  \begin{equation*}
    \frac{\nabla f_G(x)}{f_G(x)}  = \E \left(\theta - X | X = x
    \right). 
  \end{equation*}
As a result 
  \begin{equation*}
\Delta(G, \rho) \leq \int \frac{\norm{\nabla f_G}^2}{f_G} = \E
\norm{\E(\theta - X | X)}^2 \leq \E \norm{\theta - X}^2 = d.  
  \end{equation*}
For proving \eqref{cvx.eq}, note first that by the convexity of $x
\mapsto \|x\|^2$, we have  
  \begin{align*}
    \frac{\norm{\nabla f_G}^2}{f_G} &= \frac{\norm{\sum_j w_j \nabla
                                      f_{H_j}}^2}{\sum_{j} w_j
                                      f_{H_j}} \\
&= \norm{\sum_j \left(\frac{w_j f_{H_j}}{\sum_j w_j f_{H_j}} \right)
  \frac{\nabla f_{H_j}}{f_{H_j}}}^2 \left(\sum_{j} w_j f_{H_j} \right)
    \\
&\leq \left\{\sum_j \left(\frac{w_j f_{H_j}}{\sum_j w_j f_{H_j}}
  \right) \frac{\norm{\nabla f_{H_j}}^2}{f_{H_j}^2}
  \right\}\left(\sum_{j} w_j f_{H_j} \right) \\ &= \sum_{j} w_j
  \frac{\norm{\nabla f_{H_j}}^2}{f_{H_j}}. 
  \end{align*}
This, along with the trivial inequality (here $a \vee b$ stands for
$\max(a, b)$) 
\begin{equation*}
  \left(1 - \frac{f_G}{f_G \vee \rho} \right)^2 \leq \left(1 -
    \frac{f_{H_j}}{f_{H_j} \vee (\rho/w_j)} 
  \right)^2 \qt{for every $1 \leq j \leq m$}
\end{equation*}
yields \eqref{cvx.eq}.  
\end{proof}

\begin{lemma}\label{obg}
  Suppose $X_1, \dots, X_n$ are independent observations with $X_i
  \sim N(\theta_i, I_d)$ for some $\theta_1, \dots, \theta_n \in
  \R^d$. Let the Oracle Bayes estimators $\hat{\theta}_1^*, \dots,
  \hat{\theta}_n^*$  be defined as in \eqref{ob} where $\bar{G}_n$ is
  the empirical measure of $\theta_1, \dots, \theta_n$. Suppose that
  $\bar{G}_n$ is supported on a set $\{a_1, \dots, a_k\}$ of
  cardinality $k$ with $\bar{G}_n\{a_i\} = p_i$ for $i = 1, \dots, k$
  with $p_i \geq 0$ and $\sum_{i=1}^k p_i = 1$. Then 
  \begin{equation}\label{obg.eq} 
\begin{split}
 &   \E \left[\frac{1}{n} \sum_{i=1}^n \norm{\hat{\theta}^*_i -
        \theta_i}^2 \right] \\ &\leq \frac{k-1}{2 \sqrt{2 \pi}} \sum_{j, l
    : j \neq l} \left(p_j + p_l \right) \norm{a_j - a_l}  \exp
  \left(-\frac{1}{8} \norm{a_j - a_l}^2 \right).  
\end{split}
  \end{equation}
\end{lemma}

\begin{proof}[Proof of Lemma \ref{obg}]
  Note first that $\hat{\theta}_i^*$ has the following expression
  \begin{equation*}
    \hat{\theta}_i^* = \frac{\sum_{j=1}^k a_j p_j \phi_d(X_i - a_j)
    }{\sum_{j=1}^k p_j \phi_d(X_i - a_j)} \qt{for $i = 1, \dots, n$}. 
  \end{equation*}
The above expression and the fact that $X_i - \theta_i \sim N(0, I_d)$
lets us write 
\begin{align*}
     R &:= \E \left[\frac{1}{n} \sum_{i=1}^n \norm{\hat{\theta}^*_i -
        \theta_i}^2 \right] \\ &= \sum_{l=1}^k p_l \E
    \norm{\frac{\sum_{j=1}^k a_j p_j \phi_d(a_l + Z - a_j) 
    }{\sum_{j=1}^k p_j \phi_d(a_l + Z - a_j)} - a_l}^2 \\
&= \sum_{l=1}^k p_l \E \norm{\frac{\sum_{j=1}^k \left(a_j - a_{l} \right) p_j \phi_d(a_l + Z - a_j) 
    }{\sum_{j=1}^k p_j \phi_d(a_l + Z - a_j)}}^2 \\
&= \sum_{l=1}^k p_l \E \norm{\sum_{j: j \neq l} (a_j - a_l) w_{j l}(Z)}^2
\end{align*}
where $Z \sim N(0, I_d)$ and 
\begin{equation*}
  w_{j l}(Z)
  := \frac{p_j \phi_d(a_l + Z - a_j)}{\sum_{u=1}^k p_u \phi_d(a_l + Z
    - a_u)} \qt{for $1 \leq j, l \leq k$}. 
\end{equation*}
The inequality $\norm{\sum_{i=1}^m \alpha_i}^2 \leq m
\sum_{i=1}^m \norm{\alpha_i}^2$ for vectors $\alpha_1, \dots, \alpha_m
\in \R^d$ now lets us write
\begin{equation}\label{talac}
  R \le (k-1) \sum_{l = 1}^k p_l \sum_{j : j \neq l} \norm{a_j -
    a_l}^2 \E w_{jl}^2(Z). 
\end{equation}
We now bound $\E w_{jl}^2(Z)$ in the following way. Let 
\begin{equation*}
  U := \left\{z \in \R^d : \norm{a_j - a_l}^2 \geq 2 \left<Z, a_j -
      a_l \right> \right\}. 
\end{equation*}
When $Z \notin U$, we shall use the trivial upper bound $w_{jl}^2(Z)
\leq 1$. When $Z \in U$, we shall use the bound
\begin{equation*}
  w_{jl}^2(Z) \leq w_{jl}(Z) \leq \frac{p_j \phi_d(a_l + Z - a_j)}{p_l
  \phi_d(a_l + Z - a_l)} = \frac{p_j \phi_d(a_l + Z - a_j)}{p_l
  \phi_d(Z)}. 
\end{equation*}
This gives
\begin{align*}
  \E w_{jl}^2(Z) &\leq \P \left\{Z \notin U \right\} \\ &+ \int \frac{p_j \phi_d(a_l + z - a_j)}{p_l
  \phi_d(z)} I \{\norm{a_j - a_l}^2 \geq 2 \left<z, a_j - a_l
                   \right>\} \phi_d(z) dz 
\end{align*}
The change of variable $x = a_l + z - a_j$ in the integral above
allows us to write
\begin{align*}
  \E w_{jl}^2(Z)  &\leq \P \left\{\left<Z, a_j - a_l \right> > \frac{1}{2}
  \norm{a_j - a_l}^2\right\} \\ &+ \frac{p_j}{p_l} \P \left\{\left<Z, a_j - a_l \right> \leq -\frac{1}{2}
  \norm{a_j - a_l}^2 \right\} \\
&\leq \left(1 + \frac{p_j}{p_l} \right) \left(1 -  \Phi \left(\frac{1}{2}
  \norm{a_j - a_l} \right) \right) 
\end{align*}
where $\Phi$ is the standard univariate Gaussian cumulative
distribution function. The bound $1 - \Phi(t) \leq \phi(t)/t$ for $t >
0$ now gives
\begin{align*}
  \E w_{jl}^2(Z) \leq \frac{1}{2 \sqrt{2 \pi}} \left(1 +
  \frac{p_j}{p_l} \right) \frac{1}{\norm{a_j - a_l}} \exp
  \left(-\frac{1}{8} \norm{a_j - a_l}^2 \right). 
\end{align*}
This bound, when combined with \eqref{talac}, yields \eqref{obg.eq}
and hence completes the proof of Lemma \ref{obg}.  
\end{proof}

\section{Simulations for Clustering Settings}\label{desim}
Here, we shall numerically illustrate the denoising performance of
$\hat{\theta}_1, \dots, \hat{\theta}_n$ (defined as in \eqref{es.int}) when the true
vectors $\theta_1, \dots, \theta_n$ have a clustering structure. We
take $d = 2$ and consider the following four simulation settings:  
\begin{enumerate}
\item Setting One: We generate $\theta_1, \dots, \theta_n$ as i.i.d
  from the distribution which puts equal probability (0.5) at $(0, 0)$ and
  $(2, 2)$. 
\item Setting Two: We generate $\theta_1, \dots, \theta_n$ as i.i.d 
  from the distribution which puts $1/4$ probability at $(0, 0)$ and
  $3/4$ probability at $(2,2)$. 
\item Setting Three: We generate $\theta_1, \dots, \theta_n$ as i.i.d
  from the distribution which puts $1/4$ probability each at $(0, 0)$
  and $(0, 2)$ and $1/2$ probability at $(2, -2)$. 
\item Setting Four: We generate a random probability vector
  $(\alpha_1, \alpha_2,. \alpha_3, \alpha_4)$ from the Dirichlet
  distribution with parameters $(1, 1, 1, 1)$ and then generate
  $\theta_1, \dots, \theta_n$ as i.i.d from the probability
  distribution with puts probabilities $\alpha_1, \alpha_2, \alpha_3$
  and $\alpha_4$ at the four points $(0, 0)$, $(0, 3)$, $(3, 0)$ and
  $(3, 3)$. 
\end{enumerate}
The observed data $X_1, \dots, X_n$ are, as usual, generated
independently as $X_i \sim N(\theta_i, I_d)$. We allow the sample size
$n$ to take the values 300, 600, 900, 1200, 1500, 1800,
2100. For each $n$, we perform 1000 replicates to get accurate
estimates of mean squared error. For each dataset, we 
compute the Empirical Bayes estimates \eqref{es.int}. For comparison,
we also computed $k$-means estimates based on the true (Oracle) value
of $k$ and those based on the gap statistic (from
\citet{tibshirani2001estimating}). These estimates will be referred
to, in the sequel, as \texttt{kmeans-Oracle} and \texttt{kmeans-gap}
respectively. For $k$-means, we used the standard Lloyd's algorithm 
based on 10 random starts and the best solution is considered of the
random starts.  Note that because of non-convexity, no implementation
of $k$-means can provably reach global optimum.   

For each of the these three estimates, we plotted the mean squared
errors in Figure \ref{fig:denoising_sim} (see the first plot in each
pair of plots for the different settings).  
\begin{figure} 
    \centering
    \begin{subfigure}[t]{0.49\textwidth}
        \centering
        \includegraphics[width = 0.99\textwidth]{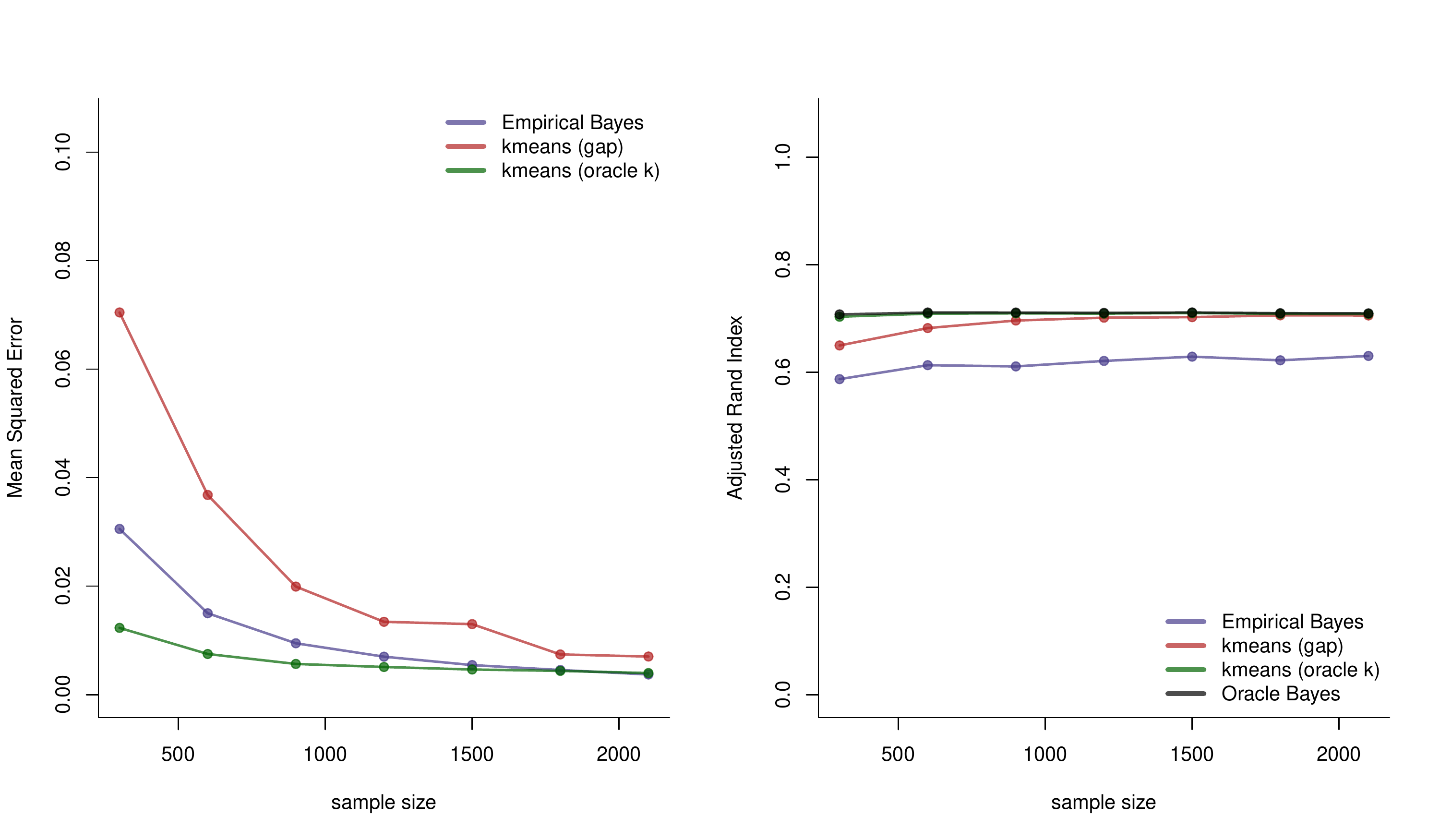}
        \caption{Setting 1. Two equally sized clusters centered at $(0,0)$ and $(2,2)$. For clarification, in the ARI plot the red and green curves coincide.}
    \end{subfigure}%
    ~ 
    \centering
    \begin{subfigure}[t]{0.49\textwidth}
        \centering
        \includegraphics[width = 0.99\textwidth]{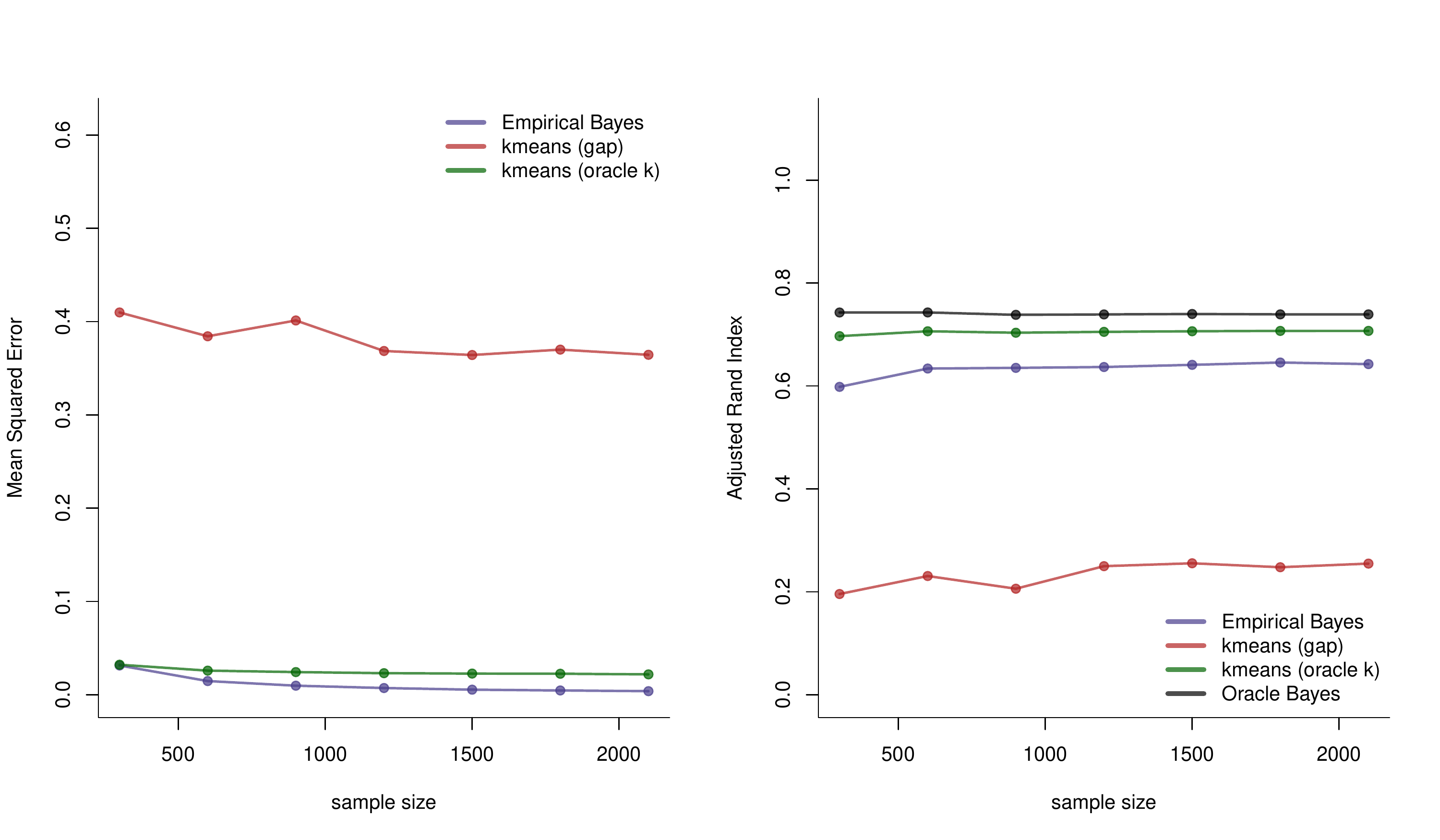}
        \caption{Setting 2. Two clusters centered at $(0,0)$ and $(2,2)$ with cluster proportions $1/4$ and $3/4$.For clarification, in the ARI plot the red and green curves coincide.}
    \end{subfigure}%
    ~ 
    \vspace{5mm}
    \\
    ~
    \centering
    \begin{subfigure}[t]{0.49\textwidth}
        \centering
        \includegraphics[width = 0.99\textwidth]{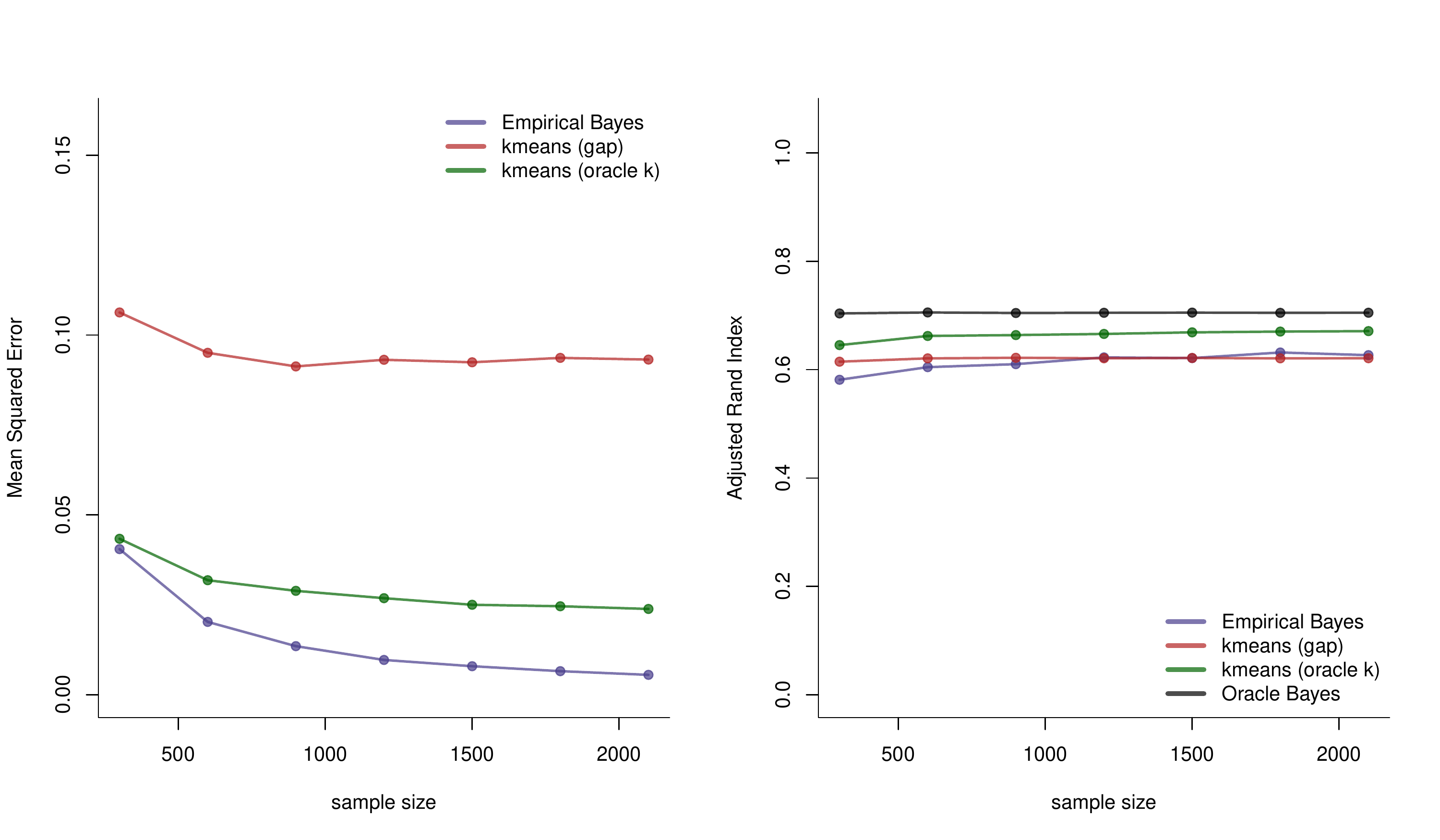}
        \caption{Setting 3. Three clusters centered at $(0,0), (0,2), (2,-2)$ with cluster proportions $1/4, 1/4, 1/2$ respectively.}
    \end{subfigure}%
    ~ 
    \centering
    \begin{subfigure}[t]{0.49\textwidth}
        \centering
        \includegraphics[width = 0.99\textwidth]{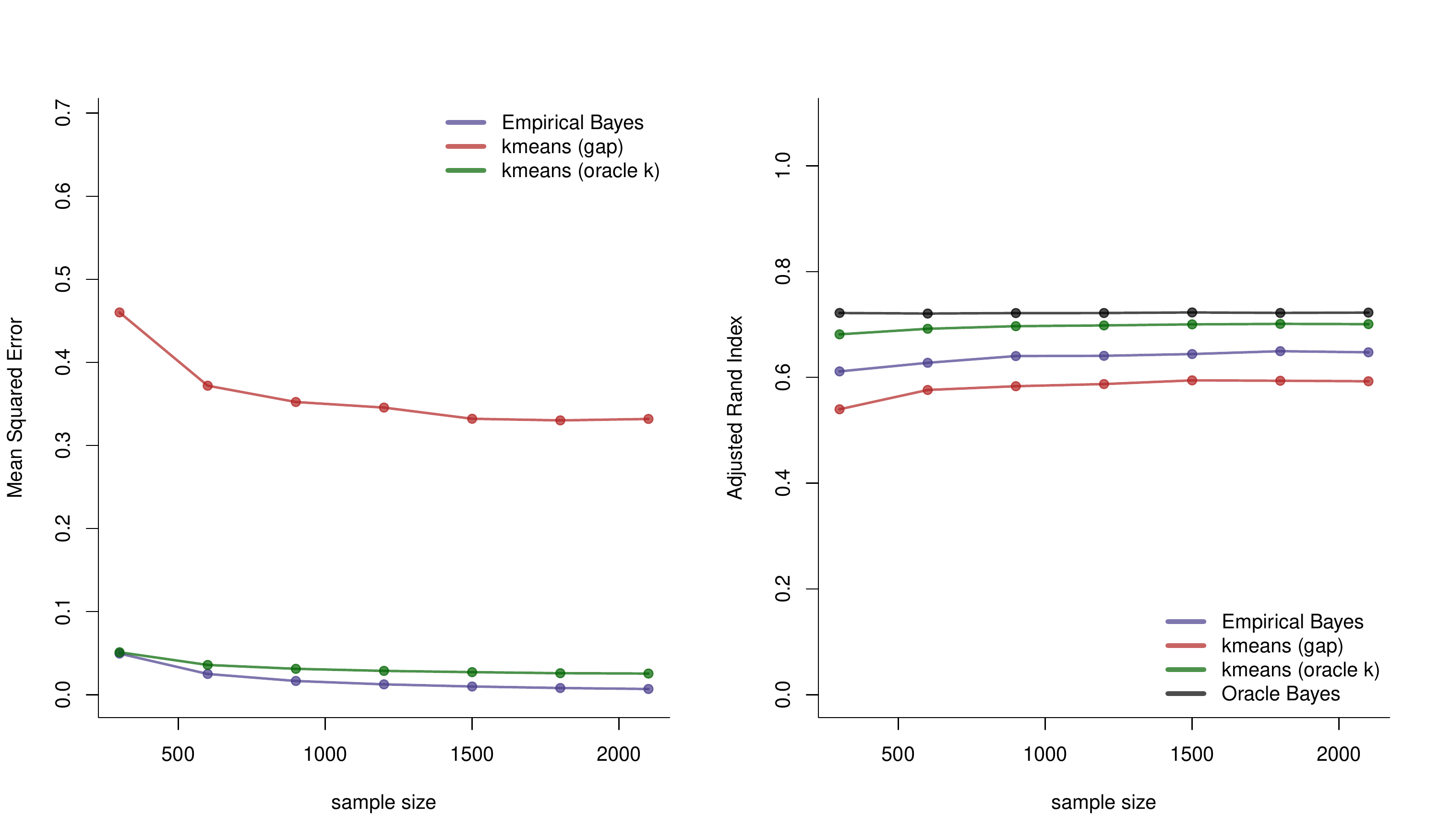}
        \caption{Setting 4. Four cluster centers centered at $(0,0), (0,3), (3,0), (3,3)$ with cluster proportions drawn from Dirichlet distribution with parameters $(1, 1, 1, 1)$}
    \end{subfigure}%
    ~ 
    \caption{Empirical performance of methods in the denoising problem
      in four different clustering settings. A method with lower MSE
      is preferred over one with higher MSE. In contrast, a method
      with higher ARI is preferred over one with lower ARI. The lines
      show mean of the metric in question over $1000$ replicates.} 
    \label{fig:denoising_sim}
\end{figure}
From these MSE plots, it is clear that the Empirical Bayes estimates
based on the NPMLE are more accurate than \texttt{kmeans-gap}. In
fact, with the exception of the  first setting, the Empirical Bayes
estimates are even more accurate than \texttt{kmeans-Oracle}. This is
probably because of the non-convexity of $k$-means.

In addition to estimating the means $\theta_1, \dots, \theta_n$ by
$\hat{\theta}_1, \dots, \hat{\theta}_n$, the Empirical Bayes method
also produces clusterings which can be compared to the true clusters
(as well as to the clusterings produced by \texttt{kmeans-Oracle} and
\texttt{kmeans-gap}) by the Adjusted Rand Index (ARI)
\cite{rand1971objective}. Before showing these 
results, let us first explain how we are producing clusters based on
the Empirical Bayes method. Our exemplar-based algorithm for computing
the NPMLE 
based on $X_1, \dots, X_n$  produces an estimate $\hat{f}_n$ 
which can be written as  
\begin{equation*}
  \hat{f}_n(x) := \sum_{j=1}^m \hat{p}_j \phi_d(x - \hat{a}_j) \qt{for
  $x \in \R^d$}
\end{equation*}
for some $m \leq n$. From this estimate, we cluster the observation
$X_i$ to the $r^{th}$ cluster (for $r = 1, \dots, m$) provided 
\begin{equation*}
  r = \argmax_{1 \leq j \leq m} \frac{\hat{p}_j \phi_d(X_i - \hat
    a_j)}{\sum_{l = 1}^m \hat{p}_l \phi_d(X_i - \hat{a}_l)}. 
\end{equation*}
Note that the number of clusters $(=m)$ produced by this method will
be different (and usually larger) than the true number of clusters but
the ARI is applicable for comparing clusterings with different numbers of
clusters. For comparison purposes, we also cluster using the true density
where $\hat{p}_j$  and $\hat{a}_j$ in the method desribed above are
replaced by $p_j$ and $a_j$ respectively with $f^*(x) := \sum_j p_j
\phi_d(x - a_j)$. We shall refer to this as the clustering based on
the Oracle Bayes estimate.  

For each of the these four clusterings (Empirical Bayes based
clustering, \texttt{kmeans-gap}, \texttt{kmeans-Oracle} and Oracle
Bayes based clustering), we plotted the ARIs as a function of sample
size in Figure \ref{fig:denoising_sim} (see the second plot in each 
pair of plots for the different settings). Higher ARIs are preferred
to lower values. Here the Oracle Bayes estimate is the best; the
\texttt{kmeans-Oracle} method is superior to the Empirical Bayes
estimate as well as \texttt{kmeans-gap}. The comparison between the
Empirical Bayes and the \texttt{kmeans-gap} estimates in terms of ARI 
can be summarized as follows. In the first  setting, the performance
of \texttt{kmeans-gap} is very good and is indistinguishable from
\texttt{kmeans-Oracle}. In more complicated settings with more than 
two clusters and/or with imbalanced cluster proportions, a
distinction between the two methods becomes apparent. In the second
and fourth settings, the Empirical Bayes method outperforms
\texttt{kmeans-gap}.  In the third setting, the performances of the
two methods start to coincide for larger sample sizes.

\bibliographystyle{chicago}
\bibliography{../../AG}

\end{document}